\documentclass[12pt]{article}
\usepackage[a4paper,text={170mm,250mm},centering]{geometry}
\usepackage[affil-it]{authblk}
\usepackage{amsmath,amsthm,amssymb}
\usepackage[inline]{asymptote}
\usepackage{enumitem}
\usepackage{indentfirst}
\usepackage[title,titletoc]{appendix}
\usepackage{textgreek}
\usepackage{hyperref}
\usepackage[symbols,nogroupskip,sort=none]{glossaries-extra}

\title{Connecting Orbits in Cooperative McKean-Vlasov SDEs}

\author[1,3,4]{Chunlin Liu\thanks{chunlinliu@mail.ustc.edu.cn}}
\author[2,3]{Baoyou Qu\thanks{qubaoyou@sdu.edu.cn}}
\author[3]{Jinxiang Yao\thanks{jxyao@mail.ustc.edu.cn, yao.jinxiang@durham.ac.uk}}
\author[3]{Yanpeng Zhi\thanks{yanpeng.zhi@durham.ac.uk}}

\affil[1]{{\small School of Mathematical Sciences, Dalian University of Technology, Dalian, 116024, P.R. China}}
\affil[2]{{\small Research Center for Mathematics and Interdisciplinary Sciences, Shandong University, Qingdao, 266237, P.R. China}}
\affil[3]{{\small Department of Mathematical Sciences, Durham University, DH1 3LE, U.K.}}
\affil[4]{{\small School of Mathematical Sciences, University of Science and Technology of China, Hefei, Anhui, 230026, P.R. China}}
\date{}

\newtheorem{theorem}{Theorem}[section]

\newtheorem{lemma}[theorem]{Lemma}
\newtheorem{proposition}[theorem]{Proposition}
\theoremstyle{definition}
\newtheorem{remark}{Remark}[section]
\newtheorem{definition}{Definition}[section]
\newtheorem{assumption}{Assumption}

\allowdisplaybreaks[4]

\def\one{\mathbf 1}
\def\bE{\mathbf E}

\def\N{\mathbb N}

\def\R{\mathbb R}

\def\T{\mathbb T}

\def\Z{\mathbb Z}

\def\BB{\mathcal B}

\def\LL{\mathcal L}
\def\MM{\mathcal M}

\def\PP{\mathcal P}

\def\RR{\mathcal R}

\def\TT{\mathcal T}

\def\WW{\mathcal W}

\def\xa{\alpha}
\def\xb{\beta}
\def\xg{\gamma}
\def\xG{\Gamma}
\def\xd{\delta}

\def\xe{\varepsilon}

\def\xl{\lambda}
\def\xk{\kappa}

\def\xs{\sigma}

\def\as{\mathrm{a.s.}}

\def\conv{\mathrm{Conv}}
\def\d{\mathop{}\!\mathrm{d}}

\def\e{\mathrm e}

\def\Trace{\mathrm{Trace}}

\def\tran#1{#1^{\top}}

\def\str#1{#1^{\ast}}

\def\abs#1{\left\lvert #1 \right\rvert}
\def\norm#1{\left\lVert #1 \right\rVert}

\def\inp#1#2{\left\langle #1,#2 \right\rangle}

\def\8{\infty}
\def\no#1{\mathrel{\phantom{#1}}}

\def\mutd{\tilde{\mu}}

\def\goto{\xrightarrow}

\hypersetup{hidelinks,colorlinks=true,linkcolor=red,citecolor=blue}

\renewcommand{\glossarysection}[2][]{}
\newglossarystyle{mylong}{%
     {\begin{longtable}[l]{@{}p{\dimexpr 0.1\textwidth -\tabcolsep}p{\dimexpr 0.9\textwidth -\tabcolsep}@{}}}%
     {\end{longtable}}%
  \renewcommand*{\glsgroupheading}[1]{}%
  \ifglsnogroupskip
  \else
  \fi
}
\glsdisablehyper
\glsxtrnewsymbol
[description={The collection of Borel measurable sets in $\R^d$}]
{BorelSet}
{\ensuremath{\BB(\R^d)}}

\glsxtrnewsymbol
[description={The set of probability measures on $\R^d$}]
{ProbSp}
{\ensuremath{\PP(\R^d)}}

\glsxtrnewsymbol
[description={The set of probability measures on $\R^d$ with finite $p$-th moments}]
{ProbSpp}
{\ensuremath{\PP_p(\R^d)}}

\glsxtrnewsymbol
[description={The set of probability measures on $\R^d$ with finite moments of all orders $p\geq1$}]
{ProbSpInfty}
{\ensuremath{\PP_{\8}(\R^d)}}

\glsxtrnewsymbol
[description={The stochastic order between probability measures}]
{StOrd}
{\ensuremath{\leq_{\mathrm{st}}}}

\glsxtrnewsymbol
[description={The order interval consisting of probability measures $\mu\in\PP(\R^d)$ with $\mu_1\leq_{\mathrm{st}}\mu\leq_{\mathrm{st}}\mu_2$}]
{order-internal-ProbS}
{\ensuremath{[\mu_1,\mu_2]_{\PP}}}

\glsxtrnewsymbol
[description={The order interval consisting of probability measures $\mu\in\PP_p(\R^d)$ with $\mu_1\leq_{\mathrm{st}}\mu\leq_{\mathrm{st}}\mu_2$}]
{order-interval}
{\ensuremath{[\mu_1,\mu_2]_{\PP_p}}}

\glsxtrnewsymbol
[description={The $p$-Wasserstein distance}]
{Wp}
{\ensuremath{\WW_p}}

\glsxtrnewsymbol
[description={The $p$-th root of the $p$-th absolute moment of a probability measure $\mu$}]
{Normp}
{\ensuremath{\norm{\mu}_p}}

\glsxtrnewsymbol
[description={The $2$-norm of a vector $b=(b_i)_{i=1}^d\in\R^d$, $\abs{b}:=(\sum_ib_i^2)^{1/2}$}]
{vector-norm}
{\ensuremath{\abs{b}}}

\glsxtrnewsymbol
[description={The transpose of a matrix $\xs\in\R^{d\times l}$}]
{matrix-transpose}
{\ensuremath{\tran{\xs}}}

\glsxtrnewsymbol
[description={The $2$-norm of a matrix $\xs\in\R^{d\times l}$, $\abs{\xs}:=(\Trace(\xs\tran{\xs}))^{1/2}$}]
{matrix-norm}
{\ensuremath{\abs{\xs}}}

\glsxtrnewsymbol
[description={The open ball consisting of probability measures $\nu\in\PP_2(\R)$ with $\WW_2(\mu,\nu)<r$}]
{open-ball}
{\ensuremath{B_{\PP_2}(\mu,r)}}

\glsxtrnewsymbol
[description={The closed ball consisting of probability measures $\nu\in\PP_2(\R)$ with $\WW_2(\mu,\nu)\leq r$}]
{closed-ball}
{\ensuremath{\overline{B_{\PP_2}(\mu,r)}}}

\begin{document}

\maketitle

\begin{abstract}
In this work we extend the framework of monotone dynamical systems to a broad and important class of stochastic equations, namely cooperative McKean–Vlasov SDEs 
with multiplicative noise. 
Under a locally dissipative assumption, our main theorem establishes the existence of multiple order-related invariant measures in the the Wasserstein space  together with monotone connecting orbits (heteroclinic orbits) between them, with respect to the stochastic order. 
The presence of such connecting orbits also reveals the unstable nature of those invariant measures appearing as their backward limits, a dynamical feature that has remained largely unexplored in stochastic equations.
The framework applies to a wide range of classical models, including granular media equations in double-well and multi-well confining potentials with quadratic interaction, perturbed double-well landscapes, and interacting multi-species population models. 
Our method is based on building a monotone dynamical system that preserves the stochastic order, achieved through a cone compatible with this order and an extension of the classical Dancer–Hess connecting orbit theorem.

\medskip

\noindent
{\bf MSC2020 subject classifications:} Primary 37C65, 60E15; secondary 60H10, 60B10

\noindent
{\bf Keywords:} Monotone dynamical systems, connecting orbits, unstable dynamics,  cooperative McKean-Vlasov SDEs,  stochastic order

\end{abstract}

{
\hypersetup{linkcolor=black}
\tableofcontents
}

\section{Introduction}

Monotone dynamical systems—also known as order-preserving systems—are those that satisfy a comparison principle with respect to a closed partial order on the underlying state space.
The theory of monotone dynamical systems was established through the pioneering contributions of Hirsch \cite{Hirsch84,Hirsch88} and Matano \cite{H79,H86}. 
In the decades that followed, it has developed into a widely applicable framework for models ranging from ordinary, functional, and partial differential equations to discrete-time systems.
We refer to \cite{HS05,ShenYi98,Smith95,Smith17,Z17} and references therein for details.
Stochastic extensions of monotone dynamical systems have also been studied in the pathwise (trajectory-based) framework of random dynamical systems, with applications to the long-time behavior of random and stochastic differential equations; see, for instance, \cite{Chue02,ChueshovScheutzow2004,Gess17} and references therein.

In this work we extend the monotone dynamical systems framework to encompass a broad and important class of stochastic equations, namely McKean–Vlasov SDEs with multiplicative noise.
We consider the equation on $\R^d$,
\begin{equation}\label{eq:mvsystem}
\d X_t=b(X_t,\LL(X_t))\d t+\xs(X_t)\d W_t,
\end{equation}
where $\{W_t\}_{t\geq0}$ is an $l$-dimensional standard Brownian motion, $\LL(X_t)$ is the law of the random variable $X_t$.
Such distribution-dependent stochastic differential equations trace back to McKean’s seminal work \cite{McKean1966} and now arise in a wide range of applications, including physics, biology, network dynamics, and control theory; see Carmona–Delarue \cite{Carmona-Delarue2018i,Carmona-Delarue2018ii} for an overview.

Due to their distributional dependence, McKean–Vlasov SDEs fail to generate a random dynamical system in the pathwise sense. 
In the analysis of the distributional evolution, and in particular of invariant measures, a substantial literature has been developed.
An invariant measure here refers to a steady state of the law evolution semigroup $\{P_t^*\}_{t\ge 0}$, where $P_t^*\mu$ denotes the distribution of the solution at time $t$ starting from the initial law $\mu$. 
For work on the existence of invariant measures, as well as the possible uniqueness and associated global convergence, we refer to  
\cite{Bao2022,Carrillo-McCann-Villani2003,Carrillo-McCann-Villani2006,Wang18}  and references therein.

A key distinction of McKean–Vlasov SDEs with non-degenerate noise, compared with usual SDEs, lies in the possibility of non-unique invariant measures.
By a usual SDE, we mean its coefficients do not depend on laws of solutions,
\begin{equation}\label{eq:usual-sde}
\d X_t=b(X_t)\d t+\xs(X_t)\d W_t.
\end{equation}
If the diffusion term $\xs$ is non-degenerate, no matter how small it is, one often expects a usual SDE has only one invariant measure.
When allowing $b,\xs$ to depend on laws of solutions, multiple invariant measures survive if the noise is not too strong.
Dawson \cite{Dawson1983} and Tugaut \cite{Tugaut2014} present phase transitions on the number of invariant measures for granular media equations in double-well landscapes when $\xs$ varies, and Alecio \cite{Alecio2023} steps further to the multi-well landscapes.
Carrillo-Gvalani-Pavliotis-Schlichting \cite{Carrillo-Gvalani-Pavliotis-Schlichting2020} and Delgadino-Gvalani-Pavliotis \cite{Delgadino-Gvalani-Pavliotis2021} also prove there are phase transitions on the number of invariant measures for weakly interacting diffusion processes on tori.

When multiple invariant measures coexist, attention naturally turns to understanding their stability and the asymptotic behavior of solutions near each of them.
For granular media equations with a double-well confining potential, where noise is sufficiently small, Tugaut \cite{Tugaut2013} demonstrates that the solution with initial conditions having finite entropy converges to one of these three invariant measures. 
The study of local convergence has been advanced by various approaches, including
Zhang's method in linearizing the nonlinear Markov semigroup \cite{Zhang2025}, 
Tugaut’s use of the WJ-inequality \cite{Tugaut2023}, Monmarché-Reygner’s application of a local log-Sobolev argument \cite{Monmarche-Reygner2024}, Cormier’s study through Lions derivatives \cite{Cormier2024}, and Tamura's early work utilizing the free energy function \cite{Tamura1984,Tamura1987}.

From a broader dynamical perspective, the coexistence of multiple steady states naturally brings into focus the notion of a connecting (heteroclinic) orbit, which plays a classically central role in the analysis of bifurcation phenomena and chaotic dynamics (see, e.g., Chow and Hale \cite{Chow82}).
This term refers to 
an entire orbit which joins two  equilibria in the phase space. 
It has been found in significant mathematical models in various areas, such as in biological, chemical, fluid mechanics models (see e.g., Balmforth \cite{Balmforth95}, May-Leonard \cite{MayLeonard75}). 

Under cooperative and locally dissipative assumptions, our main result Theorem \ref{thm:existence-unstable} shows the existence of order-related invariant measures, together with monotone connecting orbits between them (with respect to the stochastic order), for a broad class of McKean–Vlasov SDEs.
It plays a crucial role in revealing the dynamical transitions (how systems evolve) between different steady states in such equations, along with the stochastic order.
At the same time, the existence of connecting orbits provides a more precise characterization of the unstable nature of invariant measures appearing as backward limits.
Moreover, the presence of connecting orbits also highlights the complexity of the dynamics of McKean-Vlasov SDEs and provides an avenue for understanding their global dynamics.

The abstract result applies directly to a wide range of classical models with well-established applications.
As an illustration, we consider one-dimensional granular media equations of the form
\[
\d X_t
= -\nabla V(X_t)\,\d t
  - (\nabla W * \mathcal L(X_t))(X_t)\,\d t
  + \sigma(X_t)\,\d W_t,
\]
where $V:\R^d\to\R$ is a confining potential and $W:\R^d\to\R$ is an interaction potential.
A typical and widely used choice is the attractive quadratic interaction
\[
W(x)=\frac{\beta}{2}|x|^2,\qquad \beta>0, 
\]
which arise in diverse areas such as muscle contraction
\cite{Kometani75}, chemical kinetics \cite{Horsthemke77},
statistical physics \cite{Haken77}, and large economic systems \cite{Aoki1980}.
In this setting, we consider several representative choices of the confining potential $V$, 
including double-well (Theorem \ref{thm:double-well}) and multi-well landscapes (Theorem \ref{thm:multi-well}), as well as perturbed double-well case (Theorem \ref{thm:double-well-perturbation}).
In the higher-dimensional case, we consider a two-dimensional  equation (Theorem \ref{thm:high-dimension}) that 
describes the motions of an interacting two-species population under the influence
of external forces, intra-species and inter-species interacting forces, and stochastic noises (see, for instance \cite{DuongHongTugaut20,Duong-Pavliotis-Tugaut2025} and references therein).
\begin{equation*}
\begin{cases}
    \d X_t=\left[X_t-X^3_t-\tau\alpha\left(X_t-\bE X_t\right)-(1-\tau)\beta\left(X_t-\bE Y_t\right)\right]\d t+\xs\d W_1(t),\\
\,\d Y_t\,=\left[Y_t-Y^3_t-\tau\alpha\left(Y_t-\bE X_t\right)-(1-\tau)\beta\left(Y_t-\bE Y_t\right)\right]\d t+\xs\d W_2(t).
\end{cases}
\end{equation*}
For all these examples, we provide explicit parameter regimes under which the assumptions of our main theorems
are satisfied, together with the corresponding phase diagrams.

Our approach begins with the fact that the law evolution semigroup $\{P_t^*\}_{t\ge 0}$ generated by the
McKean–Vlasov SDE is order-preserving on the $2$-Wasserstein space with respect to the stochastic order (Proposition \ref{prop:continuity-Pt}), provided the cooperative condition
\[
b_i(x,\mu)\leq b_i(y,\nu),\text{ for $x_i=y_i$, $x_j\leq y_j$, $j\neq i$, and $\mu\leq_{\mathrm{st}}\nu$}.
\]
Here the stochastic order ``$\leq_{\mathrm{st}}$'' defined by
\[
\mu\leq_{\mathrm{st}}\nu\text{ if and only if }\int_{\R^d}f\d\mu\leq\int_{\R^d}f\d\nu\ \ \text{for all bounded increasing functions $f\colon\R^d\to\R$},
\]
where “increasing’’ refers to the coordinate-wise order on \(\mathbb R^d\).
The notion of order-preserving  has already appeared in the study of uniqueness of invariant measures, ergodicity, and synchronization phenomena for certain classes of Markov processes; see, for instance, \cite{Butkovsky2020,RobertRichard1996,RobertTweedie2000}. 
Here, we are instead concerned with the dynamics of McKean–Vlasov SDEs in the presence of multiple invariant measures, a setting in which these frameworks do not apply. 
Our method is to bring the tools of monotone dynamical systems into the study of the dynamics of McKean–Vlasov SDEs.
However, the classical theory of monotone dynamical systems—while highly effective for analyzing bistable and multistable behavior—rests on structural assumptions that are not available in our setting.
Among its central ingredients is the Dancer-Hess connecting orbit theorem, which stands as a fundamental result of the theory: much of the core theory for monotone systems ultimately rests on it.
The theorem asserts that within an order interval enclosed by two order-related fixed points and containing no additional fixed points, there necessarily exists a monotone (either increasing or decreasing) connecting orbit between them.
Yet its formulation requires at least strictly monotone systems on Banach spaces endowed with a partial order induced by a closed convex cone—assumptions that fail in the context of the law evolution of McKean–Vlasov SDEs.

To address these structural obstacles, we first extend the state space $\PP_2(\R^d)$ to the vector space $\MM_1(\R^d)$ of finite signed Borel measures with finite first moments, equipped with a Kantorovich-type norm (Lemma \ref{lem:m1-norm-space}).
At the same time, a cone in $\MM_1(\R^d)$ compatible with the stochastic order is introduced, \[ C:=\bigg\{\mu\in\MM_1(\R^d): \int_{\R^d}f\d\mu\geq0\text{ for all non-negative 1-Lipschitz increasing functions $f$}\bigg\}, \] so that the induced partial order extends the stochastic order on $\PP_2(\R^d)$ (Lemma \ref{lem:m1-cone}).
Correspondingly, the classical Dancer–Hess connecting orbit theorem is extended to the setting of general monotone semiflows on convex compact subsets enclosed by two order-related equilibria in Hausdorff locally convex topological vector spaces (Theorem \ref{T:semiflow-tri}).
To apply this extended framework to cooperative McKean–Vlasov SDEs, a further key step is to identify compact order intervals enclosed by order-related invariant measures. 
This is achieved by analyzing the continuity, compactness, and monotonicity properties of the measure-iterating map 
$\Psi$  (Proposition \ref{P:Psi-Cts-Cpt-Monotone}).

This paper is organised as follows.
In Section \ref{Section:Order-Cone}, we introduce the stochastic order and consider its interplay with Wasserstein metrics.
In Section \ref{Section:Generation-MDS}, we prove that, under a cooperative condition, monotone semiflows are generated on Wasserstein spaces.
In Section \ref{Section:Results-MDS},  we extend the classical connecting orbit theorem.
Section \ref{Section:Order-Related-Shrinking} is concerned with the existence of order-related invariant measures and their shrinking neighbourhoods under the  locally dissipative condition.
In Section \ref{Section:Proof-Main-Results}, we prove our main results, and Appendix \ref{Appendix} provides a detailed proof of connecting orbit theorem for monotone mappings on an ordered topological vector space.

\subsection{Notations}

\printunsrtglossary[type=symbols,style=mylong]

\subsection{Main results}

In this subsection, we state our assumptions, some necessary definitions and main results.

\begin{assumption}\label{asp:lipschitz}
(i) The function $b\colon\R^d\times\PP_2(\R^d)\to\R^d$ is continuous and one-sided Lipschitz, i.e., there exists $K>0$ such that, for all $x,y\in\R^d$, $\mu,\nu\in\PP_2(\R^d)$,
\[
\inp{x-y}{b(x,\mu)-b(y,\nu)}\leq K\abs{x-y}^2+K\abs{x-y}\WW_2(\mu,\nu).
\]
(ii) The function $\xs\colon\R^d\to\R^{d\times l}$ is Lipschitz continuous, namely, there exists $K>0$ such that, for all $x,y\in\R^d$,
\[
\abs{\xs(x)-\xs(y)}^2\leq K\abs{x-y}^2.
\]
\end{assumption}

\begin{assumption}\label{asp:cooperation}
(i) The function $b\colon\R^d\times\PP_2(\R^d)\to\R^d$ satisfies,
\[
b_i(x,\mu)\leq b_i(y,\nu),\text{ for $x_i=y_i$, $x_j\leq y_j$, $j\neq i$, and $\mu\leq_{\mathrm{st}}\nu$}.
\]
(ii) The function $\xs\colon\R^d\to\R^{d\times l}$ satisfies, for all $i=1,2,\ldots,d$, and for all $x,y\in\R^d$,
\[
\sum_{k=1}^l\abs{\xs_{ik}(x)-\xs_{ik}(y)}^2\leq K\abs{x_i-y_i}^2.
\]
\end{assumption}

Assumption \ref{asp:cooperation} is the cooperative condition, and it follows from (ii) that the dependence of
$\xs_{ik}(x)$ on $x$ is only via $x_i$.

\begin{assumption}\label{asp:dissipative-growth-nondegeneracy}
(i) The function $b\colon\R^d\times\PP_2(\R^d)\to\R^d$ is weakly dissipative and has polynomial growth, that is, there exist $\xa>\xb>0$, $\xg>0$, such that, for all $x\in\R^d$, $\mu\in\PP_2(\R^d)$,
\[
\inp{x}{b(x,\mu)}\leq-\xa\abs{x}^2+\xb\norm{\mu}_2^2+\xg,
\]
and there exist $K>0$, $\xk>0$ such that, for all $x\in\R^d$, $\mu\in\PP_2(\R^d)$,
\[
\abs{b(x,\mu)}\leq K(1+|x|^{\xk}+\norm{\mu}_2^{\xk}).
\]
(ii) The function $\xs\colon\R^d\to\R^{d\times l}$ is non-degenerate, i.e., there exist $0<\underline{\xs}<\overline{\xs}$, such that, for all $x\in\R^d$,
\[
\underline{\xs}I_d\leq \xs\xs^{\top}(x)\leq \overline{\xs} I_d,
\]
where $I_d$ is the identity matrix on $\R^d$.
\end{assumption}

\begin{definition}\label{def:connecting-orbits-measure}
For a semigroup $\str P_t$ on $\PP_2(\R^d)$, a path $\{\mu_t\}_{t\in\mathbb{R}}\subset\PP_2(\R^d)$ is called an increasing (decreasing) connecting orbit from an invariant measure $\nu$ to another invariant measure $\mu$ if
\begin{align*}
&\str P_t\mu_s=\mu_{t+s}\text{ for any $t\geq0$, $s\in\mathbb{R}$};\\
&\mu_s\leq_{\mathrm{st}}\mu_t\ \ (\mu_s\geq_{\mathrm{st}}\mu_t)\text{ for all $s\leq t$};\\
&\WW_2(\mu_t,\nu)\to0\text{ as $t\to-\8$ and }\WW_2(\mu_t,\mu)\to0\text{ as $t\to\8$}.
\end{align*}
\end{definition}

\begin{definition}\label{def:locally-dissipative}
The equation \eqref{eq:mvsystem} is called locally dissipative at $a\in\R^d$ with configuration $(r_a,g_a)$, if
\begin{enumerate}[label=(\roman*)]
\item the function $g_a\colon\R_+\times\R_+\to\R$ satisfies
\begin{equation*}
2\langle x,b(x+a,\nu)\rangle+|\sigma(x+a)|_2^2\leq -g_a(|x|^2,\|\nu^a\|_2^2),\quad\text{for all $x\in\R^d$ and $\nu\in\PP_2(\R^d)$},
\end{equation*}
where $\nu^{a}$ is the shift probability of $\nu$ by $a$, i.e.,  $\nu^{a}(B):=\nu(B+a)$ for any $B\in\BB(\R^d)$;
\item $r_a>0$ and $g_a$ satisfy
\begin{equation*}
\begin{split}
&g_a(\cdot,r_a^2) \text{ is continuous and convex};\\
&\inf_{0\leq w\leq r_a^2}g_a(z,w)=g_a(z,r_a^2);\\
& g_a(z,r_a^2)>0, \ \text{ for all } z\geq r_a^2.
\end{split}
\end{equation*}
\end{enumerate}
\end{definition}

\begin{theorem}\label{thm:existence-unstable}
Suppose that Assumption \ref{asp:lipschitz}, \ref{asp:cooperation}, \ref{asp:dissipative-growth-nondegeneracy} hold, and there exists $\{a_i\}_{i=1}^n\subset \R^d$ for some $n\geq 2$ such that the equation \eqref{eq:mvsystem} is locally dissipative at $a_i$  with configurations $(r_{a_i},g_{a_i})$ for all $1\leq i\leq n$.
If $a_1<a_2<\cdots<a_n$ and
\[
r_{a_i}+r_{a_{i+1}}\leq \abs{a_{i+1}-a_i}, \ \text{ for all } \ 1\leq i\leq n-1,
\]
then the equation \eqref{eq:mvsystem} has at least $(2n-1)$ order-related invariant measures in $\PP_{\8}(\R^d)$,
\[
\mu_1<_{\mathrm{st}}\nu_1<_{\mathrm{st}}\mu_2<_{\mathrm{st}}\cdots<_{\mathrm{st}}\mu_i<_{\mathrm{st}}\nu_i<_{\mathrm{st}}\mu_{i+1}<_{\mathrm{st}}\cdots<_{\mathrm{st}}\mu_{n-1}<_{\mathrm{st}}\nu_{n-1}<_{\mathrm{st}}\mu_n,
\]
satisfying
\begin{enumerate}[label=\textnormal{(\roman*)}]
\item There hold
\begin{equation}\label{eq:nu in ball}
\mu_i\in B_{\PP_2}(\xd_{a_i},r_{a_i}), \ \text{ for all } i=1,2,\dots,n,
\end{equation}
and
\begin{equation}\label{eq:nu notin ball}
\nu_{i}\notin \overline{B_{\PP_2}(\xd_{a_i},r_{a_i})}\bigcup\overline{B_{\PP_2}(\xd_{a_{i+1}},r_{a_{i+1}})}, \ \text{ for all } i=1,2,\dots,n-1.
\end{equation}
Moreover, $\str P_tB_{\PP_2}(\xd_{a_i},r_{a_i})\subset B_{\PP_2}(\xd_{a_i},r_{a_i})$ for each $i=1,2,\dots,n$ and $t\geq 0$;
\item There exists a decreasing connecting orbit in the order interval $[\mu_i,\nu_i]_{\PP_2}$ and an increasing connecting orbit in the order interval $[\nu_i,\mu_{i+1}]_{\PP_2}$, for each $\ 1\leq i\leq n-1$.
\item Assume further there are exactly $(2n-1)$ invariant measures, then there exists a decreasing connecting orbit $\{\mu_{i,i}(t)\}_{t\in\R}$ from $\nu_i$ to $\mu_i$ and an increasing connecting orbit $\{\mu_{i,i+1}(t)\}_{t\in\R}$ from $\nu_i$ to $\mu_{i+1}$, for each $\ 1\leq i\leq n-1$.
\end{enumerate}
\end{theorem}

\begin{remark}
In a subsequent work, the local dissipativity condition and the conditions on  the number of invariant measures for general one-dimensional granular media equations will be investigated, and, building on the present theorem, the associated basins of attraction will be analyzed.
\end{remark}

As an illustration, we apply the abstract result to one-dimensional granular media equations with double-well, multi-well, and perturbed double-well confining potentials, as well as to a two-dimensional model describing the dynamics of an interacting two-species population under external forces, intra-species and inter-species interacting forces, and stochastic noises.

\begin{theorem}[Double-well landscapes, Figure \ref{fig:double-well}]\label{thm:double-well}
Consider the following one-dimensional McKean-Vlasov SDE,
\begin{equation}\label{eq:double-well}
\d X_t=-\left[X_t(X_t-1)(X_t+1)+\xb\left(X_t-\bE X_t\right)\right]\d t+\xs\d W_t.
\end{equation}
If the parameters $\xb$ and $\xs$ satisfy
\[
\xb\geq \frac{27(9+\sqrt{17})}{128},\quad 0<\xs^2<\frac{51\sqrt{17}-107}{256},
\]
then
\begin{enumerate}[label=\textnormal{(\roman*)}]
\item there are exactly three invariant measures, $\mu_{-1}, \mu_0, \mu_1\in\PP_{\8}(\R)$ with $\mu_{-1}<_{\mathrm{st}}\mu_0<_{\mathrm{st}}\mu_1$;
\item there is a decreasing connecting orbit $\{\mu_{0,-1}(t)\}_{t\in\R}$ from $\mu_0$ to $\mu_{-1}$ and an increasing connecting orbit $\{\mu_{0,1}(t)\}_{t\in\R}$ from $\mu_0$ to $\mu_1$.
\end{enumerate}
\end{theorem}

\begin{figure}[hbtp]
\centering
\includegraphics{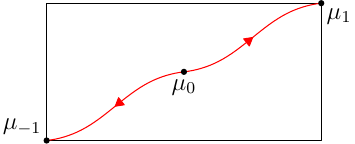}
\caption{Phase diagram for double-well landscapes}
\label{fig:double-well}
\end{figure}

\begin{remark}
The existence of connecting orbits makes explicit the unstable role played by invariant measures serving as backward limits.
More precisely, there exists a solution evolving along the  connecting orbit whose initial condition can be chosen arbitrarily close to the invariant measure $\mu_0$, yet the trajectory escapes from $\mu_0$ (and eventually converges to another invariant measure $\mu_1$).
Such unstable behavior implies that $\mu_0$ is not \emph{Lyapunov stable}.
Here, an invariant measure $\nu$ is Lyapunov stable means that, for any neighbourhood $U$ of $\nu$, there exists a neighbourhood $V \subset U$ such that
solutions starting from $V$ remain in $U$ for all forward times.
At the same time, $\mu_0$ does not \emph{attract points locally}.
That is, an invariant measure $\nu$ is said to \emph{attract points locally} if all solutions initiated within some neighbourhood of $\nu$ are attracted to it.
This violates all the key conditions for the invariant measure $\mu_0$ to be a possible local attractor (see e.g., Hale \cite{Ha88}).
\end{remark}

\begin{theorem}[Multi-well landscapes]\label{thm:multi-well}
Consider the following one-dimensional McKean-Vlasov SDE,
\begin{equation}\label{eq:multi-well}
\d X_t=-\left[X_t(X_t-1)(X_t+1)(X_t-2)(X_t+2)+\xb\left(X_t-\bE X_t\right)\right]\d t+\xs(X_t)\d W_t.
\end{equation}
If the parameter $\beta$ and the Lipschitz function $\xs\colon\R\to\R$ satisfy
\[
\beta\geq8\sqrt{5+\sqrt{13}},\quad 0<\inf_{x\in\R}\xs^2(x)\leq \sup_{x\in\R}\xs^2(x)<\frac{4(13\sqrt{13}-35)}{27},
\]
then
\begin{enumerate}[label=\textnormal{(\roman*)}]
\item there are five invariant measures, $\mu_{-2}, \mu_{-1}, \mu_0, \mu_1, \mu_2\in\PP_{\8}(\R)$ with $\mu_{-2}<_{\mathrm{st}}\mu_{-1}<_{\mathrm{st}}\mu_0<_{\mathrm{st}}\mu_1<_{\mathrm{st}}\mu_2$;
\item there is a decreasing connecting orbit in the order interval $[\mu_{-2},\mu_{-1}]_{\PP_2}$ and $[\mu_{0},\mu_{1}]_{\PP_2}$, respectively, and an increasing connecting orbit in the order interval $[\mu_{-1},\mu_{0}]_{\PP_2}$ and $[\mu_{1},\mu_{2}]_{\PP_2}$, respectively.
\end{enumerate}
\end{theorem}

\begin{theorem}[Double-well with perturbation]
\label{thm:double-well-perturbation}
Consider the following one-dimensional McKean-Vlasov SDE,
\begin{equation}\label{eq:double-well-perturbation}
\d X_t=-\left[X_t(X_t-1)(X_t+1)+f(X_t)+\xb\left(X_t-\bE X_t\right)\right]\d t+\xs(X_t)\d W_t.
\end{equation}
If the parameter $\xb$ and the Lipschitz functions $\xs,f\colon\R\to\R$ satisfies
\[
\xb\geq \frac{65\sqrt5}{48}, \ \ 0<\inf_{x\in\R}\xs^2(x)\leq \sup_{x\in\R}\xs^2(x)\leq \frac{9}{200}, \ \ 
\sup_{x\in\R}\abs{f(x)}\leq \frac13,
\]
then
\begin{enumerate}[label=\textnormal{(\roman*)}]
\item there are three invariant measures, $\mu_{-1}, \mu_0, \mu_1\in\PP_{\8}(\R)$ with $\mu_{-1}<_{\mathrm{st}}\mu_0<_{\mathrm{st}}\mu_1$;
\item  There is a decreasing connecting orbit in the order interval $[\mu_{-1},\mu_0]_{\PP_2}$ and an increasing connecting orbit in the order interval $[\mu_0,\mu_1]_{\PP_2}$.
\end{enumerate}
\end{theorem}

\begin{theorem}[Multi-species population model]\label{thm:high-dimension}
For a given $0<\tau<1$, consider the following two-dimensional McKean-Vlasov SDE, 
\begin{equation}\label{eq:higher-dimensional}
\begin{cases}
    \d X_t=\left[X_t-X^3_t-\tau\alpha\left(X_t-\bE X_t\right)-(1-\tau)\beta\left(X_t-\bE Y_t\right)\right]\d t+\xs\d W_1(t),\\
\,\d Y_t\,=\left[Y_t-Y^3_t-\tau\alpha\left(Y_t-\bE X_t\right)-(1-\tau)\beta\left(Y_t-\bE Y_t\right)\right]\d t+\xs\d W_2(t).
\end{cases}
\end{equation}
If the parameters $\xa,\xb$ and $\xs$ are positive and satisfy
\begin{equation}\label{eq:parameter-high-dimensional}
    |\tau\xa-(1-\tau)\xb|\leq 1, \quad \ \max\{\tau\xa, (1-\tau)\xb\}\geq \frac{27}{2}, \quad \ 0<\sigma^2<\frac{9}{512},
\end{equation}
then
\begin{enumerate}[label=\textnormal{(\roman*)}]
\item there are exactly three invariant measures, $\mu_{-1}, \mu_0, \mu_1\in\PP_{\8}(\R^2)$ with $\mu_{-1}<_{\mathrm{st}}\mu_0<_{\mathrm{st}}\mu_1$;
\item there is a decreasing connecting orbit $\{\mu_{0,-1}(t)\}_{t\in\R}$ from $\mu_0$ to $\mu_{-1}$ and an increasing connecting orbit $\{\mu_{0,1}(t)\}_{t\in\R}$ from $\mu_0$ to $\mu_1$.
\end{enumerate}
\end{theorem}

\section{Stochastic Order and Compatible Cone}\label{Section:Order-Cone}

In this section, we first summarize some preliminary concepts on general ordered spaces. Then we investigate some key properties of the stochastic order in Wasserstein spaces. Next, we extend Wasserstein spaces into a normed vector space and construct a cone within it to further extend the stochastic order.

\subsection{Properties of stochastic order}\label{Subsection:Properties-Order}
Firstly, we recall some concepts of general partial order relations and general ordered spaces.

\begin{definition}
Let $S$ be a topological space.
A set $\RR\subset S\times S$ is called a partial order relation if the following hold:
\begin{enumerate}[label=(\roman*)]
\item (Reflexivity) $(x,x)\in\RR$ for all $x\in S$;
\item (Antisymmetry) $(x,y)\in\RR$ and $(y,x)\in\RR$ imply $x=y$;
\item (Transitivity) $(x,y)\in\RR$ and $(y,z)\in\RR$ imply $(x,z)\in\RR$.
\end{enumerate}
We write $x\leq y$ if $(x,y)\in\RR$.
Furthermore, $\RR$ is said to be a closed partial order relation if it is a closed subset of $S\times S$.
The space $S$ is called an \emph{ordered space} if it is a topological space together with a closed partial order relation $\RR\subset S\times S$. 
\end{definition}

Let $S$ be a topological space with a closed partial order relation ``$\leq$''. 
If $x\leq y$ and $x\neq y$, we write $x<y$. 
For $x,y\in S$ with $x\leq y$, the \emph{order interval} is defined by $(x,y]_S=\{z\in S:x< z\leq y\}$, $[x,y]_S=\{z\in S:x\leq z\leq y\}$ and similarly we can define $[x,y)_S, (x,y)_S$.
Points $x, y\in S$ are said to be \emph{order-related}, if $x\leq y$ or $y\leq x$ holds.
Given two subsets $A$ and $B$ of $S$, we write $A\leq B$ ($A<B$) when $x\leq y$ ($x<y$) holds for each choice of $x\in A$ and $y\in B$.
The reversed signs $\geq$, $>$  
are used in the usual way. 

A closed partial order relation in a topological vector space can be induced by a cone. To be more precise, let $(V,\mathcal{T})$ be a topological vector space. A \emph{cone} $C$ is a closed convex subset of $V$ such that $\lambda C\subset C$ for all $\lambda>0$ and $C\cap(-C)=\{0\}$. We call $(V,C)$ an ordered topological vector space, as the cone $C$ can induce a closed partial order as follows.  For $x,y\in V$, $x\leq y$ if and only if $y-x\in C$. It follows that $x<y$ if and only if $y-x\in C\backslash\{0\}$. In addition, $[x,y]_V$ is clearly a convex closed subset of $V$.

Next, we recall the definition and properties of the stochastic order; we refer to,
e.g., Kamae-Krengel-O’Brien \cite{Kamae1977}, Lindvall \cite[Chapter IV]{Lindvall1992} for a detailed discussion. 
We first present the coordinate order on $\R^d$. 
That is,
\[
x\leq y\ \ \text{if and only if}\ \ x_i\leq y_i\ \ \text{for all $i\in\{1,2,\dots,d\}$},
\]
and a function $f\colon\R^d\to\R$ is called \emph{increasing} if $f(x)\leq f(y)$ whenever $x\leq y$.
Now we give the definition of the stochastic order.
Let $\PP(\R^d)$ be the set of probability measures on $\R^d$.

\begin{definition}\label{def:stochastic-order}
The stochastic order ``$\leq_{\mathrm{st}}$'' is a partial order on $\PP(\R^d)$, defined by
\[
\mu\leq_{\mathrm{st}}\nu\text{ if and only if }\int_{\R^d}f\d\mu\leq\int_{\R^d}f\d\nu\text{ for all bounded increasing functions $f\colon\R^d\to\R$}.
\]
\end{definition}
For simplicity, we just write $\mu\leq_{\mathrm{st}}\nu$ as $\mu\leq\nu$ hereafter.
Reflexivity and transitivity are obvious, and antisymmetry can be obtained by taking $f=\one_{(x,\8)}$, where $x=(x_i)_{i=1}^d\in\R^d$, $(x,\8)=\prod_{i=1}^d(x_i,\8)$.
So, it is indeed a partial order.
It is also known as stochastic domination.
A basic fact is Strassen's theorem (see e.g., Lindvall \cite[Theorem IV.2.4]{Lindvall1992} or \cite[equation (3)]{Lindvall1999}) that says, $\mu\leq\nu$ if and only if there exist two random variables $X,Y$ such that
\[
\LL(X)=\mu,\quad\LL(Y)=\nu,\quad X\leq Y\ \as,
\]
where $\LL(\cdot)$ represents the law of a random variable.
According to Hiai-Lawson-Lim \cite[Proposition 3.11]{Hiai2018}, test functions $f$ in Definition \ref{def:stochastic-order} can be taken as all bounded continuous increasing functions.
If $\mu$, $\nu$ both have finite first moments, test functions $f$ in Definition \ref{def:stochastic-order} can be taken as all $1$-Lipschitz increasing functions (see Fritz-Perrone \cite[Theorem 4.2.1]{Fritz2020}), where a $1$-Lipschitz function $f$ means $\abs{f(x)-f(y)}\leq\abs{x-y}$ for all $x,y\in\R^d$.
We may confine test functions in these different classes of functions hereafter for convenience.

Let $\PP_p(\R^d)$ be the set of probability measures on $\R^d$ having finite $p$-th moments,
\[
\PP_p(\R^d):=\left\{\mu\in\PP(\R^d): \norm{\mu}_p=\left(\int_{\R^d}\abs x^p\d\mu(x)\right)^{1/p}<\8\right\}.
\]
For $p\in[1,\8)$, equip $\PP_p(\R^d)$ with the $p$-Wasserstein metric defined as below,
\[
\WW_p(\mu,\nu):=\inf\left\{\left(\bE\abs{X-Y}^p\right)^{1/p}:\LL(X)=\mu,\, \LL(Y)=\nu\right\}.
\]
Then $(\PP_p(\R^d),\WW_p)$ is a Polish space (see Villani \cite[Theorem 6.18]{Villani2009}).
We will also use the notation $\mu_n\goto{\WW_p}\mu$ for a sequence $\{\mu_n\}_{n\in\N}\subset\PP_p(\R^d)$ converging to $\mu\in\PP_p(\R^d)$ as $n\to\8$ under the metric $\WW_p$.

In the following, we give some properties of stochastic order that will be used later.
The following lemma shows the stochastic order is closed in $\PP(\R^d)$ and in $\PP_p(\R^d)$.
\begin{lemma}\label{lem:stochastic-order-closed}
The stochastic order is a closed partial order relation in $\PP(\R^d)$ with the weak convergence topology, and also closed in the $p$-Wasserstein space $(\PP_p(\R^d),\WW_p)$ for $p\geq1$.
\end{lemma}

\begin{proof}
Denote a sequence $\{\mu_n\}_{n\in\N}\subset\PP(\R^d)$ weakly converging to $\mu$ as $n\to\8$ by $\mu_n\goto{w}\mu$  as $n\to\8$.
Now, suppose $\mu_n\goto{w}\mu$, $\nu_n\goto{w}\nu$, as $n\to\8$, and $\mu_n\leq\nu_n$ for each $n\in\N$.
By Hiai-Lawson-Lim \cite[Proposition 3.11]{Hiai2018}, to show $\mu\leq\nu$ we only need to show
\[
\int_{\R^d}f\d\mu\leq\int_{\R^d}f\d\nu\;\text{ for all bounded continuous increasing functions $f$}.
\]
This directly follows from $\int f\d\mu_n\leq\int f\d\nu_n$ for each $n\in\N$ and $\mu_n\goto{w}\mu$, $\nu_n\goto{w}\nu$.
Since $\mu_n\to\mu$ in $\WW_p$ implies $\mu_n\goto{w}\mu$, the stochastic order is also closed in $\PP_p(\R^d)$.
\end{proof}

Given $p\geq1$ and $\mu_1, \mu_2\in\PP(\R^d)$ with $\mu_1\leq\mu_2$. Define 
$$[\mu_1,\mu_2]_{\PP}:=\{\mu\in\PP(\R^d): \mu_1\leq\mu\leq\mu_2\}\text{\ and \ }[\mu_1,\mu_2]_{\PP_p}:=\{\mu\in\PP_p(\R^d): \mu_1\leq\mu\leq\mu_2\}$$
Then, we show that an order interval is bounded by endpoints in $\PP_p(\R^d)$ as follows.

\begin{lemma}\label{lem:order-interval-bounded}
Given $p\geq1$ and $\mu_1, \mu_2\in\PP_p(\R^d)$  with $\mu_1\leq\mu_2$.
Then the order interval $[\mu_1,\mu_2]_{\PP}$ is a bounded subset of $(\PP_p(\R^d),\WW_p)$.
\end{lemma}

\begin{proof}
    For any $x=(x_1,\cdots,x_d)\in \R^d$ and $N>0$, let
    \begin{equation*}
        f_{i,N}^+(x)=\big(|x_i|^p\wedge N\big)\one_{\{x_i> 0\}}(x),\ \ f_{i,N}^-(x)=-\big(|x_i|^p\wedge N\big)\one_{\{x_i\leq 0\}}(x), \ \ 1\leq i\leq d.
    \end{equation*}
    Then $f_{i,N}^+, f_{i,N}^-$ are bounded increasing functions for all $1\leq i\leq d$ and $N>0$. 

    For any $\mu\in [\mu_1,\mu_2]_{\PP}$, we have
    \begin{equation*}
        \int_{\R^d}f_{i,N}^+\d \mu\leq \int_{\R^d}f_{i,N}^+\d \mu_2, \ \ \int_{\R^d}f_{i,N}^-\d \mu\geq \int_{\R^d}f_{i,N}^-\d \mu_1, \ \text{ for all } \ 1\leq i\leq d.
    \end{equation*}
    Thus
    \begin{equation*}
        \int_{\R^d}\sum_{i=1}^d\big(|x_i|^p\wedge N\big)\d \mu(x)\leq \int_{\R^d}\sum_{i=1}^d\big(|x_i|^p\wedge N\big)\d \mu_1(x)+\int_{\R^d}\sum_{i=1}^d\big(|x_i|^p\wedge N\big)\d \mu_2(x).
    \end{equation*}
    Letting $N\to \8$ and by the monotone convergence theorem, we conclude that
    \begin{equation*}
        \int_{\R^d}\sum_{i=1}^d|x_i|^p\d \mu(x)\leq \int_{\R^d}\sum_{i=1}^d|x_i|^p\d \mu_1(x)+\int_{\R^d}\sum_{i=1}^d|x_i|^p\d \mu_2(x).
    \end{equation*}
    Note that
    \begin{equation*}
        |x|^p\leq \big(2^{\frac{p}{2}-1}\vee 1\big)\sum_{i=1}^d|x_i|^p\leq 2^{|\frac{p}{2}-1|}|x|^p, \ \text{ for all } \ x=(x_1,\cdots,x_d)\in \R^d.
    \end{equation*}
    We get
    \begin{equation*}
        \|\mu\|_p^p=\int_{\R^d}|x|^p\d \mu(x)\leq 2^{|\frac{p}{2}-1|}\bigg(\int_{\R^d}|x|^p\d \mu_1(x)+\int_{\R^d}|x|^p\d \mu_2(x)\bigg)\leq2^p\big(\|\mu_1\|_p^p+\|\mu_2\|_p^p\big).
    \end{equation*}
    Therefore, $[\mu_1,\mu_2]_{\PP}$  is a bounded subset of $(\PP_p(\R^d),\WW_p)$.
\end{proof}

\subsection{Construction of compatible cone}\label{Subsection:Construction-Cone}
As we mentioned in the beginning of Section \ref{Subsection:Properties-Order}, a closed partial order relation in a topological vector space can be induced by a cone.
In the following, we will extend Wasserstein space to a normed vector space equipped with a Kantorovich-type norm and identify a cone that induces a partial order in the normed  vector space, which coincides with the stochastic order when restricted to $\PP_1(\R^d)$.

Firstly, recall  that $1$-Wasserstein metric $\WW_1$ has  Kantorovich-Rubinstein duality (see e.g., Villani \cite[(5.11)]{Villani2009}),
\begin{equation}\label{eq:w1-dual}
\WW_1(\mu,\nu)=\sup\left\{\int_{\R^d}f\d\mu-\int_{\R^d}f\d\nu:\text{ $f$ is $1$-Lipschitz}\right\}.
\end{equation}
Such duality enables us to construct a norm on the space of finite signed Borel measures on $\R^d$ with finite first moments,
\[
\MM_1(\R^d):=\left\{\mu\in\MM(\R^d):\int_{\R^d}\abs x\d\abs{\mu}(x)<\8\right\}.
\]
Define
\begin{equation}\label{eq:extended-wasserstein-metric}
\norm{\mu}_{\WW_1}:=\sup\left\{\abs{\int_{\R^d}f\d\mu}:\text{ $\abs{f(0)}\leq1$, $f$ is $1$-Lipschitz}\right\}, \ \ \text{for\ }\mu\in\MM_1(\R^d).
\end{equation}
Then the following lemma implies $(\MM_1(\R^d),\norm{\cdot}_{\WW_1})$ is a normed vector space.

\begin{lemma}\label{lem:m1-norm-space}
$(\MM_1(\R^d),\norm{\cdot}_{\WW_1})$ is a normed vector space.
And, $\norm{\mu-\nu}_{\WW_1}=\WW_1(\mu,\nu)$ for all $\mu,\nu\in\PP_1(\R^d)$.
\end{lemma}

\begin{proof}
Clearly, one has $\norm{\xl\mu}_{\WW_1}=\abs{\xl}\norm{\mu}_{\WW_1}$ and $\norm{\mu+\nu}_{\WW_1}\leq\norm{\mu}_{\WW_1}+\norm{\nu}_{\WW_1}$.
By the dual representation \eqref{eq:w1-dual} of $\WW_1$, we have $\norm{\mu-\nu}_{\WW_1}=\WW_1(\mu,\nu)$ for $\mu,\nu\in\PP_1(\R^d)$.
It remains to show that $\norm{\mu}_{\WW_1}=0$ implies $\mu=0$.
Take $x=(x_i)_{i=1}^d\in\R^d$ and define $f_{n,i}\colon\R\to\R$, $f_n\colon\R^d\to\R$ by
\begin{equation}\label{eq:construction-fni}
f_{n,i}(u):=\begin{cases}
0,&u\leq x_i,\\
n(u-x_i),&x_i<u<x_i+1/n,\\
1,&u\geq x_i+1/n,
\end{cases}\qquad
f_n(u_1,\ldots,u_d):=\prod_{i=1}^df_{n,i}(u_i).
\end{equation}
Then $\abs{f_n(0)}\leq1$, $\frac{1}{n\sqrt{d}}f_n$ is $1$-Lipschitz, and $f_n\to\one_{(x,\8)}$ as $n\to\8$.
Since $\norm{\mu}_{\WW_1}=0$, it follows $\int f_n\d\mu=0$ for each $n\in\N$.
Then the dominated convergence theorem shows $\mu((x,\8))=0$, and this means $\mu=0$ as $x\in\R^d$ is arbitrary.
\end{proof}

\begin{remark}\label{rem:p1-closed-m1}
The $1$-Wasserstein space $\PP_1(\R^d)$ is a closed convex subset of $(\MM_1(\R^d),\norm{\cdot}_{\WW_1})$, since $(\PP_1(\R^d),\WW_1)$ is a complete metric space, and  $\norm{\mu-\nu}_{\WW_1}=\WW_1(\mu,\nu)$ for all $\mu,\nu\in\PP_1(\R^d)$.
\end{remark}

\begin{remark}\label{rem:m1-not-banach}
The space $(\MM_1(\R^d),\norm{\cdot}_{\WW_1})$ is NOT a Banach space.
W.L.O.G, let $d=1$, and $\mu_n=\xd_{1/n^2}-\xd_0$, $\nu_n=\sum_{k=1}^n\mu_k$, for $n\in\N$.
Then $\{\nu_n\}$ is a Cauchy sequence in $(\MM_1(\R),\norm{\cdot}_{\WW_1})$.
Assume $\norm{\nu_n-\nu}_{\WW_1}\to0$ as $n\to\8$ for some $\nu\in\MM_1(\R)$.
Define $f_m\colon\R\to\R$ by
\[
f_m(u):=\begin{cases}
0,&u\leq0,\\
m^2u,&0<u<1/m^2,\\
1,&u\geq1/m^2,
\end{cases}
\]
and we have $\int f_m\d\nu_n\geq m$ for $n\geq m$.
Let $n\to\8$, and it gives $\abs{\nu}(\R)\geq\int f_m\d\nu\geq m$.
As $m$ can be arbitrarily large, we obtain $\abs{\nu}(\R)=\8$.
This contradicts $\nu\in\MM_1(\R)$.
Indeed, the limit of $\{\nu_n\}$ lives in the dual of some Lipschitz function space (see Bouchitté-Champion-Jimenez \cite{Bouchitté-Champion-Jimenez2005}).
\end{remark}

To recover the stochastic order, a candidate cone in $\MM_1(\R^d)$ is as below,
\begin{equation}\label{eq:cone}
\begin{aligned}
C:=\bigg\{\mu\in\MM_1(\R^d): &\int_{\R^d}f\d\mu\geq0\text{ for all non-negative $1$-Lipschitz increasing functions $f$}\bigg\}.
\end{aligned}
\end{equation}
Let $\mu,\nu\in\MM_1(\R^d)$. If $\nu-\mu\in C$, we say $\mu\leq_{\MM_1}\nu$.
If $\nu-\mu\in  C\backslash\{0\}$, we write $\mu<_{\MM_1}\nu$.
And if $\mu\leq_{\MM_1}\nu$, we denote $[\mu,\nu]_{\MM_1}:=\{\rho\in\MM_1(\R^d):\ \mu\leq_{\MM_1}\rho\leq_{\MM_1}\nu\}$.

\begin{lemma}\label{lem:m1-cone}
The set $C$ given in (\ref{eq:cone}) is a cone in $\MM_1(\R^d)$.
In other words, $C$ is a closed convex subset of $\MM_1(\R^d)$, $\xl C\subset C$ for all $\xl>0$, and $C\cap(-C)=\{0\}$.
Moreover, for $\mu,\nu\in\PP_1(\R^d)$, $\mu\leq\nu$ if and only if $\mu\leq_{\MM_1}\nu$.
\end{lemma}

\begin{proof}
It is obvious that $C$ is convex, and $\xl C\subset C$ for all $\xl>0$.
To see $C\cap(-C)=\{0\}$, we only need to show $\mu=0$ if $\int f\d\mu=0$ for any non-negative $1$-Lipschitz increasing function $f$.
Considering the functions $f_n$ given in \eqref{eq:construction-fni} again, we have $\frac{1}{n\sqrt{d}}f_n$ is non-negative $1$-Lipschitz increasing, for each $n\in\N$.
Also, $f_n\to\one_{(x,\8)}$ as $n\to\8$.
It follows from the dominated convergence theorem that $\mu((x,\8))=0$.
This gives $\mu=0$ since $x\in\R^d$ is arbitrary.

To show $C$ is closed, suppose $\norm{\mu_n-\mu}_{\WW_1}\to0$, $\mu_n\in C$ for all $n\in\N$, and $\mu\in\MM_1(\R^d)$.
We need to show $\mu\in C$ as well.
Let $f$ be a non-negative $1$-Lipschitz increasing function.
Then there exists some $\xl\geq1$ such that $\frac1{\xl}\abs{f(0)}\leq1$.
We observe $\abs{\int\frac1{\xl}f\d\mu_n-\int\frac1{\xl}f\d\mu}\leq\norm{\mu_n-\mu}_{\WW_1}\to0$, so $\int f\d\mu=\lim_{n\to\8}\int f\d\mu_n\geq0$, which shows $\mu\in C$ as $f$ is arbitrary.

Next, we want to show the partial order ``$\leq_{\MM_1}$'' in $\MM_1(\R^d)$ induced by $ C$ is equivalent to the stochastic order when restricted to $\PP_1(\R^d)$.
As we mentioned before, Fritz-Perrone \cite[Theorem 4.2.1]{Fritz2020} entails that, for $\mu, \nu\in\PP_1(R^d)$, $\mu\leq\nu$ is equivalent to $\int f\d\mu\leq\int f\d\nu$ for any $1$-Lipschitz increasing function $f$.
On the one hand, for $\mu, \nu\in\PP_1(R^d)$ with $\mu\leq\nu$, it is easy to see from the definition of the cone $C$ that, $\nu-\mu\in C$.
On the other hand, for $\mu, \nu\in\PP_1(R^d)$ with $\nu-\mu\in C$, and a $1$-Lipschitz increasing function $f$,  one has
\[
\int\left(f\vee(-n)+n\right)\d\mu\leq\int\left(f\vee(-n)+n\right)\d\nu\quad\text{ for all $n\in\N$},
\]
so we have
\[
\int f\vee(-n)\d\mu\leq\int f\vee(-n)\d\nu\quad\text{ for all $n\in\N$}.
\]
Since $\abs{f(x)}\leq\abs{f(0)}+\abs x$, $\mu ,\nu\in\PP_1(\R^d)$, we invoke the dominated convergence theorem to get $\int f\d\mu\leq\int f\d\nu$.
Hence, by \cite[Theorem 4.2.1]{Fritz2020}, we have $\mu\leq\nu$.
\end{proof}

\begin{remark}
    For any $\mu_1,\mu_2\in \MM_1(\R^d)$, denote
    \begin{equation*}
        [\mu_1,\mu_2]_{\MM_1}:=\{\mu\in \MM_1(\R^d): \mu_1\leq_{\MM_1} \mu\leq_{\MM_1} \mu_2\}.
    \end{equation*}
    Even if $\mu_1,\mu_2\in\PP_{\8}(\R^d)$, we do not have $[\mu_1,\mu_2]_{\MM_1}\subset \PP(\R^d)$. 
    Here, $\PP_{\8}(\R^d)$ denotes the set of probability measures on $\R^d$ with finite moments of all orders $p\geq1$.
    In fact, let $x,y,z\in \R^d$ with $x<y<z$.
    Then it is easy to see that $\delta_x<_{\MM_1}\delta_y<_{\MM_1}\delta_z$ and hence
    \begin{equation*}
        \delta_x<_{\MM_1}\delta_z-\delta_y+\delta_x<_{\MM_1}\delta_z.
    \end{equation*}
    So $\delta_z-\delta_y+\delta_x\in [\delta_x,\delta_z]_{\MM_1}$. However, $\delta_z-\delta_y+\delta_x\notin \PP(\R^d)$.
\end{remark}

\begin{remark}\label{R:empty interior of cone}
The cone $C$ has no interior point in $\MM_1(\R^d)$.
Indeed, for any $\mu\in C$ and any $\xe>0$, we can show there is some $\nu\in\MM_1(\R^d) \backslash C$ such that $\norm{\mu-\nu}_{\WW_1}<\xe$.
Actually, there exists some $N>0$ such that $\int_{\{\abs x>N\}}(1+\abs x)\d\abs{\mu}<\frac{\xe}{2}$.
Set $y=(N+1,0,\dots,0)\in\R^d$ and let
\[
\d\nu(x)=\one_{\{\abs x\leq N\}}\cdot\d\mu(x)-\frac{\xe}{2(2+N)}\one_{\{\abs x>N\}}\cdot\d\xd_y(x).
\]
By the construction of $\nu$, we have $\nu\in\MM_1(\R^d)$ and $\norm{\mu-\nu}_{\WW_1}<\xe$.
Consider $f\colon\R^d\to\R$ defined below,
\[
f(x_1,\ldots,x_d):=\begin{cases}
0,&x_1\leq N,\\
x_1-N,&N<x_1<N+1,\\
1,&x_1\geq N+1.
\end{cases}
\]
Then $f$ is non-negative $1$-Lipschitz increasing and $\int f\d\nu=-\frac{\xe}{2(2+N)}<0$, which implies $\nu\notin C$.
\end{remark}

\section{Generation of Monotone Dynamical Systems}\label{Section:Generation-MDS}

In this section, we first verify the joint continuity of the semigroup $\str P$ induced by a McKean–Vlasov SDE, thereby showing that it can be regarded as a semiflow. 
Subsequently, under the cooperative condition, we prove  a comparison theorem for McKean–Vlasov SDEs with respect to the stochastic  order and show that the corresponding single equation gives rise to a monotone semiflow $\str P$ on $\mathcal{P}_2(\mathbb{R}^d)$.

\subsection{Joint continuity of semigroups of McKean-Vlasov SDEs}\label{Subsection:Generation-DS}

Denote $\mathbb{T_+}=\{x\in\mathbb{T}: x\geq0\}$, $\mathbb{T_-}=\{x\in\mathbb{T}: x\leq0\}$, where $\mathbb{T}=\mathbb{R}\text{ or } \mathbb{Z}$. 
Let $(X,d)$ be a complete metric space with a metric $d$.
Firstly, we give the definition of a continuous-time dynamical system, or named semiflow. 

\begin{definition}\label{D:semiflow}
A map $\Phi: 
\mathbb{R_+}\times X\rightarrow X$ is said to be a semiflow, if 
\begin{enumerate}[label=(\roman*)]
    \item  $\Phi: \mathbb{R_+}\times X\rightarrow X$ is continuous;
    \item $\Phi(0,x)=x$;
    \item $\Phi(t,\Phi(s,x))=\Phi(t+s,x)$ for all $t,s\in \mathbb{R_+}$ and $x\in X$.
\end{enumerate}
\end{definition}

In the following, we will verify that a McKean-Vlasov SDE will generate a semiflow on $\PP_2(\R^d)$.
Consider a McKean-Vlasov SDE on $\R^d$ as follows.
\begin{equation}\label{eq:mvsde}
\d X_t=b(X_t,\LL(X_t))\d t+\xs(X_t)\d W_t,\qquad X_0=\xi.
\end{equation}
Let $X_t^{\xi}$ be the solution of \eqref{eq:mvsde} with initial condition $\xi$.
In fact, under Assumption \ref{asp:lipschitz}, the existence and uniqueness of $L^2$-solutions are proved in Wang \cite[Theorem 2.1]{Wang18}.
Define $\str P_t\colon\PP_2(\R^d)\to\PP_2(\R^d)$ by
\[
\str P_t\mu:=\LL(X^{\xi}_t),\text{ with $\LL(\xi)=\mu$}.
\]
$\str P_t$ is well-defined and obviously satisfies the semigroup property.

\begin{proposition}\label{prop:continuity-Pt}
    Under Assumption \ref{asp:lipschitz}, the semigroup $\str P_t$ of \eqref{eq:mvsde} is a semiflow on $(\PP_2(\R^d),\WW_2)$.
\end{proposition}
\begin{proof}
    It suffices to show that for a sequence $(t_n,\mu_n) \in [0,\infty)\times \mathcal{P}_2(\mathbb{R}^d)$ with $t_n\to t$ and $\WW_2(\mu_n,\mu)\to 0$, we have $\WW_2(P^*_{t_n}\mu_n,P^*_{t}\mu)\to 0$.
    
    Firstly, we will prove the following two claims.

    \noindent\textbf{Claim 1.} For any $\mu,\nu\in \mathcal{P}_2(\mathbb{R}^d)$ and $0\leq t\leq T$, there exists $C_T>0$ such that $\WW_2(P^*_{t}\mu,P^*_{t}\nu)\leq C_T \WW_2(\mu,\nu)$.

    \noindent\textbf{Claim 2.} For any $\mu\in \mathcal{P}_2(\mathbb{R}^d)$, $P^*_t\mu$ is continuous in $t$, i.e., for any $t_n\to t$, we have $\WW_2(P^*_{t_n}\mu,P^*_{t}\mu)\to 0$.
    
    \noindent\textbf{Proof of Claim 1:} 
    By \cite[Theorem 4.1]{Villani2009}, we can choose
    $\xi,\eta\in L^2(\Omega;\R^d)$ such that $\LL(\xi)=\mu, \LL(\eta)=\nu$ and $\WW_2^2(\mu,\nu)=\mathbf{E}[|\xi-\eta|^2]$.
    Let $X_t^{\xi}, X_t^{\eta}$ be the solutions of \eqref{eq:mvsde} with initial conditions $\xi, \eta$ respectively. Set $\hat{X}_t:=X_t^{\xi}-X_t^{\eta}$ and
    \begin{equation*}
        \hat{b}_t:=b\big(X_t^{\xi}, \LL(X_t^{\xi})\big)-b\big(X_t^{\eta}, \LL(X_t^{\eta})\big), \ \ 
        \hat{\sigma}_t:=\sigma(X_t^{\xi})-\sigma(X_t^{\eta}).
    \end{equation*}
    Applying It\^o's formula to $|\hat{X}_t|^2$, one has
    \begin{equation*}
       \begin{split}
            \d |\hat{X}_t|^2&=\big(2\langle\hat{X}_t,\hat{b}_t\rangle+\|\hat{\xs}_t\|_2^2\big)\d t+2\langle\hat{X}_t,\hat{\xs}_t\d W_t\rangle\\
            &\leq 4K|\hat{X}_t|^2\d t+K\WW_2^2(\LL(X_t^{\xi}),\LL(X_t^{\eta}))\d t+2\langle\hat{X}_t,\hat{\xs}_t\d W_t\rangle\\
            &\leq 4K|\hat{X}_t|^2\d t+K\mathbf{E}[|\hat{X}_t|^2]\d t+2\langle\hat{X}_t,\hat{\xs}_t\d W_t\rangle.
       \end{split}
    \end{equation*}
    Then it follows that
    \begin{equation*}
       \begin{split}
            \mathbf{E}[|\hat{X}_t|^2]&\leq \mathbf{E}[|\xi-\eta|^2]+\mathbf{E}\bigg[\int_0^t\big(4K|\hat{X}_s|^2+K\mathbf{E}[|\hat{X}_s|^2]\big)\d s\bigg]\\
            &\leq \WW_2^2(\mu,\nu)+5K\int_0^t\mathbf{E}[|\hat{X}_s|^2]\d s.
       \end{split}
    \end{equation*}
    By Gronwall's inequality, for any $t\in [0,T]$, we have
    \begin{equation*}
        \WW_2^2(P^*_{t}\mu,P^*_{t}\nu)\leq \mathbf{E}[|\hat{X}_t|^2]\leq e^{5KT}\WW_2^2(\mu,\nu).
    \end{equation*}
Taking $C_T=\e^{3KT}$, we finish the proof of Claim 1.

    \noindent\textbf{Proof of Claim 2:} Applying It\^ o's formula to $|X_t^{\xi}|^2$ together with Assumption \ref{asp:lipschitz},  we conclude that
    \begin{equation}\label{eq:new1221}
       \begin{split}
            |X_t^{\xi}|^2&=|\xi|^2+\int_0^t\big(2\langle X_s^{\xi}, b(X_s^{\xi},\LL(X_s^{\xi}))\rangle+|\xs(X_s^{\xi})|\big)\d s+2\int_0^t\langle X_s^{\xi}, \xs(X_s^{\xi})\d W_s\rangle\\
            &\leq |\xi|^2+\int_0^t\big((5K+1)|X_s^{\xi}|^2+K\bE[|X_s^{\xi}|^2]+|b(0,\delta_0)|^2+2|\xs(0)|^2\big)\d s\\
            &\qquad\qquad\qquad\qquad\qquad\qquad\qquad\qquad\qquad\qquad +2\int_0^t\langle X_s^{\xi}, \xs(X_s^{\xi})\d W_s\rangle.
       \end{split}
    \end{equation}
    Hence, for any $T>0$, by the Burkholder-Davis-Gundy inequality, we have
    \begin{equation}\label{eq:1009-1}
        \begin{split}
            \bE\Big[\sup_{t\in [0,T]}|X_t^{\xi}|^2\Big]&\leq \bE[|\xi|^2]+(6K+1)\int_0^T\bE[|X_s^{\xi}|^2]\d s+(|b(0,\delta_0)|^2+2|\xs(0)|^2)T\\
            & \ \ \ \ +2\bE\bigg[\sup_{t\in [0,T]}\bigg|\int_0^t\langle X_s^{\xi}, \xs(X_s^{\xi})\d W_s\rangle\bigg|\bigg]\\
            &\leq \bE[|\xi|^2]+(|b(0,\delta_0)|^2+2|\xs(0)|^2)T+(6K+1)\int_0^T\bE\Big[\sup_{t\in [0,s]}|X_t^{\xi}|^2\Big]\d s\\
            & \ \ \ \ +2c_1\bE\bigg[\bigg(\int_0^T|X_s^{\xi}|^2|\xs(X_s^{\xi})|^2\d s\bigg)^{\frac{1}{2}}\bigg].
        \end{split}
    \end{equation}
    Note that
    \begin{equation}\label{eq:1009-2}
        \begin{split}
            &\ \ \ \ 2c_1\bE\bigg[\bigg(\int_0^T|X_s^{\xi}|^2|\xs(X_s^{\xi})|^2\d s\bigg)^{\frac{1}{2}}\bigg]\\
            &=2c_1\bE\bigg[\Big(\sup_{t\in[0,T]}|X_s^{\xi}|\Big)\bigg(\int_0^T|\xs(X_s^{\xi})|^2\d s\bigg)^{\frac{1}{2}}\bigg]\\
            &\leq \frac{1}{2}\bE\Big[\sup_{t\in [0,T]}|X_t^{\xi}|^2\Big]+2c_1^2\int_0^T\bE[|\xs(X_s^{\xi})|^2]\d s\\
            &\leq \frac{1}{2}\bE\Big[\sup_{t\in [0,T]}|X_t^{\xi}|^2\Big]+4c_1^2K\int_0^T\bE\Big[\sup_{t\in [0,s]}|X_t^{\xi}|^2\Big]\d s+4c_1^2|\xs(0)|^2T.
        \end{split}
    \end{equation}
    It follows from \eqref{eq:1009-1}, \eqref{eq:1009-2} that 
    \begin{equation*}
        \begin{split}
            \bE\Big[\sup_{t\in [0,T]}|X_t^{\xi}|^2\Big]&\leq 2\bE[|\xi|^2]+\big(2|b(0,\delta_0)|^2+(4+8c_1^2)|\xs(0)|^2\big)T\\
            &\ \ \ \ +2\big((6+4c_1^2)K+1\big)\int_0^T\bE\Big[\sup_{t\in [0,s]}|X_t^{\xi}|^2\Big]\d s.
        \end{split}
    \end{equation*}
    Then Gronwall's inequality gives
    \begin{equation}\label{eq:mvsde-estimate}
        \bE\Big[\sup_{t\in [0,T]}|X_t^{\xi}|^2\Big]\leq C_T\big(1+\bE[|\xi|^2]\big),
    \end{equation}
    where
    \begin{equation*}
        C_T=\Big(2+\big(2|b(0,\delta_0)|^2+(4+8c_1^2)|\xs(0)|^2\big)T\Big)e^{2\big((6+4c_1^2)K+1\big)T}.
    \end{equation*}
    Note that the solution $\{X_t^{\xi}\}_{t\geq 0}$ is continuous a.s., so $\lim_{n\to \infty}|X_{t_n}^{\xi}-X_t^{\xi}|^2=0$, a.s. Then it follows from \eqref{eq:mvsde-estimate} and the dominated convergence theorem that
    \begin{equation*}
        \limsup_{n\to \infty}\WW_2(P^*_{t_n}\mu,P^*_{t}\mu)\leq \lim_{n\to \infty}\big(\mathbf{E}[|X_{t_n}^{\xi}-X_t^{\xi}|^2]\big)^{1/2}=0.
    \end{equation*}   
 We finish the proof of Claim 2. 

Now, we continue the proof of the proposition. Note that there exists $T>0$ such that $\sup_{n\geq 1}t_n\vee t\leq T$. By Claim 1, we have
    \begin{equation*}
        \begin{split}
            \WW_2(P^*_{t_n}\mu_n,P^*_{t}\mu)&\leq \WW_2(P^*_{t_n}\mu_n,P^*_{t_n}\mu)+\WW_2(P^*_{t_n}\mu,P^*_{t}\mu)\\
            &\leq C_T\WW_2(\mu_n,\mu)+\WW_2(P^*_{t_n}\mu,P^*_{t}\mu),
        \end{split}
    \end{equation*}
    which together with Claim 2 implies that $\lim_{n\to \infty} \WW_2(P^*_{t_n}\mu_n,P^*_{t}\mu)=0$.
\end{proof}

\subsection{Comparison principle with respect to stochastic order}\label{Subsection:Comparison-principle}

Firstly, we give the definition of monotone dynamical systems. 
Let $X$ be a complete metric space with a closed partial order relation ``$\leq$'' and  $\Phi: 
\mathbb{R_+}\times X\rightarrow X$ be a semiflow on $X$.
We will henceforth also denote $\Phi(t,x)$ by $\Phi_{t}(x)$.  
\begin{definition}\label{D:order-preserving}
A semiflow $\Phi$ on $X$ is called monotone (strictly monotone), if $x<y$ implies $\Phi_t(x)\leq\Phi_t(y)$ ($\Phi_t(x)<\Phi_t(y)$) for any $t>0$. 
Similarly, a continuous mapping $\Psi:X\to X$ is monotone (strictly monotone), if $x<y$ implies $\Psi(x)\leq\Psi(y)$ ($\Psi(x)<\Psi(y)$).  
A monotone semiflow  or mapping  is also called an order-preserving semiflow  or mapping.
\end{definition}

In the previous subsection, we have proven that a Mckean-Vlasov SDE generates a semiflow on $\PP_2(\R^d)$. Next, we further prove that under the cooperative condition, this semiflow is a monotone semiflow with respect to the stochastic order.
Let $C(0,T;\PP_2(\R^d))$ be the set of continuous maps $\mu\colon[0,T]\to\PP_2(\R^d)$, on which a natural metric defined by
\[
\WW_{2,T}(\mu,\nu):=\sup_{0\leq t\leq T}\WW_2(\mu(t),\nu(t))
\]
makes the space $C(0,T;\PP_2(\R^d))$ complete.
Let $C(0,\8;\PP_2(\R^d))$ be the set of continuous maps $\mu\colon[0,\8)\to\PP_2(\R^d)$.
We say a sequence $\{\mu^n\}_{n=1}^{\8}$ in $C(0,\8;\PP_2(\R^d))$ converges to $\nu$ if
\[
\WW_{2,T}(\mu^n,\nu)\to0\quad\text{as $n\to\8$, for all $T\geq0$}.
\]
It is well-known that the space $C(0,\8;\PP_2(\R^d))$ is complete under some compatible metric.

We start with a comparison theorem for (extrinsic)-law-dependent SDEs. Taking $\mu,\nu\in C(0,\8;\PP_2(\R^d))$ and considering
\begin{equation}\label{eq:ldsde-comparison-X}
\d X(t)=b(X(t),\mu(t))\d t+\xs(X(t))\d W_t,\  X(0)=\xi,
\end{equation}
\begin{equation}\label{eq:ldsde-comparison-Y}
\d Y(t)= c(Y(t),\nu(t))\d t+\xs(Y(t))\d W_t,\  Y(0)=\eta.
\end{equation}
Let $X^{\mu,\xi}_t$ be the solution of \eqref{eq:ldsde-comparison-X} and $Y^{\nu,\eta}_t$ be the solution of \eqref{eq:ldsde-comparison-Y}.

\begin{assumption}\label{asp:comparison}
The functions $b,\,c\colon\R^d\times\PP_2(\R^d)\to\R^d$ satisfy, for all $x
\in\R^d$, $\mu\in\PP_2(\R^d)$,
\[
b_i(x,\mu)\leq c_i(x,\mu)\ \text{ for all $i=1,2,\dots,d$}.
\]
\end{assumption}

\begin{lemma}\label{lem:ldsde-comparison}
Assume that Assumption \ref{asp:lipschitz}(i) holds for $b$ and $c$, and that Assumption \ref{asp:lipschitz}(ii) and  Assumption \ref{asp:cooperation}(ii) holds for $\xs$. Also assume that Assumption \ref{asp:cooperation}(i) holds for either $b$ or $c$ and that Assumption \ref{asp:comparison} holds.
If $\xi,\,\eta$ are square integrable with $\LL(\xi)\leq\LL(\eta)$, and $\mu(t)\leq\nu(t)$ for all $t\geq0$, then $\LL(X^{\mu,\xi}_t)\leq\LL(Y^{\nu,\eta}_t)$ for all $t\geq0$.
\end{lemma}

\begin{proof}
By Strassen's theorem, we may assume $\xi\leq\eta$ $\as$, and it suffices to show $X(t)\leq Y(t)$ $\as$, for any $t\geq0$.
From Assumption \ref{asp:cooperation}, \ref{asp:comparison} and the condition $\mu(t)\leq\nu(t)$, the assumptions in the classical comparison theorem (see e.g., \cite[Theorem 1.2]{Geib1994}) for usual time-inhomogeneous SDEs, namely \eqref{eq:ldsde-comparison-X} and \eqref{eq:ldsde-comparison-Y}, can be verified easily.
We apply it to obtain $X^{\mu,\xi}_t\leq Y^{\nu,\eta}_t$ a.s.  for all $t\geq0$, which means $\LL(X^{\mu,\xi}_t)\leq\LL(Y^{\nu,\eta}_t)$  for all $t\geq0$.
\end{proof}

Now we want to pass the comparison theorem to the case of intrinsic-law dependence, namely to the McKean-Vlasov SDEs.
To this end, we define a map by
\begin{equation}\label{eq:phi-definition}
\Upsilon_{b}\colon C(0,\8;\PP_2(\R^d))\to C(0,\8;\PP_2(\R^d)),\quad\mu\mapsto\LL(X^{\mu,\xi}),
\end{equation}
where $X^{\mu,\xi}_t$ is the solution of \eqref{eq:ldsde-comparison-X}.
For $\rho\in\PP_2(\R^d)$, we also abuse the notation by using $\Upsilon_{b}(\rho)$ to refer to $\Upsilon_{b}(\{\rho(t)\}_{t\in[0,\8)})$, where $\rho(t)=\rho$ for any $t\in[0,\8)$.

By the definition of $\Upsilon_{b}$, it is easy to see that its fixed point is the solution of the McKean-Vlasov SDE below,
\[
\d X(t)=b(X(t),\LL(X(t)))\d t+\xs(X(t))\d W_t,\qquad X(0)=\xi.
\]
\begin{lemma}\label{lem:phi-fixed-point}
Under Assumption \ref{asp:lipschitz}, the map $\Upsilon_{b}$ has a unique fixed point.
Moreover, the fixed point of $\Upsilon_{b}$ is the limit of $\Upsilon_{b}^n(\LL(\xi))$ in the space $C(0,\8;\PP_2(\R^d))$ as $n\to\8$.
\end{lemma}

\begin{proof}
Let $\mu,\nu\in C(0,\8;\PP_2(\R^d))$.
By Assumption \ref{asp:lipschitz},  for each $\xl>0$ and $t\geq0$, It\^o's formula together with Young's inequality gives
\begin{align*}
&\ \ \ \ \e^{-2\xl t}\WW_2^2\big(\Upsilon_{b}(\mu)_t,\Upsilon_{b}(\nu)_t\big)\leq\e^{-2\xl t}\bE\big[\big|X^{\mu,\xi}_t-X^{\nu,\xi}_t\big|^2\big]\\
&=\int_0^t(-2\xl)\e^{-2\xl s}\bE\big[\big|X^{\mu,\xi}_s-X^{\nu,\xi}_s\big|^2\big]\d s\\
&\ \ \ \ +\int_0^t2\e^{-2\xl s}\bE\big[\inp{X^{\mu,\xi}_s-X^{\nu,\xi}_s}{b(X^{\mu,\xi}_s,\mu_s)-b(X^{\nu,\xi}_s,\nu_s)}\big]\d s\\
&\ \ \ \ +\int_0^t\e^{-2\xl s}\bE\big[\big|\xs(X^{\mu,\xi}_s)-\xs(X^{\nu,\xi}_s)\big|^2\big]\d s\\
&\leq\int_0^t(4K-2\xl)\e^{-2\xl s}\bE\big[\big|X^{\mu,\xi}_s-X^{\nu,\xi}_s\big|^2\big]\d s+\int_0^tK\e^{-2\xl s}\WW_2^2(\mu_s,\nu_s)\d s.
\end{align*}
Choose $\xl=2K$ and fix $\xl$ henceforth, so that
\begin{equation}\label{E:eee}
\e^{-2\xl t}\WW_2^2(\Upsilon_{b}(\mu)_t,\Upsilon_{b}(\nu)_t)\leq K\int_0^t\e^{-2\xl s}\WW_2^2(\mu_s,\nu_s)\d s,\text{ for all $t\geq0$}.
\end{equation}
For $\xl>0$ and $t>0$, define a new metric on $C(0,t;\PP_2(\R^d))$
\[
\WW_{2,t,\xl}(\mu,\nu):=\sup_{0\leq s\leq t}\e^{-\xl s}\WW_2(\mu_s,\nu_s),
\]
Obviously, it is equivalent to $\WW_{2,t}$.
By \eqref{E:eee}, we have
\begin{equation}\label{eq:phi-estime}
\WW_{2,t,\xl}^2(\Upsilon_{b}(\mu),\Upsilon_{b}(\nu))\leq K\int_0^t\WW_{2,s,\xl}^2(\mu,\nu)\d s,\text{ for all $t\geq0$}.
\end{equation}

To show the existence of a fixed point of $\Upsilon_{b}$, we set a sequence $\{\mu^n\}_{n=0}^{\8}$ in $C(0,\8;\PP_2(\R^d))$,
\[
\mu^0_t\equiv\LL(\xi),\quad\mu^n=\Upsilon_{b}(\mu^{n-1}),\text{ for all $n\in\N$}.
\]
Then by the estimate \eqref{eq:phi-estime}, we have
\[
\WW_{2,t,\xl}^2(\mu^{n+1},\mu^n)\leq K\int_0^t\WW_{2,s,\xl}^2(\mu^n,\mu^{n-1})\d s,\text{ for all $t\geq0$, $n\in\N$}.
\]
Similar to \eqref{eq:new1221}, It\^o's formula to $|X^{\mu^0,\xi}_t|^2$ together with Assumption \ref{asp:lipschitz} yields
\begin{equation*}
\|\mu^1_t\|_2^2=\mathbf{E}\big[|X^{\mu^0,\xi}_t|^2\big]\leq \mathbf{E}[|\xi|^2]+(5K+1)\int_0^t\mathbf{E}\big[|X^{\mu^0,\xi}_s|^2\big]\d s+(K\mathbf{E}[|\xi|^2]+|b(0,\delta_0)|^2+2|\xs(0)|^2)t.
\end{equation*}
Then Gronwall's inequality gives
\begin{equation*}
    \|\mu^1_t\|_2^2\leq \Big(\mathbf{E}[|\xi|^2]+(K\mathbf{E}[|\xi|^2]+|b(0,\delta_0)|^2+2|\xs(0)|^2)t\Big)e^{(5K+1)t}.
\end{equation*}
Hence, for any $t\geq0$, there exists an increasing function $K_0: \mathbb{R}^+\to \mathbb{R}^+$ such that
$W_{2,t,\xl}(\mu^1,\mu^0)\leq K_0(t)$.
It follows from induction that
\[
\WW_{2,t,\xl}^2(\mu^{n+1},\mu^n)\leq\frac{K_0(t)(Kt)^n}{n!},
\]
so we obtain
\[
\sum_{n=0}^{\8}\WW_{2,t,\xl}(\mu^{n+1},\mu^n)<\8,\text{ for all $t\geq0$}.
\]
This means $\{\mu^n\}_{n=0}^{\8}$ is a Cauchy sequence in $C(0,\8;\PP_2(\R^d))$, and hence, it has a limit $\mutd$.
The limit $\mutd$ is exactly a fixed point of $\Upsilon_{b}$, since \eqref{eq:phi-estime} implies $\Upsilon_{b}$ is continuous.

As for the uniqueness of the fixed point, we assume $\mu,\nu$ are two fixed points of $\Upsilon_{b}$.
Then by (\ref{eq:phi-estime}), we get
\[
\WW_{2,t,\xl}^2(\mu,\nu)\leq K\int_0^t\WW_{2,s,\xl}^2(\mu,\nu)\d s,\text{ for all $t\geq0$}.
\]
Gronwall's inequality yields $\WW_{2,t,\xl}(\mu,\nu)=0$ for all $t>0$, so $\mu=\nu$.
\end{proof}

Next, we consider two McKean-Vlasov SDEs with the same diffusion term,
\begin{equation}\label{eq:mvsde-comparison-X}
\d X(t)=b(X(t),\LL(X(t)))\d t+\xs(X(t))\d W_t,\  X(0)=\xi,
\end{equation}
\begin{equation}\label{eq:mvsde-comparison-Y}
\d Y(t)= c(Y(t),\LL(Y(t)))\d t+\xs(Y(t))\d W_t,\  Y(0)=\eta.
\end{equation}
Let $X_t^{\xi}$ be the solution of \eqref{eq:mvsde-comparison-X} and $Y^{\eta}_t$ be the solution of \eqref{eq:mvsde-comparison-Y}.
The following proposition gives the comparison principle for McKean-Vlasov SDEs.

\begin{proposition}\label{thm:mvsde-comparison}
Assume that Assumption \ref{asp:lipschitz}(i) holds for $b$ and $c$, and that Assumption \ref{asp:lipschitz}(ii) and  Assumption \ref{asp:cooperation}(ii) holds for $\xs$. Also assume that Assumption \ref{asp:cooperation}(i) holds for either $b$ or $c$ and that Assumption \ref{asp:comparison} holds.
If $\xi,\,\eta$ are square integrable with $\LL(\xi)\leq\LL(\eta)$, then $\LL(X_t^{\xi})\leq\LL(Y^{\eta}_t)$ for all $t\geq0$.
\end{proposition}

\begin{proof}
In the same manner as in (\ref{eq:phi-definition}), we define  a map  
\begin{equation}
\Upsilon_{c}\colon C(0,\8;\PP_2(\R^d))\to C(0,\8;\PP_2(\R^d)),\quad\nu\mapsto\LL(Y^{\nu,\eta}),
\end{equation}
where $Y^{\nu,\eta}_t$ is the solution of \eqref{eq:ldsde-comparison-Y}.

By Lemma \ref{lem:phi-fixed-point}, we see
\[
\begin{aligned}
\Upsilon_{b}^n(\LL(\xi))(t)&\to\LL(X_t^{\xi})\\
\Upsilon_{c}^n(\LL(\eta))(t)&\to\LL(Y^{\eta}_t)
\end{aligned}
\quad\text{in $\PP_2(\R^d)$, as $n\to\8$, for $t\geq0$}.
\]
Since $\LL(\xi)\leq\LL(\eta)$, it follows from Lemma \ref{lem:ldsde-comparison} that $\Upsilon_{b}^n(\LL(\xi))(t)\leq\Upsilon_{c}^n(\LL(\eta))(t)$ for all $t\geq0$, $n\in\Z_+$.
The conclusion is given by closedness of the stochastic order (Lemma \ref{lem:stochastic-order-closed}).
\end{proof}

If $b=c$, Assumption \ref{asp:comparison} is automatically satisfied. 
Hence, Proposition \ref{prop:continuity-Pt} and Proposition \ref{thm:mvsde-comparison} directly imply the following theorem.

\begin{theorem}\label{coro:mvsde-monotone}
Under Assumption \ref{asp:lipschitz}, \ref{asp:cooperation}, the following McKean-Vlasov SDE,
\[
\d X(t)=b(X(t),\LL(X(t)))\d t+\xs(X(t))\d W_t,
\]
generates a monotone semiflow $\str P_t$ on $(\PP_2(\R^d),\WW_2)$ with respect to the stochastic order.
\end{theorem}

\begin{remark}
Comparison results for McKean–Vlasov SDEs (with memory) have been studied in the literature; see, for instance, \cite{Huang-Liu-Wang2018, panpanren22}.
In this section, we present a direct proof based on the  classical comparison theorem for usual time-inhomogeneous SDEs  (\cite[Theorem 1.2]{Geib1994}).
Moreover, we show that under the cooperative  condition, a single McKean–Vlasov SDE generates a monotone semiflow with respect to the stochastic order.
\end{remark}

\section{Connecting Orbit Theorem in Monotone Dynamical Systems}\label{Section:Results-MDS}

Dancer-Hess connecting orbit theorem is a fundamental result in the theory of monotone dynamical systems, which is crucial in analyzing the existence of connecting orbits and further equilibria (fixed points) for strictly monotone semiflows (mappings) on a compact order interval $[a,b]$ contained in a Banach space with endpoints $a,b$ as two order-related equilibria (fixed points) (see \cite{DH91,Hess91,HSW96,M84,WFM95} for related results and applications). 
To the best of our knowledge, connecting orbit theorem for semiflows is obtained for strictly monotone semiflows on a compact order interval in a Banach space (see Hess \cite[Proposition 9.1]{Hess91} and Dancer-Hess \cite[Remark 1.2]{DH91}). 
However, the Banach space structure and strictly monotone requirement restrict the application of this theorem to the monotone semiflows on $\mathcal{P}_2(\mathbb{R}^d)$ generated by the cooperative McKean-Vlasov SDEs.
In this section, we are going to extend the classical connecting orbit theorem to monotone semiflows on a convex compact subset of a Hausdorff locally convex topological vector space (see Theorem \ref{T:semiflow-tri}).

\subsection{Preliminaries of dynamical systems on ordered spaces}\label{Subsection:Pre-MDS}
We first introduce some definitions and fundamental properties of general semiflows (Definition \ref{D:semiflow}).
Let $X$ be a complete metric space and  $\Phi: 
\mathbb{R_+}\times X\rightarrow X$ be a semiflow on $X$.
For any $x\in X$, the \emph{positive orbit} $O_+(x,\Phi)$ of $x$ is $\{\Phi_t(x): t\in\mathbb{R_+}\}$. 
A \emph{negative orbit} $O_-(x,\Phi)$ of $x$ is a net $\{x_{s}\}_{s\in\mathbb{R}_-}$ in $X$ such that $x_0=x$ and $\Phi_t(x_{s})=x_{t+s}$, for any $0\leq t\leq -s$. 
An \emph{entire orbit} $O(x,\Phi)$ of $x$ is a net $\{x_s\}_{s\in\mathbb{R}}$ in $X$ such that $\Phi_t(x_{s})=x_{t+s}$ for any $t\in\mathbb{R_+}$, $s\in \mathbb{R}$.
Notice that the negative orbit and entire orbit of a point $x$ may not exist or not be unique.
A point $x\in X$ is an \emph{equilibrium} of $\Phi$, if $\Phi_t(x)=x$ for any $t\geq0$. 

A subset $B\subset X$ is said to be \emph{positively invariant}, if $\Phi_t(B)\subset B$, for any $t\in\mathbb{R_+}$, and  \emph{invariant}, if $\Phi_t(B)=B$, for any $t\in\mathbb{R_+}$. 
It is not difficult to see that $B\subset X$ invariant is equivalent to, for any $x\in B$, there exists an entire orbit $\{x_t\}_{t\in\mathbb{R}}\subset B$ such that $x_0=x$.
For any set $B\subset X$, the \emph{$\omega$-limit set} of $B$ is defined by
$$\omega(B,\Phi):=\{y\in X:\text{ there exist sequences $t_k\to\infty$ and $y_k\in B$ such that }\mathop{\lim}\limits_{k \to\infty}\Phi_{t_k}(y_k)=y\}.$$
Now, we give a basic property of $\omega$-limit sets which is easy to check by definition, one may also see Hale \cite[Lemma 3.2.1]{Ha88}.

\begin{lemma}\label{L:omega-limit-set}
If $B$ is a  positively invariant compact subset of $X$, then $\omega(B,\Phi)$ is a nonempty invariant compact subset of $B$.
\end{lemma}

Next, we give some corresponding definitions for discrete-time dynamical systems, or named mappings.
For a continuous mapping $\Psi:X\to X$ and any $x\in X$, the \emph{positive orbit} $O_+(x,\Psi)$ of $x$ is the set $\{\Psi^n(x): n\in\mathbb{Z_+}\}$. A \emph{negative orbit} $O_-(x,\Psi)$ of $x$ is a sequence $\{x_{n}\}_{n\in\mathbb{Z_-}}$ in $X$ such that $x_0=x$ and $\Psi(x_n)=x_{n+1}$, for any $n\leq-1$. An \emph{entire orbit} $O(x,\Psi)$ of $x$  is a sequence $\{x_n\}_{n\in\mathbb{Z}}$ in $X$ such that $\Psi(x_n)=x_{n+1}$ for any $n\in\mathbb{Z}$. 
If $\Psi(x)=x$, we say $x$ is a \emph{fixed point} of $\Psi$.

Now, we further introduce some additional fundamental definitions and properties of ordered spaces.
Let $X$ be a complete metric space with a closed partial order relation ``$\leq$''. 
A \emph{totally ordered set} (which is also referred as \emph{chain}) $A$ in $X$ means that, any two points $x$ and $y$ in $A$  are order-related. 
A net $\{x_t\}_{t\in\T}\subset S$ ($\T=\mathbb{R},\mathbb{R}_+,\mathbb{R}_-,\mathbb{Z},\mathbb{Z}_+ \textnormal{or } \mathbb{Z}_-$) is called \emph{increasing} (\emph{decreasing}), if $x_t\leq x_{s}$ ($x_t\geq x_{s}$) for any $t,s\in\T$ with $t\leq s$.
Given $A\subset X$, $x\in A$ is said to be \emph{maximal} (\emph{minimal}) in $A$, if there is no point $y\in A$ such that $y>x$ ($y< x$). 
For a point $x\in X$ and a set $A\subset X$, we denote $x\leq A$ ($x\geq A$), if $x\leq y$ ($x\geq y$) for any $y\in A$. 
A point $x\in X$ is an \emph{upper bound} (\emph{lower bound}) of $A$ if $x\geq A$ ($x\leq A$).
An upper bound $x_0$ of $A$ is said to be the \emph{supremum} of $A$, denoted by $x_0=\sup A$, if any other upper bound $x$ satisfies $x\geq x_0$. 
Similarly, a lower bound $x_0$ of $A$ is said to be the \emph{infimum} of $A$, denoted by $x_0=\inf A$, if any other lower bound $x$ satisfies $x\leq x_0$. 
Clearly, if the supremum (infimum) of $A$ exists, it must be unique, but may not belongs to $A$.

\begin{lemma}\label{L:MCT}
Assume that $A\subset X$ is compact. Then,
\begin{enumerate}[label=\textnormal{(\roman*)}]
    \item every increasing (decreasing) net $\{x_t\}_{t\in\T}$ in $A$ converges as $t\to\infty$ and $t\to-\infty$;
    \item if $A$ is totally ordered, then $\sup A$ ($\inf A$) exists and $\sup A\in A$ ($\inf A\in A$);
    \item the set $A$ contains a maximal element and a minimal element;
    \item if a positive (negative) orbit $O_+(x,\Phi)$ ($O_-(x,\Phi)$) of a semiflow $\Phi$ on $X$ contained in $A$ is increasing or decreasing, then it must converge to an equilibrium as $t\to\infty$ ($t\to-\infty$).
    The corresponding result is also true for mappings.
\end{enumerate}
\end{lemma}

\begin{proof}
For the proof of items (i)-(iii), we refer to Hirsch-Smith \cite[Lemma 1.1]{HS05}. Here, we give the proof of (iv). We only prove the case for increasing negative orbit of a semiflow $\Phi$, as other cases are similar. Suppose $\{x_{s}\}_{s\in\mathbb{R}_-}$ is an increasing negative orbit, that is, $\Phi_t(x_{s})=x_{t+s}$, for any $0\leq t\leq -s$,  and $x_{s_1}\leq x_{s_2}$ for any $s_1\leq s_2\leq0$. By (i), we have $x_{s}$ converges to some point $x\in A$  as $s\to-\infty$. Then for any fixed $t\geq0$, $\Phi_t(x)=\Phi_t(\mathop{\lim}\limits_{s \to-\infty}x_{s})=\mathop{\lim}\limits_{s\to-\infty}\Phi_t(x_{s})=\mathop{\lim}\limits_{s\to-\infty}x_{t+s}=x$. Therefore, $x$ is an equilibrium.
\end{proof}

\subsection{Statement and proof of connecting orbit theorem}
Hereafter in this section, we fix the following settings.
\begin{enumerate}[label=(M\arabic*)]
\item $(V,\mathcal{T})$ is a Hausdorff locally convex topological vector space with a cone $C\subset V$, which induces a closed partial order relation $\leq$ on $V$ (as we introduced in Section \ref{Subsection:Properties-Order});
\item $(S,d)$ is a non-singleton convex compact metric subspace of $(V,\mathcal{T})$, where $d$ is a metric on $S$ inducing the relative topology on $S$;
\item $p:=\inf S$, $q:=\sup S$ exist and belong to  $S$;
\item $\Phi$ is a monotone semiflow on $S$;
\item $p$ and $q$ are equilibria of $\Phi$.
\end{enumerate}
In particular, (M1)(M2) imply $S$ is a complete metric space with closed partial order relation $\leq$.

In order to prove the connecting orbit theorem for monotone semiflows (Theorem \ref{T:semiflow-tri}), firstly we need a connecting orbit theorem for monotone mappings (see Lemma \ref{T:map-tri}). 
By invoking the fixed point index lemma for metrizable convex compact subsets of Hausdorff locally convex topological vector spaces (Lemma \ref{L:fixed pt index}), the proof of Lemma \ref{T:map-tri} proceeds exactly as in Dancer–Hess \cite[Proposition 1]{DH91} of Banach space setting. 
We give the detail in Appendix for the sake of completeness. 

\begin{lemma}\label{T:map-tri}
Assume that (M1)-(M3) hold. Let $\Psi:S\to S$ be a continuous monotone map.
If $p, q$ are two fixed points of $\Psi$, then at least one of the following holds:
\begin{enumerate}[label=\textnormal{(\alph*)}]
\item $\Psi$ has a fixed point distinct from $p, q$ in $S$;  
\item there exists an increasing entire orbit  $\{x_n\}_{n\in\mathbb{Z}}$ from $p$ to  $q$, i.e., $\Psi(x_n)=x_{n+1}$, $x_n\leq x_{n+1}$, for all $n\in\mathbb{Z}$, $x_n\to q$, as $n\to\infty$ and $x_n\to p$,
 as $n\to-\infty$; 
\item there exists a decreasing entire orbit  $\{x_n\}_{n\in\mathbb{Z}}$ from $q$ to $p$, i.e., $\Psi(x_n)=x_{n+1}$, $x_n\geq x_{n+1}$, for all $n\in\mathbb{Z}$, $x_n\to p$, as $n\to\infty$ and $x_n\to q$,
 as $n\to-\infty$.
\end{enumerate}
\end{lemma}

With the help of the consequence for mappings in Lemma \ref{T:map-tri}, we obtain our connecting orbit result for semiflows.
\begin{theorem}\label{T:semiflow-tri}
Assume that (M1)-(M5) hold. Then at least one of the following holds:
\begin{enumerate}[label=\textnormal{(\alph*)}]
    \item $\Phi$ has an equilibrium distinct from $p, q$ in $S$;
    \item there exists an increasing entire orbit  $\{x_t\}_{t\in\mathbb{R}}$ from $p$ to $q$, i.e., $\Phi_t(x_{s})=x_{t+s}$ for any $t\geq 0$, $s\in \mathbb{R}$, and $x_s\leq x_t$, for all $s\leq t$, and $x_t\to q$, as $t\to\infty$, and $x_t\to p$,
 as $t\to-\infty$;
 \item there exists a decreasing entire orbit  $\{x_t\}_{t\in\mathbb{R}}$ from $q$ to $p$, i.e., $\Phi_t(x_{s})=x_{t+s}$ for any $t\geq 0$, $s\in \mathbb{R}$, and $x_s\geq x_t$, for all $s\leq t$, and $x_t\to p$, as $t\to\infty$, and $x_t\to q$,
 as $t\to-\infty$.
\end{enumerate}
\end{theorem}

\begin{remark}
Compared to classical connecting orbit theorem for semiflows (see Hess \cite[Proposition 9.1]{Hess91} and Dancer-Hess \cite[Remark 1.2]{DH91}), the requirements on the additional Banach spaces structure and the strict monotonicity of systems are relaxed. 
They will play a crucial role in the study of the dynamics of cooperative McKean-Vlasov SDEs.
\end{remark}

\begin{proof}[Proof of Theorem \ref{T:semiflow-tri}]
Hereafter in the proof, for $x\in S$ and $\epsilon>0$, denote $B_S(x,\epsilon)=\{y\in S:d(x,y)<\epsilon\}$. 
For any subset $A\subset S$,  denote by  $\overline A$ and $\partial A$ the closure and boundary of $A$ relative to the topology on $S$, respectively.
For $x,y\in S$ with $x\leq y$, denote $[x,y]_S:=\{z\in S:x\leq z\leq y\}$ and $[x,y]_V:=\{z\in V:x\leq z\leq y\}$. Clearly, $[x,y]_S=[x,y]_V\cap S$.
Since $[x,y]_V$ is a convex closed subset of $V$ and $(S,d)$ is a convex compact metric subspace of $(V,\mathcal{T})$, one has $[x,y]_S$ is also a convex compact  metric subspace of $(V,\mathcal{T})$.

Assume that there is no further equilibrium of $\Phi$ distinct from $p, q$ in $S$. We are going to prove (b) or (c) holds. 

Define 
$$S_l:=\{x\in S:\Phi_t(x)\geq x,\text{ for any }t\geq0\}.$$
Then $S_l$ is compact, since $S_l\subset S$ is closed and $S$ is compact. Besides, the monotonicity of $\Phi$ entails that $S_l$ is positively invariant. Similarly, we can define 
$$S_u:=\{x\in S:\Phi_t(x)\leq x,\text{ for any }t\geq0\},$$
which is also a positively invariant compact subset of $S$. 
Lemma \ref{L:omega-limit-set} entails that the $\omega$-limit set  
$\omega(S_l,\Phi)$ (resp. $\omega(S_u,\Phi)$) is a nonempty invariant compact subset of $S_l$ (resp. $S_u$).
We claim that,
\begin{equation}\label{E:SlSu not empty}
   \text{either}\qquad \omega(S_l,\Phi)\backslash\{p,q\}\neq\emptyset \qquad\text{ or else,}\qquad \omega(S_u,\Phi)\backslash\{p,q\}\neq\emptyset.
\end{equation}
Before we prove \eqref{E:SlSu not empty}, we show that how it implies (b) or (c) holds in Theorem \ref{T:semiflow-tri}. In fact, if $\omega(S_l,\Phi)\backslash\{p,q\}\neq\emptyset$, take $y\in\omega(S_l,\Phi)\backslash\{p,q\}$. Since $\omega(S_l,\Phi)$ is invariant, there exists an entire orbit $\{y_t\}_{t\in\mathbb{R}}\subset\omega(S_l,\Phi)$ such that $y_0=y$. The fact that $\omega(S_l,\Phi)\subset S_l$ implies that, $y_s\leq y_t$ for all $s\leq t$. By virtue of Lemma \ref{L:MCT} (iv), $y_t$ converges to an equilibrium, as $t\to\infty$ and $t\to-\infty$ respectively. Since there is no further equilibrium of $\Phi$ distinct from $p, q$ in $S$, one has $y_t\to q$, as $t\to\infty$ and $y_t\to p$, as $t\to-\infty$. Thus, we have obtained (b) in Theorem  \ref{T:semiflow-tri}. Similarly, $\omega(S_u,\Phi)\backslash\{p,q\}\neq\emptyset$ implies (c) in Theorem  \ref{T:semiflow-tri}.

Now, we focus on the proof of \eqref{E:SlSu not empty}. Firstly, we give the following claims.

\noindent\textbf{Claim 1.} If $x_n\in S$, $x_n\to x$ as $n\to\infty$, $t_1>t_2>t_3>\cdots\to0$, and $\Phi_{t_n}x_n\geq (\text{resp. }\leq,=)x_n$ for any $n\geq1$, then $\Phi_t(x)\geq (\text{resp. }\leq,=) x$ for any $t\geq0$.

\noindent\textbf{Proof of Claim 1:}
We only prove the case $\geq$, as other cases are similar. In fact, for any fixed $t>0$, we have $t=k_{n,t}t_n+\tau_{n,t}$, $k_{n,t}\in\mathbb{N}$ and $\tau_{n,t}\in[0,t_n)$. Then,
\begin{equation*}
d(\Phi_t(x),\Phi_{k_{n,t}t_n}(x_n))\leq d(\Phi_t(x),\Phi_{\tau_{n,t}}(\Phi_{k_{n,t}t_n}(x_n)))+d(\Phi_{\tau_{n,t}}(\Phi_{k_{n,t}t_n}(x_n)),\Phi_{k_{n,t}t_n}(x_n))
\end{equation*}
By virtue of the continuity of $\Phi_t$, we have $d(\Phi_t(x),\Phi_{\tau_{n,t}}(\Phi_{k_{n,t}t_n}(x_n)))=d(\Phi_t(x),\Phi_t(x_n))\to0$, as $n\to\infty$.
Since $\tau_{n,t}\to0$, as $n\to\infty$, by the  joint continuity of the semiflow and the compactness of $S$, we have 
$d(\Phi_{\tau_{n,t}}(\Phi_{k_{n,t}t_n}(x_n)),\Phi_{k_{n,t}t_n}(x_n))\to0$, as $n\to\infty$.
Hence, $d(\Phi_t(x),\Phi_{k_{n,t}t_n}(x_n))\to0$, as $n\to\infty$.
By $\Phi_{t_n}x_n\geq x_n$ and $\Phi$ is monotone, one has $\Phi_{k_{n,t}t_n}x_n\geq x_n$. 
Then, the closedness of the partial order $\leq$,
(i.e., the cone $C$ which induces the partial order is closed), 
entails that $\Phi_t(x)\geq x$.
Thus, we have proved Claim 1.

\noindent\textbf{Claim 2.} For any $\epsilon>0$, there exists $\delta>0$ such that, if $\Phi_t(x)=x$ for some $t\in(0,\delta)$ and $x\in S$, then $x\in B_S(p,\epsilon)\cup  B_S(q,\epsilon)$.

\noindent\textbf{Proof of Claim 2:} Otherwise, there exists $\epsilon>0$, $t_1>t_2>t_3>\cdots\to0$, $x_n\in S$ such that $\Phi_{t_n}x_n=x_n$ and $x_n\notin B_S(p,\epsilon)\cup  B_S(q,\epsilon)$. Since $S$ is compact, we assume $x_{n_k}$ converges to some point $x\in S$ and $x\notin B_S(p,\epsilon)\cup  B_S(q,\epsilon).$ By virtue of Claim 1, we have $\Phi_t(x)=x$ for any $t\geq0$. That is, $x$ is an equilibrium of $\Phi$ distinct from $p, q$ in $S$, a contradiction. The proof of Claim 2 is completed.

 Now, take an integer $N_0\geq 1$ such that 
\begin{equation}\label{E:N0}
    B_S(p,\frac{1}{N_0})\cap B_S(q,\frac{1}{N_0})=\emptyset.
\end{equation}
 For any $n\geq N_0+1$, by Claim 2, there exists $t_n>0$ small enough such that, if $\Phi_{t_n}(x)=x$ for some $x\in S$, then 
\begin{equation}\label{E:1/n control}
    x\in B_S(p,\frac{1}{n})\cup B_S(q,\frac{1}{n}).
\end{equation}
 We can choose $t_n$ such that 
 \begin{equation}\label{E:tn to 0}
     1>t_{N_0+1}>t_{N_0+2}>t_{N_0+3}>\cdots\to0.
 \end{equation}
 
Now, denote the set of the fixed points of the mapping $\Phi_{t_n}$ by
$$E_{t_n}:=\{x\in S: \Phi_{t_n}(x)=x\}.$$
Since $E_{t_n}$ is closed and $S$ is compact, we have $E_{t_n}\cap\overline{B_S(p,\frac{1}{n})}$ is compact. Obviously, $p\in E_{t_n}\cap\overline{B_S(p,\frac{1}{n})}$, so it is not empty. Therefore, Lemma \ref{L:MCT} (iii) entails that, $E_{t_n}\cap\overline{B_S(p,\frac{1}{n})}$ 
  contains a maximal element $l_n$. Moreover, since $E_{t_n}\cap\overline{B_S(q,\frac{1}{n})}\cap[l_n,q]_S$ is compact, Lemma \ref{L:MCT} (iii) again implies that, it 
  contains a minimal element $u_n$. Since \eqref{E:N0} and $n\geq N_0+1$, we have $B_S(p,\frac{1}{n})\cap B_S(q,\frac{1}{n})=\emptyset$. Thus $l_n\neq u_n$. By the way $u_n$ is taken, we have $l_n\leq u_n$. So, $l_n<u_n$. Moreover, we claim that
  \begin{equation}\label{E:no fixed points in lnun}
      E_{t_n}\cap[l_n,u_n]_S=\{l_n,u_n\}.
  \end{equation}
  In fact, suppose on the contrary that there exists $x\in E_{t_n}\cap[l_n,u_n]_S$ and $x\neq l_n, u_n$. By \eqref{E:1/n control}, one has $x\in B_S(p,\frac{1}{n})$ or $x\in B_S(q,\frac{1}{n})$. If $x\in B_S(p,\frac{1}{n})$, then $x>l_n$ contradicts the way $l_n$ is taken. Otherwise, if $x\in B_S(q,\frac{1}{n})$, then $x<u_n$ contradicts the way $u_n$ is taken. Hence, we proved \eqref{E:no fixed points in lnun}.

  By \eqref{E:no fixed points in lnun}, we can apply Lemma \ref{T:map-tri} to the monotone map  $\Phi_{t_n}$ on  the convex compact subspace $[l_n,u_n]_S$, and obtain that (b) or (c) holds in Lemma \ref{T:map-tri}.  Without loss of generality, we assume that, for any $n\geq N_0+1$, (b) occurs in Lemma \ref{T:map-tri} for $\Phi_{t_n}$. That is, there exists an increasing entire orbit $\{x^j_n\}_{j\in\mathbb{Z}}$ from $l_n$ to $u_n$ for any $n\geq N_0+1$, i.e.,   $\Phi_{t_n}(x^j_n)=x^{j+1}_{n}$, $x^j_n\leq x^{j+1}_{n}$, for all $j\in\mathbb{Z}$, and $x^j_n\to u_n$, as $j\to\infty$ and $x^j_n\to l_n$,
 as $j\to-\infty$. We are going to prove $\omega(S_l,\Phi)\backslash\{p,q\}\neq\emptyset$ in \eqref{E:SlSu not empty} occurs (Otherwise, if there exists a subsequence $\{n_i\}_{i\geq 1}$ of $\{n\}_{n\geq N_0}$ such that (c) occurs in Lemma \ref{T:map-tri} for each $\Phi_{t_{n_i}}$ on $[l_{n_i},u_{n_i}]_S$, the proof of $\omega(S_u,\Phi)\backslash\{p,q\}\neq\emptyset$ in \eqref{E:SlSu not empty} is similar).

Next, we give the following claim.

\noindent\textbf{Claim 3.} There exists  a sequence $z_i\in S_l\backslash\{p,q\}$ such that $z_i\to p$ as $i\to\infty$.

\noindent\textbf{Proof of Claim 3:} By continuity of semiflow $\Phi$ and $p$ being an equilibrium of $\Phi$, there exists $\delta_i>0$ for $i\geq0$ such that
 \begin{equation}\label{E:delta i to 0}
     \frac{1}{N_0+1}>\delta_0>\delta_1>\delta_2>\cdots\to0,
 \end{equation}
 and
\begin{equation}\label{E:cts-control-semiflow}
  \Phi_t(B_S(p,\delta_i))\subset B_S(p,\delta_{i-1}),  \text{ for any   $t\in[0,1]$, $i\geq 1$}.
\end{equation}
Now, fix $i\geq 1$. Together with the fact that $\{x^j_n\}_{j\in\mathbb{Z}}$ is an increasing entire orbit from $l_n\in\overline{B_S(p,\frac{1}{n})}$ to $u_n\in\overline{B_S(q,\frac{1}{n})}$, \eqref{E:cts-control-semiflow} implies that, for any $n\geq N_0+1$ with $\frac{1}{n}<\delta_i$, there exists $j_{n,i}\in\mathbb{Z}$ such that $x^{j_{n,i}}_{n}\in B_S(p,\delta_{i-1})\backslash B_S(p,\delta_i)$. Since $S$ is compact, we can assume that $x^{j_{n_k,i}}_{n_k}$ converges to some point $z_i$ as $k\to\infty$. Clearly, $z_i\in \overline{B_S(p,\delta_{i-1})}\backslash B_S(p,\delta_i)$. By virtue of $\Phi_{t_{n_k}}(x^{j_{n_k,i}}_{n_k})\geq x^{j_{n_k,i}}_{n_k}$ and \eqref{E:tn to 0}, Claim 1 implies that $z_i\in S_l$. Hence, we proved Claim 3.

Finally, by Claim 3, we are going to prove $\omega(S_l,\Phi)\backslash\{p,q\}\neq\emptyset$, i.e.,  \eqref{E:SlSu not empty}. In fact, by the definition of $S_l$ and $z_i\in S_l\backslash\{p,q\}$, Lemma \ref{L:MCT} (iv) implies that $\Phi_t(z_i)$ converges to some equilibrium of $\Phi$ as $t\to\8$. Since there is no further equilibrium of $\Phi$ except $\{p,q\}$, one has for any $i\geq1$,  $\Phi_t(z_i)\to q$ as $t\to\infty$. On the other hand, \eqref{E:N0} entails that $q\notin B_S(p,\frac{1}{N_0})$. By \eqref{E:delta i to 0}, $z_i\in\overline{B_S(p,\delta_{i-1})}\subset B_S(p,\frac{1}{N_0+1})$, for any $i\geq1$. Since $d(\Phi_t(z_i),p)\to d(q,p)>\frac{1}{N_0+1}$ as $t\to\8$, by the continuity of $d(\Phi_t(z_i),p)$ with respect to $t$, for any $i\geq 1$, there exists $t_i>0$ such that 
\begin{equation}\label{E:1/N0+1}
    d(\Phi_{t_i}(z_i),p)=\frac{1}{N_0+1}.
\end{equation}
Since $S$ is compact, one can assume $\Phi_{t_{i_k}}(z_{i_k})\to v$, as $k\to\infty$.
By the continuity of the semiflow $\Phi$ on the compact space $S$ and the fact that $p$ is an equilibrium of $\Phi$, one has for any fixed time $T>0$, there exists $\delta_T>0$ such that, $\Phi_t(B_S(p,\delta_T))\subset B_S(p,\frac{1}{2(N_0+1)})$ for any $t\in[0,T]$. Together with the fact that $z_i\to p$ as $i\to\infty$ and \eqref{E:1/N0+1}, we have $t_i\to\infty$. Therefore, by definition of $\omega(S_l,\Phi)$, one has $v\in\omega(S_l,\Phi)$. In addition, $d(v,p)=\frac{1}{N_0+1}$ and $q\notin B_S(p,\frac{1}{N_0})$ give $v\neq p,q$. Hence, we proved $\omega(S_l,\Phi)\backslash\{p,q\}\neq\emptyset$, which completes our proof.
\end{proof}

\section{Order-Related Invariant Measures with Shrinking Neighbourhoods}\label{Section:Order-Related-Shrinking}

In this section, our attention returns to the McKean-Vlasov SDEs,
\begin{equation}\label{MVE}
\d X_t=b(X_t,\LL(X_t))\d t+\xs(X_t)\d W_t.
\end{equation}
By establishing the continuity, compactness, and monotonicity of a measure-iterating mapping (Proposition \ref{P:Psi-Cts-Cpt-Monotone}), and applying the comparison theorem, we prove the existence of order-related invariant measures  under  locally dissipative conditions (Theorem \ref{Thm: two ordered im}).
Furthermore, in  Theorem \ref{Thm: two ordered im} we show that these invariant measures possess shrinking neighbourhoods under the law evoluation semigroup $\str P$, which is essential for determining the direction of the connecting orbit in the proof of our main theorem.

\subsection{Properties of measure-iterating map}

For any fixed $\mu\in \mathcal{P}_2(\mathbb{R}^d)$, we consider the following SDE:
\begin{equation}\label{SDE 1}
\d X_t=b(X_t,\mu)\d t+\xs(X_t)\d W_t, \ \ \ t\geq 0.
\end{equation}
Under Assumption \ref{asp:lipschitz},  for any initial condition $x\in \mathbb{R}^d$, the unique solution $X_t^{\mu,x}$ of \eqref{SDE 1}  generates a Markov transition kernel and Markov semigroup by the following way:
\begin{equation*}
    P_t^{\mu}(x,\xG):=\mathbf{P}\{X_t^{\mu,x}\in \xG\}, \ \text{ for any } \xG\in \mathcal{B}(\mathbb{R}^d),
\end{equation*}
\begin{equation*}
    P_t^{\mu}f(x):=\int_{\mathbb{R}^d}f(y)P_t^{\mu}(x,\d y), \ \text{ for any } f\in B_b(\mathbb{R}^d),
\end{equation*}
where $B_b(\R^d)$ denotes the collections of  all bounded Borel-measurable functions from $\R^d$ to $\R$.
Also, we have the dual of $P_t^{\mu}$
\begin{equation*}
    P_t^{\mu,*}\nu:=\int_{\mathbb{R}^d}P_t^{\mu}(x,\cdot)\nu(\d x), \ \text{ for any } \nu\in \mathcal{P}(\mathbb{R}^d).
\end{equation*}
Under Assumption \ref{asp:lipschitz}, \ref{asp:dissipative-growth-nondegeneracy}, it is well-known that the SDE \eqref{SDE 1} has a unique invariant measure $\rho^{\mu}\in \PP_2(\R^d)$ (see e.g.,  Meyn-Tweedie \cite{Meyn-Tweedie1993-book,Meyn-Tweedie1993} or  Feng-Qu-Zhao \cite[Theorem 4.8]{Feng-Qu-Zhao2023}). Now define a measure-iterating map $\Psi: \PP_2(\R^d)\to \PP_2(\R^d)$ by
\begin{equation}\label{def of Psi}
    \Psi(\mu):=\rho^{\mu} \ \text{(the unique invariant measure of  \eqref{SDE 1})}.
\end{equation}
\noindent Notice that $\mu\in \PP_{2}(\R^d)$ is an invariant measure of \eqref{MVE} if and only if $\mu$ is a fixed point of $\Psi$. 
This fixed-point approach to studying invariant measures has been considered in the literature; see, for example,   \cite{Ahmed-Ding1993,Bao2022,Bao-Wang2025,Dawson1983,Zhang2023}. 
Here, to investigate the existence of multiple order-related invariant measures, we first study the following properties of  $\Psi$.

\begin{proposition}\label{P:Psi-Cts-Cpt-Monotone}
    Under Assumption \ref{asp:lipschitz}, \ref{asp:dissipative-growth-nondegeneracy}, the map $\Psi: \PP_2(\R^d)\to \PP_2(\R^d)$ is continuous and compact.
    Furthermore, if Assumption \ref{asp:cooperation} is also satisfied, then $\Psi$ is monotone, i.e., if $\mu_1,\mu_2\in \PP_2(\R^d)$ with $\mu_1\leq \mu_2$, we have $\Psi(\mu_1)\leq \Psi(\mu_2)$.
\end{proposition}

Before proceeding with the proof of Proposition \ref{P:Psi-Cts-Cpt-Monotone}, we need some necessary lemmas. Firstly, we have the following lemma from Hairer-Mattingly \cite[Theorem 1.3]{Hairer-Mattingly2011} or \cite[Lemma 2.2]{Feng-Qu-Zhao2023}.

 \begin{lemma}\label{Lem: HM}
     Assume that  $P(x,\xG)$, $x\in\R^d$, $\xG\in \mathcal{B}(\R^d)$ is a discrete Markovian transition kernel. If there exist a function $V: \R^d\to [0,\infty)$, a probability measure $\nu\in \mathcal{P}(\R^d)$ and non-negative constants $\tilde{\gamma}, \tilde{K}, \eta, R$ such that
	\begin{equation*}
		PV(x)\leq \tilde{\gamma} V(x)+\tilde{K} \ \text{ for all } x\in \R^d, \ \text{ and } \ \inf_{\{x:V(x)\leq R\}}P(x,\xG)\geq \eta\nu(\xG) \ \text{ for all } \xG\in\BB(\R^d),
	\end{equation*}
	then for any $\tilde{\beta}\geq 0$ and $\mu_1, \mu_2\in \mathcal{P}(\R^d)$, we have
	\begin{equation*}
		\rho_{\tilde{\beta},V}(P^*\mu_1, P^*\mu_2)\leq \zeta \rho_{\tilde{\beta},V}(\mu_1, \mu_2),
	\end{equation*}
	where
	\begin{equation*}
		\zeta=\max \bigg\{1-\eta+\tilde{\beta} \tilde{K}, \ \frac{2+\tilde{\beta}(\tilde{\gamma} R+2\tilde{K})}{2+\tilde{\beta} R}\bigg\},
	\end{equation*}
    and
    \begin{equation*}
		\rho_{\tilde{\beta},V}(\mu_1, \mu_2):=\sup_{\{f: |f|\leq 1+\tilde{\beta} V\}}\bigg|\int_{\R^d}f\d \mu_1-\int_{\R^d}f\d \mu_2\bigg|.
	\end{equation*} 
 \end{lemma}

Before stating the next lemma, we introduce a new metric on $\PP_p(\R^d)$. 
For any $p\geq 1$ and $\mu,\nu\in \mathcal{P}_p(\mathbb{R}^d)$, we give the following weighted total variation
\begin{equation}\label{0322-1}
    d_{p}(\mu, \nu):=\sup_{\{f: |f(x)|\leq 1+|x|^p\}}\bigg\{\bigg|\int_{\mathbb{R}^d}f \d \mu-\int_{\mathbb{R}^d}f \d \nu\bigg|\bigg\}=\int_{\mathbb{R}^d}(1+|x|^p)\d |\mu-\nu|(x).
\end{equation}
It is easy to check that $(\PP_p(\R^d),d_p)$ is a complete metric space. 

Inspired by \cite[Theorem 4.8]{Feng-Qu-Zhao2023}, we have the following lemma.
\begin{lemma}\label{Thm: esti of invari measure}
    Under Assumption \ref{asp:lipschitz}, \ref{asp:dissipative-growth-nondegeneracy}, for any fixed $\mu\in \mathcal{P}_2(\mathbb{R}^d)$, we have
    \begin{enumerate}[label=\textnormal{(\roman*)}]
        \item for any $p\geq 2$, there exist $C, \lambda>0$ depending only on $(K,\alpha,\beta,\gamma,\underline{\xs},\overline{\xs},\kappa,d,p)$ and $\|\mu\|_2$ such that 
        \begin{equation}\label{converge to im}
           d_{p}(P_t^{\mu}(x,\cdot), \Psi(\mu))\leq C(1+|x|^p)e^{-\lambda t},\ \ x\in\R^d, \ t\geq0;
        \end{equation} 
        \item $\Psi(\mu)$ has the following estimates:
        \begin{equation}\label{uniform bounds of im}
        \|\Psi(\mu)\|_p^p\leq 
            \begin{cases}
                \frac{\beta}{\alpha}\|\mu\|_2^2+\frac{2\gamma+\overline{\xs}d}{2\alpha}, & p=2,\\
                \frac{2}{p\alpha}\Big(\frac{p-2}{p\alpha}\Big)^{\frac{p-2}{2}}\big(2\beta\|\mu\|_2^2+2\gamma+\overline{\xs}(d+p-2)\big)^{\frac{p}{2}}, & p>2.
            \end{cases}
        \end{equation}
    \end{enumerate}
\end{lemma}
\begin{proof}
    (i). Let $V_p(x):=|x|^p$, $p\geq2$. Firstly we consider the case $p=2$. Applying It\^o's formula to $|X_t^{\mu,x}|^2$, we have
    \begin{equation}\label{est im 1}
        \begin{split}
            \d |X_t^{\mu,x}|^2&=\big(2\inp{X_t^{\mu,x}}{b(X_t^{\mu,x},\mu)}+\abs{\xs(X_t^{\mu,x})}^2\big)\d t+2\inp{X_t^{\mu,x}}{\xs(X_t^{\mu,x})\d W_t}\\
            &\leq \big(-2\alpha |X_t^{\mu,x}|^2+2\beta \|\mu\|_2^2+2\gamma+\overline{\xs}d\big)\d t+2\inp{X_t^{\mu,x}}{\xs(X_t^{\mu,x})\d W_t}.
        \end{split}
    \end{equation}
    Then we have
    \begin{equation*}
        \begin{split}
            \d e^{2\alpha t}|X_t^{\mu,x}|^2
            \leq e^{2\alpha t}\big(2\beta \|\mu\|_2^2+2\gamma+\overline{\xs}d\big)\d t+2e^{2\alpha t}\inp{X_t^{\mu,x}}{\xs(X_t^{\mu,x})\d W_t}.
        \end{split}
    \end{equation*}
    Hence
    \begin{equation}\label{ineq p=2}
        \begin{split}
            P_t^{\mu}V_2(x)=\mathbf{E}[|X_t^{\mu,x}|^2]\leq e^{-2\alpha t}V_2(x)+\frac{\beta}{\alpha}\|\mu\|_2^2+\frac{2\gamma+\overline{\xs}d}{2\alpha}.
        \end{split}
    \end{equation}
    In the case of $p>2$, It\^o's formula to $|X_t^{\mu,x}|^p$ together with \eqref{est im 1} gives
    \begin{equation}\label{Ito for p-th}
        \begin{split}
            \d |X_t^{\mu,x}|^p&=\frac{p}{2}|X_t^{\mu,x}|^{p-2}\d |X_t^{\mu,x}|^2+\frac{1}{2}\frac{p}{2}\Big(\frac{p}{2}-1\Big)|X_t^{\mu,x}|^{p-4}\d \inp{|X^{\mu,x}|^2}{|X^{\mu,x}|^2}_t\\
            &\leq -p\alpha |X_t^{\mu,x}|^p\d t+\Big(p\beta \|\mu\|_2^2+p\gamma+\frac{p\overline{\xs}(d+p-2)}{2}\Big)|X_t^{\mu,x}|^{p-2}\d t\\
            &\ \ \ \ +p|X_t^{\mu,x}|^{p-2}\inp{X_t^{\mu,x}}{\xs(X_t^{\mu,x})\d W_t}\\
            &\leq -\frac{p\alpha}{2}|X_t^{\mu,x}|^p\d t+\Big(\frac{p-2}{p\alpha}\Big)^{\frac{p-2}{2}}\big(2\beta\|\mu\|_2^2+2\gamma+\overline{\xs}(d+p-2)\big)^{\frac{p}{2}}\d t\\
            &\ \ \ \ +p|X_t^{\mu,x}|^{p-2}\inp{X_t^{\mu,x}}{\xs(X_t^{\mu,x})\d W_t},
        \end{split}
    \end{equation}
    where in the last inequality, we have used the following Young's inequality
    \[
    \Big(p\beta \|\mu\|_2^2+p\gamma+\frac{p\overline{\xs}(d+p-2)}{2}\Big)|X_t^{\mu,x}|^{p-2}\leq \frac{p\alpha}{2}|X_t^{\mu,x}|^p+\Big(\frac{2}{p}\cdot\frac{p-2}{p\alpha}\Big)^{\frac{p-2}{2}}\big(2\beta\|\mu\|_2^2+2\gamma+\overline{\xs}(d+p-2)\big)^{\frac{p}{2}}.
    \]
    Similarly, we have
    \begin{equation}\label{ineq p>2}
        \begin{split}
            P_t^{\mu}V_p(x)=\mathbf{E}[|X_t^{\mu,x}|^p]\leq e^{-\frac{p\alpha}{2} t}V_p(x)+\frac{2}{p\alpha}\Big(\frac{p-2}{p\alpha}\Big)^{\frac{p-2}{2}}\big(2\beta\|\mu\|_2^2+2\gamma+\overline{\xs}(d+p-2)\big)^{\frac{p}{2}}.
        \end{split}
    \end{equation}

    On the other hand, by \cite[Theorem 3.10]{Feng-Qu-Zhao2023}, the density $p^{\mu}_t(x,y)$ of the transition kernel $P^{\mu}_t(x,\cdot)$ has the following lower bound estimation: there exist $\eta_1,\eta_2,\eta_3>0$ depending only on $(K,\alpha,\beta,\gamma,\underline{\xs},\overline{\xs},\kappa,d)$ and $\norm{\mu}_2$, such that for all $0<t\leq1$,
    \begin{equation}\label{E:lower bound of transition kernel}
        p^{\mu}_t(x,y)\geq \eta_1 t^{-\frac{d}{2}}\exp\{-\eta_2(1+|x|^{2(d+1)\kappa})(1+|x-y|^{2\kappa})-\eta_3t^{-1}(1+|x-y|^2)\}.
    \end{equation}
    Hence for any $M>0$ and any $0<t\leq1$,
    \begin{equation*}
        \begin{split}
            \inf_{\{(x,y): |x|\leq M, |y|\leq 1\}}p^{\mu}_t(x,y)&\geq \eta_1 t^{-\frac{d}{2}}\exp\{-\eta_2(1+M^{2(d+1)\kappa})(1+(M+1)^{2\kappa})\\
            &\no{\geq \eta_1 t^{-\frac{d}{2}}\exp\{-\eta_2(1+M}
            -\eta_3t^{-1}(1+(M+1)^2)\}\\
            &=:\bar{\eta}(M,t)>0.
        \end{split}
    \end{equation*}
    Let $B_1$ be the unit ball in $\R^d$ and $\mathrm{Leb}(\cdot)$ be the Lebesgue measure on $\R^d$. Then $\nu(\cdot):=\mathrm{Leb}(\cdot\cap B_1)/\mathrm{Leb}(B_1)$ is a probability measure and
    \begin{equation*}
        \inf_{\{x: |x|\leq M\}}P^{\mu}_t(x,\xG)\geq \bar{\eta}(M,t)\mathrm{Leb}(B_1)\nu(\xG) \ \text{ for all } \xG\in\BB(\R^d).
    \end{equation*}
    Together with \eqref{ineq p=2}, \eqref{ineq p>2}, for any $p\geq 2$ and $R>0$, we have
    \begin{equation*}
        P_1^{\mu}V_p\leq \gamma_pV_p+K_p, \ \text{ and } \ \inf_{\{x: V_p(x)\leq R\}}P^{\mu}_1(x,\xG)\geq \eta_{p,R}\nu(\xG) \ \text{ for all } \xG\in\BB(\R^d),
    \end{equation*}
    where $\eta_{p,R}=\bar{\eta}(R^{1/p},1)\mathrm{Leb}(B_1)$ and
    \begin{equation*}
        \gamma_p=
        \begin{cases}
            e^{-2\alpha}, & p=2,\\
            e^{-\frac{p\alpha}{2}}, & p>2,
        \end{cases}
        \ \text{ and  } \ K_p=
        \begin{cases}
            \frac{\beta}{\alpha}\|\mu\|_2^2+\frac{2\gamma+\overline{\xs}d}{2\alpha}, & p=2,\\
            \frac{2}{p\alpha}\Big(\frac{p-2}{p\alpha}\Big)^{\frac{p-2}{2}}\big(2\beta\|\mu\|_2^2+2\gamma+\overline{\xs}(d+p-2)\big)^{\frac{p}{2}}, & p>2.
        \end{cases}
    \end{equation*}
    Now choose
    \begin{equation*}
       R_p=\frac{3K_p}{1-\gamma_p}>\frac{2K_p}{1-\gamma_p}, \ \text{ and } \ \beta_p=\min\Big\{\frac{\eta_{p,R_p}}{2K_p}, 1\Big\}<\frac{\eta_{p,R_p}}{K_p},
    \end{equation*}
    then Lemma \ref{Lem: HM} yields that for any $\mu_1,\mu_2\in \PP_p(\R^d)$,
    \begin{equation}\label{0111-1}
       \rho_{\beta_p,V_p}(P_1^{\mu,*}\mu_1,P_1^{\mu,*}\mu_2)\leq \zeta_p\rho_{\beta_p,V_p}(\mu_1,\mu_2),
    \end{equation}
    where
    \begin{equation*}
       \zeta_p=\max \bigg\{1-\eta_{p,R_p}+\beta_p K_p, \ \frac{2+\beta_p(\gamma_p R_p+2K_p)}{2+\beta_p R_p}\bigg\}<1.
    \end{equation*}
    From $\eqref{ineq p=2}$, \eqref{ineq p>2}, for any $t\geq 0$, one also has
    \begin{equation}\label{E:lya for all t}
        P_t^{\mu}V_p\leq V_p+K_p.
    \end{equation}
    Note that
   $\inf_{\{x: V_p(x)\leq R_p\}}P^{\mu}_t(x,\xG)\geq 0$   for all $t\geq0$ and $\xG\in\BB(\R^d)$.
    Then it follows from Lemma \ref{Lem: HM} again that for any $t\geq 0$,
    \begin{equation}\label{0111-2}
       \rho_{\beta,V_p}(P_t^{\mu,*}\mu_1,P_t^{\mu,*}\mu_2)\leq \tilde{\zeta}_p\rho_{\beta_p,V_p}(\mu_1,\mu_2),
    \end{equation}
    where
    \begin{equation*}
       \tilde{\zeta}_p=\max \bigg\{1+\beta_p K_p, \ \frac{2+\beta_p(R_p+2K_p)}{2+\beta_p R_p}\bigg\}<\infty.
    \end{equation*}
    By \eqref{0111-1} and \eqref{0111-2}, we have for any $t\geq 0$ and $\mu_1,\mu_2\in \PP_p(\R^d)$,
    \begin{equation*}
       \rho_{\beta_p,V_p}(P_t^{\mu,*}\mu_1,P_t^{\mu,*}\mu_2)\leq \frac{\tilde{\zeta}_p}{\zeta_p}e^{(\log\zeta_p)t}\rho_{\beta_p,V_p}(\mu_1,\mu_2).
    \end{equation*}
    Note that $\beta_p\leq 1$, by the definition of $d_p$ and $\rho_{\beta_p,V_p}$ we have
    \begin{equation*}
       \beta_p d_p(\mu_1,\mu_2)\leq \rho_{\beta_p,V_p}(\mu_1,\mu_2)\leq d_p(\mu_1,\mu_2),
    \end{equation*}
    then let $\lambda_p:=-\log \zeta_p>0$, we have for all $t\geq 0$ and $\mu_1,\mu_2\in \PP_p(\R^d)$,
    \begin{equation}\label{E:synchro}    d_p(P_t^{\mu,*}\mu_1,P_t^{\mu,*}\mu_2)\leq \frac{\tilde{\zeta}_p}{\beta_p\zeta_p}e^{-\lambda_p t}d_p(\mu_1,\mu_2).
    \end{equation}
    Hence, SDE \eqref{SDE 1} has at most one invariant measure in $\PP_p(\R^d)$. 
    Therefore, by \eqref{0322-1}, \eqref{E:lya for all t}, and \eqref{E:synchro}, for any $x\in \R^d$ and $s\geq t\geq 0$ we have
    \begin{equation}\label{0111-3}
       \begin{split}
           d_p(P_t^{\mu}(x,\cdot),P_s^{\mu}(x,\cdot))&=d_p(P_t^{\mu,*}\delta_x,P_t^{\mu,*}P_{s-t}^\mu(x,\cdot))\\
           &\leq \frac{\tilde{\zeta}_p}{\beta_p\zeta_p}e^{-\lambda_p t}\big(2+|x|^p+\mathbf{E}[|X_{s-t}^{\mu,x}|^p]\big)\\
           &\leq \frac{\tilde{\zeta}_p(2+K_p)}{\beta_p\zeta_p}(1+|x|^p)e^{-\lambda_p t}.
       \end{split}
    \end{equation}
    Then $\{P^{\mu}_s(x,\cdot)\}_{s\geq 0}$ is a Cauchy net in the complete metric space $(\PP_p(\R^d),d_p)$ and its limit is an (and hence the unique) invariant measure of \eqref{SDE 1} in $\PP_p(\R^d)$. Since $\PP_p(\R^d)$ is decreasing as $p$ increases, the invariant measures in $\PP_p(\R^d)$ are the same for all $p\geq2$. We obtain \eqref{converge to im} by letting $s\to \infty$ in \eqref{0111-3}.

    (ii). By \eqref{0322-1}, \eqref{converge to im}, 
    one has 
    \begin{equation*}
        \|\Psi(\mu)\|_p^p=\lim_{t\to \infty}\|P^{\mu}_t(0,\cdot)\|_p^p=\lim_{t\to \infty}P^{\mu}_tV_p(0), \text{ for all } p\geq2.
    \end{equation*}
    Therefore, by \eqref{ineq p=2}, \eqref{ineq p>2}, we have 
    \begin{equation*}
        \|\Psi(\mu)\|_2^2\leq \frac{\beta}{\alpha}\|\mu\|_2^2+\frac{2\gamma+\overline{\xs}d}{2\alpha},
    \end{equation*}
    and for $p>2$,
    \begin{equation*}
        \|\Psi(\mu)\|_p^p\leq \frac{2}{p\alpha}\Big(\frac{p-2}{p\alpha}\Big)^{\frac{p-2}{2}}\big(2\beta\|\mu\|_2^2+2\gamma+\overline{\xs}(d+p-2)\big)^{\frac{p}{2}}.
    \end{equation*}
\end{proof}

\begin{remark}\label{R:common converge speed in bdd set}
In fact, \cite[Theorem 3.10]{Feng-Qu-Zhao2023} shows that for all $\mu$ in a bounded subset of  $\PP_2(\R^d)$, there exist common constants $\eta_1,\eta_2,\eta_3>0$ such that \eqref{E:lower bound of transition kernel} holds.
Therefore, from the definition of $K_p$, $\beta_p$, $\zeta_p$, $\tilde{\zeta}_p$, $\lambda_p$ in the proof of Lemma \ref{Thm: esti of invari measure}, the conclusion of Lemma \ref{Thm: esti of invari measure} (i) can be strengthened as: 
for any $p\geq 2$, there exist constants $C, \lambda>0$ depending only on $K$, $\alpha$, $\beta$, $\gamma$, $\underline{\xs}$, $\overline{\xs}$, $\kappa$, $d$, $p$ and $\tilde{R}>0$ such that \eqref{converge to im} holds for any $\mu\in\overline{B_{\PP_2}(\xd_0,\tilde{R})}$, $x\in\R^d$ and $t\geq0$, where  
$\overline{B_{\PP_2}(\delta_0,\tilde{R})}:=\{\mu\in \PP_2(\R^d): \WW_2(\mu,\delta_0)=\|\mu\|_2\leq \tilde{R}\}.$
\end{remark}

Next, we give a compact embedding lemma between Wasserstein spaces,  which will also be used in Section \ref{Section:Proof-Main-Results} when verifying the compactness of the order interval enclosed by two order-related $\PP_{\8}-$invariant measures.

\begin{lemma}\label{lem:compact-embed}
For any $p>q\geq 1$, the map $\PP_p(\R^d)\hookrightarrow \PP_q(\R^d)$ is a compact embedding, i.e., any bounded set in $(\PP_p(\R^d),\WW_p)$ is pre-compact in $(\PP_q(\R^d),\WW_q)$.
\end{lemma}

\begin{proof}
It suffices to show that any sequence $\{\mu_n\}_{n\geq 1}$ with $R:=\sup_{n\geq 1}\|\mu_n\|_p<\infty$ has a Cauchy subsequence in $\PP_q(\R^d)$. For any $N>0$, let $B_N=B(0,N)$ be the open ball centered at $0$ in $\R^d$ with radius $N$. By Chebyshev's inequality, 
\begin{equation}\label{E:tight}
    \sup_{n\geq 1}\mu_n(B_N^c)\leq \sup_{n\geq 1}\frac{1}{N^p}\int_{B_N^c}|x|^p\d \mu_n(x)\leq \frac{R^p}{N^p}.
\end{equation}
Then $\{\mu_n\}_{n\geq 1}$ is tight and hence weakly pre-compact. Thus there exist a subsequence of $\{\mu_n\}_{n\geq 1}$ which is still denoted by $\{\mu_n\}_{n\geq 1}$ and some $\mu\in \PP(\R^d)$ such that $\mu_n\xrightarrow{w} \mu$, i.e.,
\begin{equation*}
    \lim_{n\to \infty}\int_{\R^d}f(x)\d \mu_n(x)= \int_{\R^d}f(x)\d \mu(x) \ \text{ for all bounded continuous function  $f$}.
\end{equation*}
Choose $f_N(x)=|x|^p\wedge N$. Since $\mu_n\goto{w}\mu$, we have
\begin{equation}\label{1115-1}
    \int_{\R^d}f_N(x)\d \mu(x)= \lim_{n\to \infty}\int_{\R^d}f_N(x)\d \mu_n(x)\leq \sup_{n\geq 1}\int_{\R^d}|x|^p\d \mu_n(x)= R^p.
\end{equation}
Note that $f_N(x)\uparrow |x|^p$ as $N\to \infty$. Then the monotone convergence theorem yields that
\begin{equation*}
    \int_{\R^d}|x|^p\d \mu(x)=\lim_{N\to \infty}\int_{\R^d}f_N(x)\d \mu(x)\leq R^p.
\end{equation*}
Hence $\mu\in \PP_p(\R^d)\subset \PP_q(\R^d)$. 
Then by Villani \cite[Theorem 6.9]{Villani2009}, to show $\lim_{n\to \infty}\WW_q(\mu_n,\mu)=0$, it suffices to prove
\begin{equation}\label{1115-2}
    \lim_{n\to \infty}\int_{\R^d}|x|^q\d \mu_n(x)= \int_{\R^d}|x|^q\d \mu(x).
\end{equation}
For any $N>0$, $\mu_n\goto{w}\mu$ implies
\begin{equation}\label{E:mun-weak-mu}
    \lim_{n\to \infty}\int_{\R^d}|x|^q\wedge N^q\d \mu_n(x)=\int_{\R^d}|x|^q\wedge N^q\d \mu(x).
\end{equation}
Note also that for all $n\geq 1$, H\"older's inequality and \eqref{E:tight} gives
\begin{equation}\label{E:mun-uniform-control}
    \bigg|\int_{\R^d}|x|^q\d \mu_n(x)-\int_{\R^d}|x|^q\wedge N^q\d \mu_n(x)\bigg|\leq \int_{B_N^c}|x|^q \d\mu_n(x)\leq \frac{R^p}{N^{p-q}}.
\end{equation}
Then for any $N>0$, it follows from \eqref{E:mun-weak-mu}-\eqref{E:mun-uniform-control} that
\begin{equation}\label{1115-3}
    \begin{split}
        &\ \ \ \ \limsup_{n\to \infty}\bigg|\int_{\R^d}|x|^q\d \mu_n(x)-\int_{\R^d}|x|^q\d \mu(x)\bigg|\\
        &\leq \limsup_{n\to \infty}\bigg|\int_{\R^d}|x|^q\d \mu_n(x)-\int_{\R^d}|x|^q\wedge N^q\d \mu_n(x)\bigg|\\
        &\ \ \ \ +\limsup_{n\to \infty}\bigg|\int_{\R^d}|x|^q\wedge N^q\d \mu_n(x)-\int_{\R^d}|x|^q\wedge N^q\d \mu(x)\bigg|\\
        &\ \ \ \ +\bigg|\int_{\R^d}|x|^q\wedge N^q\d \mu(x)-\int_{\R^d}|x|^q\d \mu(x)\bigg|\\
        &\leq \frac{R^p}{N^{p-q}}+\bigg|\int_{\R^d}|x|^q\wedge N^q\d \mu(x)-\int_{\R^d}|x|^q\d \mu(x)\bigg|.
    \end{split}
\end{equation}
By monotone convergence theorem, we have 
\begin{equation}\label{E:mct-mu}
    \lim_{N\to \infty}\bigg|\int_{\R^d}|x|^q\wedge N^q\d \mu(x)-\int_{\R^d}|x|^q\d \mu(x)\bigg|=0.
\end{equation}
Hence, by letting $N\to \infty$ in \eqref{1115-3}, \eqref{E:mct-mu} implies \eqref{1115-2}.
\end{proof}

We are now in a position to present the proof of Proposition \ref{P:Psi-Cts-Cpt-Monotone}.

\begin{proof}[Proof of Proposition \ref{P:Psi-Cts-Cpt-Monotone}]
    \textbf{Compactness:} For any $M>0$, let 
    \begin{equation*}
    \overline{B_{\PP_2}(\delta_0,M)}=\{\mu\in \PP_2(\R^d): \WW_2(\mu,\delta_0)=\|\mu\|_2\leq M\}.
    \end{equation*}
It suffices to show that $\Psi(\overline{B_{\PP_2}(\delta_0,M)})$ is pre-compact in $\PP_2(\R^d)$. By \eqref{uniform bounds of im} in Lemma \ref{Thm: esti of invari measure}, for any $p\geq 2$,
\begin{equation*}
    \sup_{\mu\in \Psi(\overline{B_{\PP_2}(\delta_0,M)})}\|\mu\|_p<\infty.
\end{equation*}
The pre-compactness of $\Psi(\overline{B_{\PP_2}(\delta_0,M)})$ in $\PP_2(\R^d)$ follows from Lemma \ref{lem:compact-embed}.

\textbf{Continuity:} For any sequence $\{\mu_n\}_{n\geq 1}\subset \PP_2(\R^d)$ and $\mu\in \PP_2(\R^d)$ with $\mu_n\goto{\WW_2}\mu$ as $n\to\8$, we need to show that 
$\Psi(\mu_n)\xrightarrow{\WW_2}\Psi(\mu)$.

From the definition of $d_2$ in \eqref{0322-1} and \cite[Theorem 6.15]{Villani2009}, for any $\mu_1,\mu_2\in \PP_2(\R^d)$,
\begin{equation}\label{new1116-1}
    \WW_2(\mu_1,\mu_2)\leq \sqrt{2d_{2}(\mu_1,\mu_2)}.
\end{equation}
Note that $\mu_n\xrightarrow{\WW_2}\mu$, we have $\sup_{n\geq 1}\|\mu_n\|_2\vee \|\mu\|_2<\infty$. 
Then according to \eqref{converge to im} in Lemma \ref{Thm: esti of invari measure}, Remark \ref{R:common converge speed in bdd set} and \eqref{new1116-1}, there exist $C,\lambda>0$ such that for all $t\geq 0$,
\begin{equation*}
    \sup_{n\geq 1}\WW_2\big(\Psi(\mu_n),\LL(X^{\mu_n,0}_t)\big)\vee \WW_2\big(\Psi(\mu),\LL(X^{\mu,0}_t)\big) \leq Ce^{-\lambda t}.
\end{equation*}
Hence, for any $\epsilon>0$, there exists $T>0$ such that 
\begin{equation}\label{newE:T-mun-mu-control}
    \sup_{n\geq 1}\WW_2\big(\Psi(\mu_n),\LL(X^{\mu_n,0}_T)\big)\vee \WW_2\big(\Psi(\mu),\LL(X^{\mu,0}_T)\big)\leq \frac{\epsilon}{2}.
\end{equation}
Similar to the proof of Claim 1 in Proposition \ref{prop:continuity-Pt}, we have
\begin{equation*}
    \mathbf{E}[|X_t^{\mu_n,0}-X_t^{\mu,0}|^2]\leq K\WW_2^2(\mu_n,\mu) t+4K\int_0^t\mathbf{E}[|X_s^{\mu_n,0}-X_s^{\mu,0}|^2]ds.
\end{equation*}
Then the Gronwall's inequality shows that
\begin{equation}\label{newE:XTmun-XTmu}
   \WW_2\big(\LL(X^{\mu_n,0}_T),\LL(X^{\mu,0}_T)\big)
\leq \big(\mathbf{E}[|X_T^{\mu_n,0}-X_T^{\mu,0}|^2]\big)^{1/2}
\leq \sqrt{KT}e^{2KT}\WW_2(\mu_n,\mu).
\end{equation}
Therefore, by \eqref{newE:T-mun-mu-control}-\eqref{newE:XTmun-XTmu}, one has
\begin{equation*}
   \begin{split}
       \limsup_{n\to\infty} \WW_2\big(\Psi(\mu_n),\Psi(\mu)\big)&\leq \limsup_{n\to\infty} \Big\{\WW_2\big(\Psi(\mu_n),\LL(X^{\mu_n,0}_T)\big)\\
       &\ \ \ \ \qquad\qquad \ +\WW_2\big(\LL(X^{\mu_n,0}_T),\LL(X^{\mu,0}_T)\big)+\WW_2\big(\LL(X^{\mu,0}_T),\Psi(\mu)\big)\Big\}\\
       &\leq \epsilon.
   \end{split}
\end{equation*}
So $\Psi(\mu_n)\xrightarrow{\WW_2} \Psi(\mu)$ by the arbitrariness of $\epsilon>0$.

\textbf{Monotonicity:} By Lemma \ref{lem:ldsde-comparison},   we have $\LL(X^{\mu_1,0}_t)\leq \LL(X^{\mu_2,0}_t)$, for all $t\geq0$. According to \eqref{converge to im} in Lemma \ref{Thm: esti of invari measure} and \eqref{new1116-1}, we have $\LL(X^{\mu_i,0}_t)\xrightarrow{\WW_2} \Psi(\mu_i)$ as $t\to \8$ for all $i=1,2$. Then, the closedness of  stochastic order (see Lemma \ref{lem:stochastic-order-closed}) gives $\Psi(\mu_1)\leq \Psi(\mu_2)$.
\end{proof}

\subsection{Existence of order-related invariant measures}\label{subsec:order-related-im}

For any $a\in \mathbb{R}^d$ and $\nu\in \mathcal{P}(\mathbb{R}^d)$, denote by $\nu^a$ the shift probability of $\nu$ by $a$:
	\begin{equation}\label{0913-1}
		\int_{\mathbb{R}^d}f(x)\d\nu^a(x):=\int_{\mathbb{R}^d}f(x-a)\d\nu(x), \ \text{ for any bounded measurable function } f.  
	\end{equation}
The following result gives the existence of multiple order-related invariant measures and also shrinking neighbourhoods of invariant measures under the semiflow $\str P_t$.

\begin{theorem}\label{Thm: two ordered im}
   Suppose that Assumption \ref{asp:lipschitz}, \ref{asp:cooperation}， \ref{asp:dissipative-growth-nondegeneracy} hold. If the equation \eqref{MVE} is locally dissipative at $a_i$ with configuration $(r_{a_i},g_{a_i})$ for any $1\leq i\leq n$ with some $n\geq 2$ and 
   \begin{equation}\label{1123-1}
		a_1<a_2<\cdots<a_n \ \text{ and } \ r_{a_i}+r_{a_{i+1}}\leq \abs{a_{i+1}-a_i}, \ \text{ for all } \ 1\leq i\leq n-1,
	\end{equation}
    then
    \begin{enumerate}
    \item [\textnormal{(i)}] equation \eqref{MVE} has $n$ order-related invariant measures $\mu_1<\mu_2<\cdots<\mu_n$, satisfying 
    \[\text{$\mu_i\in B_{\PP_2}(\xd_{a_i},r_{a_i})\cap \PP_{\infty}(\R^d)$ for all $i=1,2,\ldots,n$;}\]
    \item [\textnormal{(ii)}] there exists $T>0$ and $0<\tilde{r}_{a_i}<r_{a_i}, \ i=1,2,\cdots,n$ such that for all $t\geq T$,
    \[
    \text{$\str P_t\overline{B_{\PP_2}(\xd_{a_i},r_{a_i})}\subset B_{\PP_2}(\xd_{a_i},\tilde{r}_{a_i})$, \ \ $i=1,2,\cdots,n$.}
    \]
   \end{enumerate}
\end{theorem}

Before proceeding with the proof of Theorem \ref{Thm: two ordered im}, we need the following two lemmas.

\begin{lemma}\label{Lemma: im in P_infty}
    Suppose that Assumption \ref{asp:lipschitz}, \ref{asp:dissipative-growth-nondegeneracy} hold. Then if $\mu\in \PP_{2}(\R^d)$ is an invariant measure of \eqref{MVE}, we have $\mu\in \PP_{\infty}(\R^d)$. More precisely, for any $p>2$,
    \begin{equation*}
      \|\mu\|_p\leq \sqrt{\frac{2\gamma+\overline{\xs}(d+p-2)}{\alpha-\beta}}.
  \end{equation*}
\end{lemma}
\begin{proof}
  Note that $\mu\in \PP_{2}(\R^d)$ is an invariant measure of \eqref{MVE} if and only if $\mu$ is a fixed point of $\Psi$. By \eqref{uniform bounds of im} in Lemma \ref{Thm: esti of invari measure}, we have
  \begin{equation}\label{eq:mynew1222}
      \|\mu\|_2^2=\|\Psi(\mu)\|_2^2\leq \frac{\beta}{\alpha}\|\mu\|_2^2+\frac{2\gamma+\overline{\xs}d}{2\alpha} \ \ \Rightarrow \ \ \|\mu\|_2\leq \sqrt{\frac{2\gamma+\overline{\xs}d}{2(\alpha-\beta)}},
  \end{equation}
  and hence,  by \eqref{uniform bounds of im} again,  for any $p>2$,
  \begin{equation*}
     \begin{split}
         \|\mu\|_p^p&\leq \frac{2}{p\alpha}\Big(\frac{p-2}{p\alpha}\Big)^{\frac{p-2}{2}}\Big(\frac{\beta(2\gamma+\overline{\xs}d)}{\alpha-\beta}+2\gamma+\overline{\xs}(d+p-2)\Big)^{\frac{p}{2}}\\
         &\leq \bigg(\frac{2\gamma+\overline{\xs}(d+p-2)}{\alpha-\beta}\bigg)^{\frac{p}{2}}.
     \end{split}
  \end{equation*}
  It follows that \eqref{eq:mynew1222} holds.
\end{proof}

\begin{lemma}\label{Lem: equavilent converge}
    Given $\mathcal{A}\subset \PP_{\infty}(\R^d)$, suppose that, for any $p\geq 1$, there exists $C_p>0$ such that 
    \begin{equation*}
        \sup_{\mu\in \mathcal{A}}\|\mu\|_p\leq C_p.
    \end{equation*}
    Then for any sequence $\{\mu_n\}_{n\geq 1}\subset \mathcal{A}$ and $\mu\in \PP(\R^d)$, the following convergences are equivalent:
    \begin{itemize}
        \item $\mu_n\xrightarrow{\WW_p}\mu$ for all $p\geq 1$;
        \item $\mu_n\xrightarrow{\WW_p}\mu$ for some $p\geq 1$;
        \item $\mu_n\xrightarrow{w}\mu$.
    \end{itemize}
\end{lemma}
\begin{proof}
    By Villani \cite[Theorem 6.9]{Villani2009}, for any fixed $p\geq 1$, $\mu_n\xrightarrow{\WW_p}\mu$ is equivalent to that $\mu_n\xrightarrow{w}\mu$ and 
    \begin{equation}\label{1122-2}
        \lim_{n\to \infty}\int_{\R^d}|x|^p\d \mu_n(x)=\int_{\R^d}|x|^p\d \mu(x).
    \end{equation}
    Hence, to prove Lemma \ref{Lem: equavilent converge}, it suffices to show that if $\mu_n\xrightarrow{w}\mu$, then \eqref{1122-2} holds.

    Let
    \begin{equation*}
        a_{N,n}:=\int_{\R^d}(|x|^p\wedge N)\d \mu_n(x), \ b_n:=\int_{\R^d}|x|^p\d \mu_n(x), \ c_N:=\int_{\R^d}(|x|^p\wedge N)\d \mu(x).
    \end{equation*}
    Since $\mu_n\goto{w}\mu$, we have for any $N>0$, $\lim_{n\to \infty}a_{N,n}=c_N$. Hence 
    \begin{equation*}
        \sup_{N\geq 1}c_N\leq \sup_{N,n\geq1}a_{N,n}\leq \sup_{n\geq 1}\|\mu_n\|_p^p\leq C_p^{p}.
    \end{equation*}
    By monotone convergence theorem, we have 
    \begin{equation*}
        \lim_{N\to \infty}c_N=\int_{\R^d}|x|^p\d \mu(x)\leq C_p^p.
    \end{equation*}
    Now choose $p'>p$, we have for any $n\geq 1$,
    \begin{equation*}
        |a_{N,n}-b_n|\leq \int_{\{x\in\R^d:|x|^p\geq N\}}|x|^p\d \mu_n(x)\leq N^{-\frac{p'-p}{p}}\int_{\{x\in\R^d:|x|^p\geq N\}}|x|^{p'}\d \mu_n(x)\leq C_{p'}^{p'}N^{-\frac{p'-p}{p}}.
    \end{equation*}
    Then $\lim_{N\to \infty}a_{N,n}=b_n$ uniformly in $n$. By Moore-Osgood theorem, we conclude
    \begin{equation*}
        \lim_{n\to\infty}\int_{\R^d}|x|^p\d \mu_n(x)=\lim_{n\to\infty}b_n=\lim_{N\to \infty}c_N=\int_{\R^d}|x|^p\d \mu(x).
    \end{equation*}
\end{proof}

Now let us give the proof of Theorem \ref{Thm: two ordered im}.

\begin{proof}[Proof of Theorem \ref{Thm: two ordered im}]
    (i). Recall Definition \ref{def:locally-dissipative} that the equation \eqref{MVE} being locally dissipative at $a_i$ with configuration $(r_{a_i},g_{a_i})$ means that
    \begin{itemize}
       \item for any $x\in \R^d$ and $\nu\in \PP_{2}(\R^d)$
       \begin{equation}\label{1119-1}
			2\langle x,b(x+a_i,\nu)\rangle+|\sigma(x+a_i)|_2^2\leq -g_{a_i}(|x|^{2},\|\nu^{a_i}\|_{2}^2);
		\end{equation}
       \item $r_{a_i}>0$ and $g_{a_i}$ satisfies
       \begin{equation}\label{1119-2}
			\begin{split}
			    &g_{a_i}(\cdot,r_{a_i}^2) \text{ is continuous and convex};\\
                &\inf_{0\leq w\leq r_{a_i}^2}g_{a_i}(z,w)=g_{a_i}(z,r_{a_i}^2), \ \text{ for all } z\geq 0;\\
                & g_{a_i}(z,r_{a_i}^2)>0, \ \text{ for all } z\geq r_{a_i}^2.
			\end{split}
		\end{equation}
   \end{itemize}  
    Note that $\mu\in \PP_2(\R^d)$ is an invariant measure of \eqref{MVE} if and only if it is a fixed point of $\Psi$, where the map $\Psi$ is defined in \eqref{def of Psi}. 
    We first prove the following property of $\Psi$:
    \begin{equation}\label{eq:shrink-psi}
        \Psi\big(\overline{B_{\PP_2}(\delta_{a_i},r_{a_i})}\big)\subset B_{\PP_2}(\delta_{a_i},r_{a_i}), \ \text{ for all } \ 1\leq i\leq n.
    \end{equation}
    Fix any $\mu\in \overline{B_{\PP_2}(\delta_{a_i},r_{a_i})}$ in \eqref{SDE 1}. 
    Consider the solution $\{X_t^{\mu,a_i}\}_{t\geq 0}$ of \eqref{SDE 1} with initial condition $a_i$, the It\^o's formula to $|X_t^{\mu,a_i}-a_i|^2$ together with \eqref{1119-1}-\eqref{1119-2} gives
    \begin{equation*}
        \begin{split}
            |X_t^{\mu,a_i}-a_i|^2&=\int_0^t\big(2\langle X_s^{\mu,a_i}-a_i, b(X_s^{\mu,a_i},\mu)\rangle+|\xs(X_s^{\mu,a_i})|^2\big)\d s\\
            &\ \ \ \ +2\int_0^t\langle X_s^{\mu,a_i}-a_i, \xs(X_s^{\mu,a_i})\d W_s\rangle\\
            &\leq -\int_0^tg_{a_i}(|X_s^{\mu,a_i}-a_i|^{2}, r_{a_i}^2)\d s+2\int_0^t\langle X_s^{\mu,{a_i}}-a_i, \xs(X_s^{\mu,a_i})\d W_s\rangle.
        \end{split}
    \end{equation*} 
    Taking expectations on both sides and by the convexity of $g_{a_i}(\cdot,r_{a_i}^2)$, we have that for any $t>0$,
    \begin{equation}\label{1122-1}
        g_{a_i}\bigg(\frac{1}{t}\int_0^t\mathbf{E}\big[|X_s^{\mu,{a_i}}-a_i|^{2}\big] \d s, r_{a_i}^2\bigg)\leq \frac{1}{t}\int_0^t\mathbf{E}\big[g_{a_i}\big(|X_s^{\mu,a_i}-a_i|^{2},r_{a_i}^2\big)\big] \d s\leq -\frac{1}{t}\mathbf{E}[|X_t^{\mu,a_i}-a_i|^2]\leq 0.
    \end{equation}
    According to \eqref{converge to im}  in Lemma \ref{Thm: esti of invari measure}, there exist $C>0$, $\xl>0$ such that
    \begin{equation*}
        \big|\mathbf{E}\big[|X_s^{\mu,a_i}-a_i|^{2}\big]-\|\Psi(\mu)^{a_i}\|_2^2\big|=\Big|\int_{\R^d}|x-a_i|^2 \big(P^\mu_s(a_i,\d x)-\d \Psi(\mu)(x)\big)\Big|\leq C(1+|a_i|^{2})e^{-\lambda s}.
    \end{equation*}
    Thus
    \begin{equation*}
        \lim_{t\to \infty}\frac{1}{t}\int_0^t\mathbf{E}\big[|X_s^{\mu,a_i}-a_i|^{2}\big] \d s=\|\Psi(\mu)^{a_i}\|_2^2.
    \end{equation*}
    Then the continuity of $g_{a_i}(\cdot,r_{a_i}^2)$ and \eqref{1122-1} yield $g_{a_i}(\|\Psi(\mu)^{a_i}\|_2^2,r_{a_i}^2)\leq 0$. 
    Hence, by \eqref{1119-2}, we have $\|\Psi(\mu)^{a_i}\|_2^2<r_{a_i}^2$, which means that $\Psi(\mu)\in B_{\PP_2}(\delta_{a_i},r_{a_i})$. Thus, \eqref{eq:shrink-psi} holds.

By \eqref{eq:shrink-psi}, one has
    \begin{equation}\label{1122-4}
        \Psi\big(\overline{B_{\PP_2}(\delta_{a_i},r_{a_i})}\big)\subset B_{\PP_2}(\delta_{a_i},r_{a_i})\subset \overline{B_{\PP_2}(\delta_{a_i},r_{a_i})}.
    \end{equation}
Together with \eqref{uniform bounds of im} in Lemma \ref{Thm: esti of invari measure}, for any $p\geq2$, one has 
$$\sup\big\{\|\mu\|_p: \mu\in \Psi\big(\overline{B_{\PP_2}(\delta_{a_i},r_{a_i})}\big)\big\}<\infty.$$
Hence, for any $p\geq2$,
\begin{equation}\label{E:convex hull bdd}
\sup\big\{\|\mu\|_p: \mu\in \conv\big(\Psi\big(\overline{B_{\PP_2}(\delta_{a_i},r_{a_i})}\big)\big)\}=\sup\big\{\|\mu\|_p: \mu\in \Psi\big(\overline{B_{\PP_2}(\delta_{a_i},r_{a_i})}\big)\big\}<\infty,
\end{equation}
where $\conv\big(\Psi\big(\overline{B_{\PP_2}(\delta_{a_i},r_{a_i})}\big)\big)$ is the convex hull of $\Psi\big(\overline{B_{\PP_2}(\delta_{a_i},r_{a_i})}\big)$ in $\PP_{2}(\R^d)$.
Let
    \begin{equation*}
        \MM_i:=\overline{\conv\big(\Psi\big(\overline{B_{\PP_2}(\delta_{a_i},r_{a_i})}\big)\big)},
    \end{equation*}
    the closure of $\conv\big(\Psi\big(\overline{B_{\PP_2}(\delta_{a_i},r_{a_i})}\big)\big)$ in $\PP_{2}(\R^d)$.
By \eqref{E:convex hull bdd} and Lemma \ref{Lem: equavilent converge}, we have 
\begin{equation*}
        \MM_i=\overline{\conv\big(\Psi\big(\overline{B_{\PP_2}(\delta_{a_i},r_{a_i})}\big)\big)}^{\WW_1},
    \end{equation*}
     where $\overline{A}^{\WW_1}$ is the closure of $A$ in $\PP_1(\R^d)$.
Then by \eqref{E:convex hull bdd} and Lemma \ref{lem:compact-embed}, $\MM_i$ is compact convex in $\PP_1(\R^d)$. 
    Note that $\PP_1(\R^d)$ is a closed convex subset of $(\MM_1(\R^d),\norm{\cdot}_{\WW_1})$ (see Remark \ref{rem:p1-closed-m1}), where $(\MM_1(\R^d),\norm{\cdot}_{\WW_1})$ is a normed vector space (see Lemma \ref{lem:m1-norm-space}), thus, $\MM_i$ is a compact convex subset of a locally convex topological vector space. 
    On the other hand, $\overline{B_{\PP_2}(\delta_{a_i},r_{a_i})}$ is closed convex in $\PP_2(\R^d)$, then by \eqref{1122-4}, we have $\MM_i\subset \overline{B_{\PP_2}(\delta_{a_i},r_{a_i})}$. 
    Hence $\Psi(\MM_i)\subset \MM_i$. 
    \eqref{E:convex hull bdd} and Lemma \ref{Lem: equavilent converge} entails that for any $p\geq2$
    \begin{equation*}
        \MM_i=\overline{\conv\big(\Psi\big(\overline{B_{\PP_2}(\delta_{a_i},r_{a_i})}\big)\big)}^{\WW_p},
    \end{equation*}
     where $\overline{A}^{\WW_p}$ is the closure of $A$ in $\PP_p(\R^d)$.
     Then \eqref{E:convex hull bdd} implies that $\sup_{\mu\in \MM_i}\|\mu\|_p<\infty$, for any $p\geq2$.
     Therefore, by Proposition \ref{P:Psi-Cts-Cpt-Monotone} and Lemma \ref{Lem: equavilent converge}, we also have $\Psi: \MM_i\to \MM_i$ is $\WW_1$-continuous.  Till now, we have shown that
     $\{\MM_i\}_{i=1}^n$ are all compact convex subsets of the locally convex topological vector space $\MM_1(\R^d)$ and for all $i=1,2,\cdots,n$,
    \begin{equation}\label{1122-5}
        \MM_i\subset \overline{B_{\PP_2}(\delta_{a_i},r_{a_i})}, \ \ \ \Psi(\MM_i)\subset \MM_i\cap B_{\PP_2}(\delta_{a_i},r_{a_i}), \ \Psi:\MM_i\to\MM_i\  \text{is}\  \WW_1\text{-continuous}.
    \end{equation}
    Since $r_{a_i}+r_{a_{i+1}}\leq \abs{a_{i+1}-a_i}=\WW_2(\delta_{a_i},\delta_{a_{i+1}})$, one has $B_{\PP_2}(\delta_{a_i},r_{a_i})\cap B_{\PP_2}(\delta_{a_{i+1}},r_{a_{i+1}})=\emptyset$.
    Then by \eqref{1122-5}, $\Psi(\MM_i)\cap \Psi(\MM_{i+1})=\emptyset$ for all $1\leq i\leq n-1$.
     Now define $\widetilde{\Psi}: [\PP_2(\R^d)]^{\otimes n}\to [\PP_2(\R^d)]^{\otimes n}$ by
    \begin{equation*}
        \widetilde{\Psi}(\mu_1,\mu_2,\cdots,\mu_n):=(\Psi(\mu_1),\Psi(\mu_2),\cdots,\Psi(\mu_n)).
    \end{equation*}
    Set
    \begin{equation*}
        E:=\big\{(\mu_1,\mu_2,\cdots,\mu_n)\in [\PP_2(\R^d)]^{\otimes n}: \mu_1\leq \mu_2\leq \cdots\leq \mu_n\big\}.
    \end{equation*}
    By Proposition \ref{P:Psi-Cts-Cpt-Monotone} and \eqref{1122-5}, we have
    \begin{equation*}
        \widetilde{\Psi}\big((\MM_1\times \MM_2\times \cdots\times\MM_n)\cap E\big)\subset (\MM_1\times \MM_2\times \cdots\times\MM_n)\cap E.
    \end{equation*}
    By Lemma \ref{lem:stochastic-order-closed} and Lemma \ref{lem:m1-cone}, we have $E$ is a closed convex subset of $[\MM_1(\R^d)]^{\otimes n}$. 
    Notice also that $\MM_1\times\MM_2\times \cdots\times\MM_n$ is a compact convex subset of $[\MM_1(\R^d)]^{\otimes n}$, then  $(\MM_1\times\MM_2\times \cdots\times\MM_n)\cap E$ is a compact convex subset of $[\MM_1(\R^d)]^{\otimes n}$. 
    Since $a_1<a_2<\cdots<a_n$ and $\Psi$ is monotone (Proposition \ref{P:Psi-Cts-Cpt-Monotone}), one has $(\Psi(\delta_{a_1}),\Psi(\delta_{a_2}),\cdots,\Psi(\delta_{a_n}))\in (\MM_1\times \MM_2\times \cdots\times\MM_n)\cap E$, which means $(\MM_1\times \MM_2\times \cdots\times\MM_n)\cap E\neq \emptyset$.
    Note that $\widetilde{\Psi}$ is  continuous by the continuity $\Psi$ in  \eqref{1122-5}. 
    Then Tychonoff fixed-point theorem implies that $\widetilde{\Psi}$ has a fixed point $(\mu_1,\mu_2,\cdots,\mu_n)$ in $(\MM_1\times \MM_2\times \cdots\times\MM_n)\cap E$, i.e.,
    \begin{equation*}
        \Psi(\mu_i)=\mu_i, \ \mu_i\in \Psi(\MM_i)\subset B_{\PP_2}(\delta_{a_i},r_{a_i}), \ \text{ for all } \ i=1,2,\cdots,n, \ \text{ and } \ \mu_1\leq \mu_2\leq \cdots\leq \mu_n.
    \end{equation*}
    Since $\Psi(\MM_i)\cap \Psi(\MM_{i+1})=\emptyset$ for all $1\leq i\leq n-1$, we conclude that the equation \eqref{MVE} has $n$ distinct invariant measures $\mu_1,\mu_2,\cdots,\mu_n$ such that $\mu_1<\mu_2<\cdots<\mu_n$ with $\mu_i\in B_{\PP_2}(\delta_{a_i},r_{a_i})$.  The conclusion $\mu_i\in \PP_{\8}$ follows from Lemma \ref{Lemma: im in P_infty}.

    (ii). According to \eqref{1119-2}, there exists $0<\hat{r}_{a_i}<r_{a_i}$ such that
    \begin{equation}\label{eq:tilde-r-a}
        g_{a_i}(z,r_{a_i}^2)>0, \ \text{ for all } \ z\geq \hat{r}_{a_i}.
    \end{equation}
    Fix any $\tilde{r}_{a_i}\in (\hat{r}_{a_i},r_{a_i})$, we are going to show that
    \begin{equation}\label{0111-4}
        P_t^*B_{\PP_2}(\delta_{a_i},r)\subset B_{\PP_2}(\delta_{a_i},r), \ \text{ for all } \ t\geq 0, \ \text{ and } \ \tilde{r}_{a_i}\leq r\leq r_{a_i}, \ 1\leq i\leq n
    \end{equation}
    To prove \eqref{0111-4}, it suffices to show that for all $1\leq i\leq n$,
    \begin{equation}\label{E:closed ball positively invariant}
    P_t^*\overline{B_{\PP_2}(\delta_{a_i},r)}\subset \overline{B_{\PP_2}(\delta_{a_i},r)}, \ \text{ for all } \ t\geq 0, \ \hat{r}_{a_i}\leq r<r_{a_i}.
    \end{equation}
    For any fixed $\hat{r}_{a_i}\leq r<r_{a_i}$ and $\mu\in \overline{B_{\PP_2}(\delta_{a_i},r)}$, let
    \begin{equation*}
        \TT_{\mu}^i:=\sup\big\{t\geq 0: P_s^*\mu \in \overline{B_{\PP_2}(\delta_{a_i},r)} \ \text{ for all } \ 0\leq s\leq t\big\}.
    \end{equation*}
    To prove \eqref{E:closed ball positively invariant}, it suffices to show that $\TT_{\mu}^i=\infty$. 
    Assume the contrary that $\TT_{\mu}^i<\infty$.
    Define 
    $$\alpha^i_t:=\|(P_t^*\mu)^{a_i}\|_2^2=\WW_2^2(P_t^*\mu,\delta_{a_i}).$$
    It follows from Claim 2 in Proposition \ref{prop:continuity-Pt} that $\alpha^i_t$ is continuous in $t$.
    Notice that 
    \begin{equation}\label{E:Tmu equi def}
        \TT^i_{\mu}=\sup\big\{t\geq 0: \alpha^i_s\leq r^2 \ \text{ for all } \ 0\leq s\leq t\big\}.
    \end{equation}
    Then by $\TT^i_{\mu}<\infty$ and \eqref{E:Tmu equi def}, there exists $T_i\geq \TT^i_{\mu}$ and $\epsilon_i>0$ such that
    \begin{equation}\label{1127-4}
        \alpha^i_{T_i}=r^2, \ \text{ and } \ r^2<\alpha^i_t\leq r_{a_i}^2 \ \text{ for all } \ T_i<t\leq T_i+\epsilon_i.
    \end{equation}
    
    Denote by $X_t^{\mu}$ the solution  of \eqref{MVE} with intial distribution $\mu$ at starting time $0$. Then $\LL(X_t^{\mu})=P_t^*\mu$ for all $t\geq 0$. Applying It\^o's formula to $|X_t^{\mu}-a_i|^2$ on $[T_i,T_i+\epsilon_i]$, by \eqref{1119-1}-\eqref{1119-2}, we have for any $T_i<t\leq T_i+\epsilon_i$,
    \begin{equation*}
        \begin{split}
            |X_t^{\mu}-a_i|^2&=|X_{T_i}^{\mu}-a_i|^2+\int_{T_i}^t\big(2\langle X_s^{\mu}-a_i, b(X_s^{\mu},P_s^*\mu)\rangle+|\xs(X_s^{\mu})|^2\big)\d s\\
            &\ \ \ \ +2\int_{T_i}^t\langle X_s^{\mu}-a_i, \xs(X_s^{\mu})\d W_s\rangle\\
            &\leq |X_{T_i}^{\mu}-a_i|^2-\int_{T_i}^tg_{a_i}\big(|X_s^{\mu}-a_i|^{2}, r_{a_i}^2\big)\d s+2\int_{T_i}^t\langle X_s^{\mu}-a_i, \xs(X_s^{\mu})\d W_s\rangle.
        \end{split}
    \end{equation*}
    Taking expectations on both sides and by the convexity of $g_{a_i}(\cdot,r_{a_i}^2)$, we get
    \begin{equation}\label{1124-1}
        \alpha_{T_i+\epsilon_i}\leq r^2-\int_{T_i}^{T_i+\epsilon_i}g_{a_i}(\alpha^i_t, r_{a_i}^2)\d t.
    \end{equation}
     Then by \eqref{eq:tilde-r-a}, \eqref{1127-4} and \eqref{1124-1}, we conclude that $\alpha_{T_i+\epsilon_i}<r^2$ which contradicts \eqref{1127-4}.
     Hence, we obtain \eqref{E:closed ball positively invariant}, and \eqref{0111-4} follows.
     
     Next we use \eqref{0111-4} to finish the proof of (ii).
    For any $\mu\in \overline{B_{\PP_2}(\delta_{a_i},r_{a_i})}\backslash B_{\PP_2}(\delta_{a_i},\tilde{r}_{a_i})$, i.e., $\tilde{r}_{a_i}^2\leq \|\mu^{a_i}\|_2^2\leq r_{a_i}^2$, let
    \begin{equation}\label{eq:new1217-1}
        \widetilde{\TT}^i_{\mu}:=\sup\big\{t\geq 0: P_s^*\mu \in \overline{B_{\PP_2}(\delta_{a_i},r_{a_i})}\backslash B_{\PP_2}(\delta_{a_i},\tilde{r}_{a_i}), \ \text{ for all } \ 0\leq s\leq t\big\}.
    \end{equation}
    Set  $T^i:=\frac{1}{\theta^i}(r_{a_i}^2-\tilde{r}_{a_i}^2)+1$,  where by \eqref{eq:tilde-r-a}, $\theta^i:=\inf_{\tilde{r}_{a_i}^2\leq z\leq r_{a_i}^2}g_{a_i}(z,r_{a_i}^2)>0$.  We will show that
    \begin{equation}\label{eq:new1217-2}
        \text{$\widetilde{\TT}^i_{\mu}< T^i$ \ for all \ $\mu\in \overline{B_{\PP_2}(\delta_{a_i},r_{a_i})}\backslash B_{\PP_2}(\delta_{a_i},\tilde{r}_{a_i})$, \ $1\leq i\leq n$.}
    \end{equation}
    Otherwise, there exists a $\mu\in \overline{B_{\PP_2}(\delta_{a_i},r_{a_i})}\backslash B_{\PP_2}(\delta_{a_i},\tilde{r}_{a_i})$ such that
    \begin{equation*}
        P_t^*\mu \in \overline{B_{\PP_2}(\delta_{a_i},r_{a_i})}\backslash B_{\PP_2}(\delta_{a_i},\tilde{r}_{a_i}), \ \text{ for all } \ 0\leq t\leq T^i.
    \end{equation*}
   That is to say,
    \begin{equation}\label{eq:new-shrinking}
        \tilde{r}_{a_i}^2\leq \alpha^i_t\leq r_{a_i}^2, \ \text{ for all } \ 0\leq t\leq T^i.
    \end{equation}
    Applying It\^o's formula to $|X_t^{\mu}-a_i|^2$ on $[0,T^i]$, similar to \eqref{1124-1}, we obtain that
    \begin{equation*}
        \alpha^i_{T^i}\leq r_{a_i}^2-\int_{0}^{T^i}g_{a_i}(\alpha^i_t, r_{a_i}^2)\d t\leq r_{a_i}^2-\theta^i T^i=\tilde{r}_{a_i}^2-\theta^i<\tilde{r}_{a_i}^2,
    \end{equation*}
    which give a contradiction to \eqref{eq:new-shrinking}. Hence,  we obtain \eqref{eq:new1217-2}. Let $T=\max_{1\leq i\leq n}T^i$. Then (ii) follows from \eqref{0111-4}, \eqref{eq:new1217-1} and \eqref{eq:new1217-2}.
\end{proof}

\begin{remark}
The prototype of the locally dissipative condition originates from Zhang \cite{Zhang2023}, where it was introduced to demonstrate the non-uniqueness of invariant measures.
This condition was subsequently applied by Feng-Qu-Zhao \cite{Feng-Qu-Zhao2023b} to time-inhomogeneous McKean–Vlasov SDEs, and by Bao-Wang \cite{Bao-Wang2025} to McKean–Vlasov SDEs with jumps.
Here, building on a detailed analysis of the measure-iterating map $\Psi$ (Proposition \ref{P:Psi-Cts-Cpt-Monotone}), we exploit the locally dissipative condition to derive the existence of multiple order-related invariant measures.
Moreover,  we obtain shrinking neighbourhoods of invariant measures under the semiflow $\str P_t$, which is essential for determining the direction of the connecting orbit in the proof of our main theorem.
\end{remark}

As an example, we illustrate how to fulfill the locally dissipative condition for double-well landscapes.
Other specific examples can be verified in a similar manner.

\begin{proposition}\label{Prop: cubic symmetric}
    We consider the following one-dimensional McKean-Vlasov SDE:
    \begin{equation}\label{eq:0229-1}
    \d X_t=-\left[X_t(X_t-1)(X_t+1)+\xb\left(X_t-\bE X_t\right)\right]\d t+\xs(X_t)\d W_t.
    \end{equation}
    Assume
    \begin{equation}\label{1127-8}
        \beta\geq\frac{27(9+\sqrt{17})}{128} \ \text{ and } \ \sup_{x\in\R}\xs^2(x)< \frac{51\sqrt{17}-107}{256}.
    \end{equation}
    Let $r_1=r_{-1}=\frac{9-\sqrt{17}}{8}$, then the equation \eqref{eq:0229-1} is locally dissipative at $\pm1$ with configurations $(r_{1},g_{1})$ and $(r_{-1},g_{-1})$ respectively, where $g_{1}=g_{-1}$ as defined in \eqref{1127-7}.
\end{proposition}

\begin{proof}
    Let $\overline{\xs}:=\sup_{x\in\R}|\xs(x)|$. It suffices to show that there exist functions $g_{\pm 1}$ such that \eqref{1119-1} and \eqref{1119-2} hold with $a=\pm 1$.

     \noindent\textbf{Case a=1}.
We have for any $x\in \R, \mu\in\PP_2(\R)$,
    \begin{equation*}
        \begin{split}
            2xb(x+1,\mu)+|\xs(x+1)|^2
            &\leq -\big(2|x|^4-6|x|^3+(4+2\beta)|x|^2-2\beta |x|\|\mu^1\|_2-\overline{\xs}^2\big).
        \end{split}
    \end{equation*}
    Define $g_{1}$ as follows: for any $z,w\geq 0$,
    \begin{equation}\label{1127-7}
       g_{1}(z,w)=2z^2-6z^{\frac{3}{2}}+(4+2\beta)z-2\beta z^{\frac{1}{2}}w^{\frac{1}{2}}-\overline{\xs}^2.
    \end{equation}
    Then it is easy to check that \eqref{1119-1} holds. It remains to show that \eqref{1119-2} holds.
    Obviously, $g_1$ is continuous and $g_1(z,\cdot)$ is decreasing for any $z\geq 0$ and hence
    \begin{equation}\label{eq:new1207-1}
       \inf_{0\leq w\leq r^2}g_1(z,w)=g_1(z,r^2), \ \text{ for any } \ r>0.
    \end{equation}
    Note that
    \begin{equation*}
       \frac{\partial g_1(z,w)}{\partial z}=4z-9z^{\frac{1}{2}}+4+2\beta-\beta w^{\frac{1}{2}}z^{-\frac{1}{2}},
    \end{equation*}
    and
    \begin{equation*}
       \frac{\partial^2 g_1(z,w)}{\partial z^2}=4-\frac{9}{2}z^{-\frac{1}{2}}+\frac{1}{2}\beta w^{\frac{1}{2}}z^{-\frac{3}{2}}.
    \end{equation*}
    Since $r_1=\frac{9-\sqrt{17}}{8}$, by calculation, we have
    \begin{equation}\label{1127-6}
       \text{$g_1(\cdot,r_1^2)$ is convex if and only if } \ \beta\geq \frac{27}{16r_1}=\frac{27(9+\sqrt{17})}{128}.
    \end{equation}
     Thus, \eqref{1127-8} entails that $g_1(\cdot,r_1^2)$ is convex. 

    On the other hand, by a simple calculation and \eqref{1127-8}, we have
    \begin{equation*}
       g_{1}(r_1^2,r_1^2)=2r_1^2(r_1-1)(r_1-2)-|\overline{\xs}|^2=\frac{51\sqrt{17}-107}{256}-\overline{\xs}^2>0,
    \end{equation*}
    and
    \begin{equation*}
       \frac{\partial g_1}{\partial z}(r_1^2,r_1^2)=4r_1^2-9r_1+4+\beta=\beta>0.
    \end{equation*}
    Note that $g_1(\cdot,r_1^2)$ is convex, we conclude that
    \begin{equation*}
       \frac{\partial g_1}{\partial z}(z,r_1^2)\geq \frac{\partial g_1}{\partial z}(r_1^2,r_1^2)>0 \ \text{ for all } \ z\geq r_1^2,
    \end{equation*}
    and thus
    \begin{equation}\label{eq:new1207-2}
       g_1(z,r_1^2)\geq g_1(r_1^2,r_1^2)>0 \ \text{ for all } \ z\geq r_1^2.
    \end{equation}
    Combining \eqref{eq:new1207-1}-\eqref{eq:new1207-2}, we obtain \eqref{1119-2}. Hence, equation \eqref{eq:0229-1} is locally dissipative at $1$ with configuration $(r_{1},g_{1})$.

    \noindent\textbf{Case a=-1}. Choose $g_{-1}=g_1$ as in \eqref{1127-7}, similar to that of the case $a=1$, we can also prove that equation \eqref{eq:0229-1} is locally dissipative at $-1$ with configuration $(r_{-1},g_{-1})$.
\end{proof}

\section{Proof of Main Results}\label{Section:Proof-Main-Results}
Before proceeding with the proofs of our main theorems, we put at the beginning the common arguments.
They serve to verify the settings (M1)-(M5) proposed in Section \ref{Section:Results-MDS} with the prior information that $\mu_1<\mu_2$ are two order-related invariant measures of the McKean-Vlasov SDE \eqref{eq:mvsystem}.
Recall that $[\mu_1,\mu_2]_{\PP_2}:=\{\mu\in\PP_2(\R^d): \mu_1\leq\mu\leq\mu_2\}$.

\begin{enumerate}[label=(M\arabic*)]
\item Lemma \ref{lem:m1-norm-space} gives that $(\MM_1(\R^d),\norm{\cdot}_{\WW_1})$ is a normed space, and Lemma \ref{lem:m1-cone} shows that $C$ defined in \eqref{eq:cone} is a cone in $\MM_1(\R^d)$.
\item It is obvious that the order interval $[\mu_1,\mu_2]_{\PP_2}$ is not a singleton and is convex in $\MM_1(\R^d)$.
Under Assumption \ref{asp:lipschitz}, \ref{asp:dissipative-growth-nondegeneracy}, by Lemma \ref{Lemma: im in P_infty} and Lemma \ref{lem:order-interval-bounded}, we see that $[\mu_1,\mu_2]_{\PP_2}$ is bounded in $\PP_2(\R^d)$.
Lemma \ref{lem:compact-embed} and Lemma \ref{lem:stochastic-order-closed} result in the compactness of $[\mu_1,\mu_2]_{\PP_2}$ in $\PP_1(\R^d)$.
As Remark \ref{rem:p1-closed-m1} tells that $\PP_1(\R^d)$ is closed in $\MM_1(\R^d)$, we have $[\mu_1,\mu_2]_{\PP_2}$ is compact in $\MM_1(\R^d)$.
Lemma \ref{Lemma: im in P_infty} and Lemma \ref{lem:order-interval-bounded} entails  that $[\mu_1,\mu_2]_{\PP_2}$ is bounded in $\PP_p(\R^d)$ for all $p\geq1$, and thus a subset of $\PP_{\8}(\R^d)$.
Hence, by Lemma \ref{lem:m1-norm-space} and Lemma \ref{Lem: equavilent converge}, the $2$-Wasserstein metric on $[\mu_1,\mu_2]_{\PP_2}$ induces its relative topology in $\MM_1(\R^d)$.
Thus, $([\mu_1,\mu_2]_{\PP_2},\WW_2)$ is a non-singleton convex compact metric subspace of $(\MM_1(\R^d),\norm{\cdot}_{\WW_1})$.
\item Since the partial order induced by the cone $C$ in the normed space $\MM_1(\R^d)$ coincides with the stochastic order when restricted on $\PP_2(\R^d)$ (Lemma \ref{lem:m1-cone}), the infimum and the supremum of $[\mu_1,\mu_2]_{\PP_2}$ in $\MM_1(\R^d)$ are exactly endpoints $\mu_1$, $\mu_2$.
\item Under Assumption \ref{asp:lipschitz}, \ref{asp:cooperation}, the semigroup $\str P_t$ generated by \eqref{eq:mvsystem} is a monotone semiflow on $[\mu_1,\mu_2]_{\PP_2}$ by Theorem \ref{coro:mvsde-monotone}.
\item $\mu_1$ and $\mu_2$ are invariant measures (equilibria) of $\str P_t$.
\end{enumerate}

\begin{proof}[Proof of Theorem \ref{thm:existence-unstable}]
(i)-(ii). By Theorem \ref{Thm: two ordered im} (i), there exist $n$ order-related invariant measures $\mu_1<\mu_2<\cdots<\mu_n$, satisfying
 $\mu_i\in B_{\PP_2}(\xd_{a_i},r_{a_i})\cap \PP_{\infty}(\R^d)$ for all $i=1,2,\ldots,n$. 
The positive invariance of $B_{\PP_2}(\xd_{a_i},r_{a_i})$ in (i) follows from \eqref{0111-4}. 
On the other hand, Theorem \ref{Thm: two ordered im} (ii)  asserts that
there exists $T>0$ and $0<\tilde{r}_{a_i}<r_{a_i}, \ i=1,2,\cdots,n$ such that for all $t\geq T$,
\begin{equation}\label{E:shrink}
    \text{$\str P_t\overline{B_{\PP_2}(\xd_{a_i},r_{a_i})}\subset B_{\PP_2}(\xd_{a_i},\tilde{r}_{a_i})$, \ \ $i=1,2,\cdots,n$;}
\end{equation}
Now we are going to show \eqref{eq:nu notin ball} and (ii).
Fix any $i\in\{1,2,\cdots,n-1\}$.
Firstly, we give the following claim.

\noindent\textbf{Claim 1.} There exists an invariant measure $\nu_i\in[\mu_i,\mu_{i+1}]_{\PP_2}$ satisfying $\nu_{i}\notin \overline{B_{\PP_2}(\xd_{a_i},r_{a_i})}\bigcup\overline{B_{\PP_2}(\xd_{a_{i+1}},r_{a_{i+1}})}$ and there exists a decreasing connecting orbit $\{\mu_{i,i}(t)\}_{t\in\R}$ from $\nu_i$ to some invariant measure $\hat{\mu}_i\geq\mu_i$.

\noindent\textbf{Proof of Claim 1:} 
Let $E$ denote the set of all invariant measures of the equation \eqref{eq:mvsystem}. 
By (M2), $[\mu_i,\mu_{i+1}]_{\PP_2}$ is a nonempty compact set in $\PP_2(\R^d)$. 
Lemma \ref{L:MCT} (iii) entails that, $E\cap([\mu_i,\mu_{i+1}]_{\PP_2}\backslash B_{\PP_2}(\xd_{a_i},r_{a_i}))$ 
  contains a minimal point $\nu_i$. 
  \eqref{E:shrink} implies that 
  \begin{equation}\label{E:invariant measure not in boundary}
  E\cap \overline{B_{\PP_2}(\xd_{a_{i}},r_{a_{i}})}\subset B_{\PP_2}(\xd_{a_{i}},\tilde{r}_{a_{i}}).
  \end{equation}
Together with $\nu_i\notin B_{\PP_2}(\xd_{a_{i}},r_{a_{i}})$, we have
\begin{equation}\label{E:nu i not in closed ball}
\nu_i\notin\overline{B_{\PP_2}(\xd_{a_i},r_{a_i})}.
\end{equation}
  We \emph{assert} that 
  \begin{equation}\label{E:nui not in upper ball}
    \nu_i\notin\overline{B_{\PP_2}(\xd_{a_{i+1}},r_{a_{i+1}})}.  
  \end{equation}
  Otherwise, assume for contradiction that $\nu_i\in\overline{B_{\PP_2}(\xd_{a_{i+1}},r_{a_{i+1}})}$.
  Again, \eqref{E:shrink} implies that $E\cap \overline{B_{\PP_2}(\xd_{a_{i+1}},r_{a_{i+1}})}\subset B_{\PP_2}(\xd_{a_{i+1}},\tilde{r}_{a_{i+1}})$.
  Hence, $\nu_i\in B_{\PP_2}(\xd_{a_{i+1}},\tilde{r}_{a_{i+1}})$.
  Using again Lemma \ref{L:MCT} (iii), we can take a maximal point $\tilde{\mu}_i$ of the nonempty compact set $E\cap[\mu_i,\nu_i]_{\PP_2}\cap \overline{B_{\PP_2}(\xd_{a_i},r_{a_i})}$.
  \eqref{E:invariant measure not in boundary}
  entails that 
  $\tilde{\mu}_i\in B_{\PP_2}(\xd_{a_i},\tilde{r}_{a_i})$.
  By $r_{a_i}+r_{a_{i+1}}\leq \abs{a_{i+1}-a_i}$, $B_{\PP_2}(\xd_{a_{i}},\tilde{r}_{a_{i}})\cap B_{\PP_2}(\xd_{a_{i+1}},\tilde{r}_{a_{i+1}})=\emptyset$.
  Hence, $\tilde{\mu}_i\neq\nu_i$.
  Together with $\tilde{\mu}_i\in[\mu_i,\nu_i]$, we have $\tilde{\mu}_i<\nu_i$.
  By the choice of $\nu_i$ and $\tilde{\mu}_i$, one has
  \begin{equation}\label{E:no other invariant measures}
  E\cap[\tilde{\mu}_i,\nu_i]_{\PP_2}=\{\tilde{\mu}_i,\nu_i\}.
  \end{equation}
  Now we apply Theorem \ref{T:semiflow-tri} to the order interval $[\tilde{\mu}_i,\nu_i]_{\PP_2}$, 
\eqref{E:no other invariant measures} excludes Theorem \ref{T:semiflow-tri} (a).
However, \eqref{E:shrink} also excludes Theorem \ref{T:semiflow-tri} (b)-(c).
This contradiction completes the proof of \eqref{E:nui not in upper ball}.

By virtue of Lemma \ref{L:MCT} (iii), we can take a maximal point $\hat{\mu}_i$ of the nonempty compact set $E\cap[\mu_i,\nu_i]_{\PP_2}\cap \overline{B_{\PP_2}(\xd_{a_i},r_{a_i})}$.
\eqref{E:invariant measure not in boundary} and \eqref{E:nu i not in closed ball} implies that 
$\hat{\mu}_i\in B_{\PP_2}(\xd_{a_{i}},\tilde{r}_{a_{i}})$ and $\hat{\mu}_i\neq\nu_i$.
Applying Theorem \ref{T:semiflow-tri} to the order interval $[\hat{\mu}_i,\nu_i]_{\PP_2}$, 
the choice of $\nu_i$ and $\tilde{\mu}_i$ excludes Theorem \ref{T:semiflow-tri} (a). 
\eqref{E:shrink}, \eqref{E:nu i not in closed ball}, and $\hat{\mu}_i\in B_{\PP_2}(\xd_{a_{i}},\tilde{r}_{a_{i}})$ excludes Theorem  \ref{T:semiflow-tri} (b). So,  Theorem  \ref{T:semiflow-tri} (c) is valid in $[\hat{\mu}_i,\nu_i]_{\PP_2}$. Together with \eqref{E:nu i not in closed ball} and \eqref{E:nui not in upper ball}, we proved Claim 1.

Now, we are going to show the existence of an increasing connecting orbit in the order interval $[\nu_i,\mu_{i+1}]_{\PP_2}$. 
By Lemma \ref{L:MCT} (iii), we can first take a maximal point $\bar{\nu}_i$ of the nonempty compact set $E \cap\bigl([\nu_i,\mu_{i+1}]_{\PP_2}
        \setminus B_{\PP_2}(\xd_{a_{i+1}},r_{a_{i+1}})\bigr),$
and then a minimal point $\hat{\mu}_{i+1}$ of 
$E \cap [\bar{\nu}_i,\mu_{i+1}]_{\PP_2}  
\cap \overline{B_{\PP_2}(\xd_{a_{i+1}},r_{a_{i+1}})}.$
Applying Theorem  \ref{T:semiflow-tri} on the order interval $[\bar{\nu}_i,\hat{\mu}_{i+1}]_{\PP_2}$, 
and using arguments entirely analogous to those in Claim 1, 
we obtain an increasing connecting orbit from $\bar{\nu}_i$ to $\hat{\mu}_{i+1}$.
Since $[\bar{\nu}_i,\hat{\mu}_{i+1}]_{\PP_2}\subset [\nu_i,\mu_{i+1}]_{\PP_2}$, the increasing connecting orbit  also lies entirely within $[\nu_i,\mu_{i+1}]_{\PP_2}$.

(iii). From the definition, the forward and backward limits of any connecting orbit 
must be invariant measures. 
Therefore, (iii) is a direct consequence of (ii).
Hence, we finish the proof of Theorem \ref{thm:existence-unstable}.
\end{proof}

\begin{proof}[Proof of Theorem \ref{thm:double-well}]
By Proposition \ref{Prop: cubic symmetric}, the equation \eqref{eq:double-well} is locally dissipative at $\pm1$.
By Tugaut \cite[Theorem 2.1]{Tugaut2014} (see also Dawson \cite{Dawson1983}),  \eqref{eq:double-well} has at most three invariant measures, which means exactly $\mu_{-1},\mu_0,\mu_1$, and this proves (i).
Then the result directly follows from Theorem \ref{thm:existence-unstable}.
\end{proof}

\begin{proof}[Proof of Theorem \ref{thm:multi-well}]
Take
\begin{align*}
g_0(z,w)&=2z^3-10z^2+(8+2\xb)z-2\xb z^{\frac12}w^{\frac12}-\overline{\xs}^2,& r_0&=\frac{\sqrt{15-3\sqrt{13}}}{3},\\
g_2(z,w)&=2z^3-20z^{\frac52}+70z^2-100z^{\frac32}+(48+2\xb)z-2\xb z^{\frac12}w^{\frac12}-\xs^2,& r_2&=r_0,\\
g_{-2}(z,w)&=g_2(z,w),& r_{-2}&=r_0,
\end{align*}
 where $\overline{\xs}:=\sup_{x\in\R}|\xs(x)|$.
Through a similar calculation in Proposition \ref{Prop: cubic symmetric}, locally dissipative condition holds at $0, \pm2$ for the equation \eqref{eq:multi-well}. Then Theorem \ref{thm:existence-unstable} gives the result. 
\end{proof}

\begin{proof}[Proof of Theorem \ref{thm:double-well-perturbation}]
Take
\begin{align*}
g_1(z,w)&=2z^2-6z^{\frac32}+(4+2\xb)z-\frac23z^{\frac12}-2\xb z^{\frac12}w^{\frac12}-\overline{\xs}^2,& r_1&=\frac{\sqrt5}5,\\
g_{-1}(z,w)&=g_1(z,w),& r_{-1}&=r_1,
\end{align*}
 where $\overline{\xs}:=\sup_{x\in\R}|\xs(x)|$.
Similar to the calculation in Proposition \ref{Prop: cubic symmetric}, locally dissipative conditions hold at $\pm1$ for the equation \eqref{eq:double-well-perturbation}.
The result follows from Theorem \ref{thm:existence-unstable}.
\end{proof}

\begin{proof}[Proof of Theorem \ref{thm:high-dimension}]
    For any $\mu\in\PP_2(\R^2)$, denote by $\mu_x$ and $\mu_y$ the marginal distributions of $\mu$ with respect to the first ($x$-) and the second ($y$-) component, respectively.
    Note that equation \eqref{eq:higher-dimensional} can be written as
    \begin{equation*}
        \d (X_t,Y_t)^{\top}=b\big((X_t,Y_t),\LL((X_t,Y_t))\big)\d t+\Sigma \d W_t,
    \end{equation*}
    where for all $(x,y)\in\mathbb{R}^2, \ \mu\in\PP_2(\R^2)$,
\begin{equation*}
    b((x,y),\mu)=
    \begin{pmatrix}
     x-x^3-\tau\alpha\Big(x-\int_{\mathbb{R}}z\d\mu_x(z)\Big)-(1-\tau)\beta\Big(x-\int_{\mathbb{R}}z\d\mu_y(z)\Big) \\
    y-y^3-\tau\alpha\Big(y-\int_{\mathbb{R}}z\d\mu_x(z)\Big)-(1-\tau)\beta\Big(y-\int_{\mathbb{R}}z\d\mu_y(z)\Big)
    \end{pmatrix}
\end{equation*}
and
\begin{equation*}
    \Sigma=
    \begin{pmatrix}
     \xs & 0  \\
     0& \xs
    \end{pmatrix}, \ \ 
    W_t=\begin{pmatrix}
     W^1_t \\
     W^2_t 
    \end{pmatrix}.
\end{equation*}
Then we show that equation \eqref{eq:higher-dimensional} with the parameter range in \eqref{eq:parameter-high-dimensional} is locally dissipative at $(1,1), \ (-1,-1)$ with radius $r_{(1,1)}=r_{(-1,-1)}=\frac{1}{4}$.

     \noindent\textbf{Case $a=(1,1)$}.
Note that for any $(x,y)\in \R^2, \mu\in\PP_2(\R^2)$,
    \begin{align*}
        &\no{=}2\big\langle (x,y), b((x+1,y+1),\mu)\big\rangle+\|\Sigma\|_2^2\\
        &=-2x^4-6x^3-(4+2\tau\xa+2(1-\tau)\beta)x^2\\
        &\quad\quad\quad\quad\quad\quad\quad \ +2\tau\xa x\int_{\R}z\d\mu_x^{(1,1)}(z)+2(1-\tau)\xb x\int_{\R}z\d\mu_y^{(1,1)}(z)\\
        &\ \ \ \, -2y^4-6y^3-(4+2\tau\xa+2(1-\tau)\beta)y^2\\
        &\quad\quad\quad\quad\quad\quad\quad \ +2\tau\xa y\int_{\R}z\d\mu_x^{(1,1)}(z)+2(1-\tau)\xb y\int_{\R}z\d\mu_y^{(1,1)}(z)+2\xs^2\\
        &\leq -|(x,y)|^4+6|(x,y)|^3-(4+2\tau\alpha+2(1-\tau)\xb)|(x,y)|^2\\
        &\ \ \ \ +2\tau\xa(x+y)\int_{\R}z\d\mu_x^{(1,1)}(z)+2(1-\tau)\xb(x+y)\int_{\R}z\d\mu_y^{(1,1)}(z)+2\xs^2\\
        &\leq -|(x,y)|^4+6|(x,y)|^3-(4+2\tau\alpha+2(1-\tau)\xb)|(x,y)|^2\\
        &\ \ \ \ +2\max\{2\tau\xa,2(1-\tau)\xb\}|x+y|\bigg(\bigg|\int_{\R}z\d\mu_x^{(1,1)}(z)\bigg|+\bigg|\int_{\R}z\d\mu_y^{(1,1)}(z)\bigg|\bigg)+2\xs^2\\
        &\leq -\Big(|(x,y)|^4-6|(x,y)|^3+(4+2\tau\alpha+2(1-\tau)\xb)|(x,y)|^2\\
        &\ \ \ \quad\quad\quad\quad\quad\quad\quad \ -2\max\{2\tau\xa, 2(1-\tau)\xb\}|(x,y)|\|\mu^{(1,1)}\|_2-2\xs^2\Big),
    \end{align*}
    where we have used the fact that
    \[
    \int_{\R}z\d\mu^{(1,1)}_x(z)=\int_{\R}z\d\mu_x(z)-1, \ \ \ \ \int_{\R}z\d\mu^{(1,1)}_y(z)=\int_{\R}z\d\mu_y(z)-1,
    \]
    and
    \begin{equation*}
        \begin{split}
            |x+y|\Big(\Big|\int_{\R}z\d\mu_x(z)\Big|+\Big|\int_{\R}z\d\mu_y(z)\Big|\Big)
            &\leq 2|(x,y)|\Big(\int_{\R}|z|^2\d\mu_x(z)+\int_{\R}|z|^2\d\mu_y(z)\Big)^{\frac{1}{2}}\\
            &=2|(x,y)|\|\mu\|_2.
        \end{split}
    \end{equation*}
    Hence we choose $g_{(1,1)}$ as follows: for any $z,w\geq 0$,
    \begin{equation}\label{eq:pf-high-dimensional-1}
       g_{(1,1)}(z,w)=z^2-6z^{\frac{3}{2}}+(4+2\tau\alpha+2(1-\tau)\beta)z-2\max\{2\tau\xa, 2(1-\tau)\xb\} z^{\frac{1}{2}}w^{\frac{1}{2}}-2\xs^2.
    \end{equation}
    It is easy to verify 
    \begin{equation}\label{eq:1207-1}
        2\big\langle (x,y), b((x+1,y+1),\mu)\big\rangle+\|\Sigma\|_2^2\leq -g_{(1,1)}\big(|(x,y)|^2, \|\mu^{(1,1)}\|_2^2\big).
    \end{equation}
    Obviously, $g_{(1,1)}$ is continuous and $g_{(1,1)}(z,\cdot)$ is decreasing for any $z\geq 0$ and hence
    \begin{equation}\label{eq:1207-2}
       \inf_{0\leq w\leq r_{(1,1)}^2}g_{(1,1)}(z,w)=g_{(1,1)}(z,r_{(1,1)}^2).
    \end{equation}
    Note that
    \begin{equation*}
       \frac{\partial g_{(1,1)}(z,w)}{\partial z}=2z-9z^{\frac{1}{2}}+4+2\tau\alpha+2(1-\tau)\beta-\max\{2\tau\xa, 2(1-\tau)\xb\} w^{\frac{1}{2}}z^{-\frac{1}{2}},
    \end{equation*}
    and
    \begin{equation*}
       \frac{\partial^2 g_{(1,1)}(z,w)}{\partial z^2}=2-\frac{9}{2}z^{-\frac{1}{2}}+\max\{\tau\xa, (1-\tau)\xb\} w^{\frac{1}{2}}z^{-\frac{3}{2}}.
    \end{equation*}
    Note that $r_{(1,1)}=\frac{1}{4}$, then by calculation, we have 
    \begin{equation}\label{eq:pf-high-dimensional-2}
       \text{$g_{(1,1)}\big(\cdot,r_{(1,1)}^2\big)$ is convex if and only if  } \ \max\{\tau\xa, (1-\tau)\xb\}\geq \frac{27}{8r_{(1,1)}}=\frac{27}{2}.
    \end{equation}
    Then \eqref{eq:parameter-high-dimensional} and \eqref{eq:pf-high-dimensional-2} entail that $g_{(1,1)}(\cdot,r_{(1,1)}^2)$ is convex. Moreover, it follows from \eqref{eq:parameter-high-dimensional} that
    \begin{equation*}
       g_{(1,1)}\big(r_{(1,1)}^2,r_{(1,1)}^2\big)=r_{(1,1)}^4-6r_{(1,1)}^3+4r_{(1,1)}^2-2|\tau\xa-(1-\tau)\xb|r_{(1,1)}^2-2\xs^2\geq \frac{9}{256}-2\xs^2>0, \ \ 
    \end{equation*}
    and
    \begin{equation*}
        \frac{\partial g_{(1,1)}}{\partial z}\big(r_{(1,1)}^2,r_{(1,1)}^2\big)=2r_{(1,1)}^2-9r_{(1,1)}+4+\min\{2\tau\xa, 2(1-\tau)\xb\}=\frac{15}{8}+\min\{2\tau\xa, 2(1-\tau)\xb\}>0.
    \end{equation*}
    Since $g_{(1,1)}(\cdot,r_{(1,1)}^2)$ is convex, we have
    \begin{equation*}
       \frac{\partial g_{(1,1)}}{\partial z}\big(z,r_{(1,1)}^2\big)\geq \frac{\partial g_{(1,1)}}{\partial z}\big(r_{(1,1)}^2,r_{(1,1)}^2\big)>0 \ \text{ for all } \ z\geq r_{(1,1)}^2,
    \end{equation*}
    and thus
    \begin{equation}\label{eq:1207-3}
       g_{(1,1)}\big(z,r_{(1,1)}^2\big)\geq g_{(1,1)}\big(r_{(1,1)}^2,r_{(1,1)}^2\big)>0 \ \text{ for all } \ z\geq r_{(1,1)}^2.
    \end{equation}
    Combining \eqref{eq:1207-1}-\eqref{eq:1207-3}, we conclude that equation \eqref{eq:higher-dimensional} is locally dissipative at $(1,1)$ with radius $r_{(1,1)}=\frac{1}{4}$. 

    \noindent\textbf{Case $a=(-1,-1)$}. Choose $g_{(-1,-1)}=g_{(1,1)}$ as in \eqref{eq:pf-high-dimensional-1} and $r_{(-1,-1)}=\frac{1}{4}$. By the same argument as in the case of $a=(1,1)$, we an prove equation \eqref{eq:higher-dimensional} is also locally dissipative at $(-1,-1)$ with radius $r_{(-1,-1)}=\frac{1}{4}$. 

    According to \cite[Theorem 3.2]{Duong-Pavliotis-Tugaut2025}, the equation \eqref{eq:higher-dimensional} has at most three invariant measures.
    Then the result follows directly from Theorem \ref{thm:existence-unstable}.
\end{proof}

 \section*{Acknowledgement}
 We acknowledge the financial supports of EPSRC grant ref (ref. EP/S005293/2), Royal Society through the Award of
 Newton International Fellowship (ref. NIF/R1/221003), National Natural Science Foundation of China (No. 12501184, No. 12171280), and CSC of China (No. 202206340035),  the
		Postdoctoral Fellowship Program and China Postdoctoral
		Science Foundation under Grant Number BX20250067, and the China Postdoctoral Science
		Foundation under Grant Number 2025M773074.

\begin{appendices}

\section{Proof of Lemma \ref{T:map-tri}}\label{Appendix}
In the appendix, we  finish the proof of Lemma \ref{T:map-tri}. 
Thanks to the fixed point index lemma for metrizable convex compact subsets of Hausdorff locally convex topological vector spaces (Lemma \ref{L:fixed pt index}), the proof of Lemma \ref{T:map-tri} follows exactly the same  argument of Dancer–Hess \cite[Proposition 1]{DH91} for monotone mappings on compact order intervals in Banach spaces.
We give the detail here for the sake of completeness. 
Now, we fix the following settings.

\begin{enumerate}[label=(H\arabic*)]
\item $(V,\mathcal{T})$ is a Hausdorff locally convex topological vector space;
\item $(S,d)$ is a convex compact metric subspace of $(V,\mathcal{T})$, where $d$ is a metric on $S$ inducing the relative topology on $S$.
\end{enumerate}

Hereafter, for any subset $G\subset S$, the closure  $\overline G$ denotes the closure  of $G$ relative to the topology on $S$. In order to prove Lemma \ref{T:map-tri}, we need the following classical fixed point index lemma (see e.g., \cite{N71, N93}).

\begin{lemma}\label{L:fixed pt index}
Assume that (H1)-(H2) hold. Then there exists an integer valued function $i(f,G)$ defined for any relatively open subset $G\subset S$ and continuous map $f:\overline G\to S$ with no fixed point in $\overline G\backslash G$, satisfying
\begin{enumerate}[label=\textnormal{(\roman*)}]
    \item (Additivity). If $G=S$, and $G_1$, $G_2$ are relatively open in $S$, $G_1\cap G_2=\emptyset$ and all fixed points of $f$ lie in $G_1\cup G_2$, then $i(f,S)=i(f,G_1)+i(f,G_2)$;
    \item (Homotopy Invariance). If $F:\overline G\times[0,1]\to S$ is a continuous map, and $F_{\lambda}(x):=F(x,\lambda)$ has no fixed point in $\overline G\backslash G$ for all $\lambda\in[0,1]$, then $i(F_0,G)=i(F_1,G)$;
    \item (Normalisation). If there exists $y\in G$ such that $f(x)=y$ for all $x\in\overline G$, then $i(f,G)=1$.
\end{enumerate}
\end{lemma}

Besides (H1)-(H2), assume further there is a cone $C\subset V$, which induces a closed partial order relation $\leq$ on $V$ (as we introduced in Section \ref{Subsection:Properties-Order}). A point   $x\in S$ is called a (strict) \emph{subsolution} of a mapping $\Psi:S\to S$, if $\Psi(x)\geq x$ ($\Psi(x)>x$). Similarly, $x\in S$ is said to be a (strict) \emph{supersolution} of $\Psi:S\to S$, if $\Psi(x)\leq x$ ($\Psi(x)<x$).

\begin{proof}[Proof of Lemma \ref{T:map-tri}]
Assume that there is no further fixed point of $\Psi$ distinct from $p=\inf S, q=\sup S$ in $S$. We are going to prove (b) or (c) holds.

Define maps 
$$F_{\lambda}(x):=\lambda\Psi(x)+(1-\lambda)p,\ \ x\in S$$ 
and 
$$\tilde{F}_{\lambda}(x):=\lambda\Psi(x)+(1-\lambda)q,\ \ x\in S$$ 
for any $\lambda\in[0,1]$. Clearly, $F$ and $\tilde{F}$ are continuous maps from $S\times[0,1]$ to $S$. Noticing that $F_1=\tilde{F}_1=\Psi$, $F_0(x)=p$ and $\tilde{F}_0(x)=q$ for any $x\in S$.

We claim that, if $F_{\lambda}(x)=x$ for some $x\in S\backslash\{p,q\}$, $\lambda\in[0,1]$, then 
\begin{equation}\label{E:homo-sub}
 \Psi(x)>x.  
\end{equation}
Similarly, if $\tilde{F}_{\lambda}(x)=x$ for some $x\in S\backslash\{p,q\}$, $\lambda\in[0,1]$, then 
\begin{equation}\label{E:homo-sup}
\Psi(x)<x.  
\end{equation}
We only prove the former case, as the latter one is same. In fact, if  $F_{\lambda}(x)=x$ for some $x\in S\backslash\{p,q\}$, then one has $\lambda\neq0,1$. By 
$\lambda(\Psi(x)-x)=(1-\lambda)(x-p),$
we have $\Psi(x)>x$, \eqref{E:homo-sub} is proved.

Let $r>0$ be such that $r<\frac{1}{2}d(p,q)$. Define $B_S(y,\epsilon)=\{x\in S:d(x,y)<\epsilon\}$, for any $y\in S$ and $\epsilon>0$. $\overline{B_S(y,\epsilon)}$ and $\partial B_S(y,\epsilon)$ means the closure and boundary of $B_S(y,\epsilon)$ in $S$. Now, we give the following claim.

\noindent\textbf{Claim 1.} Either there exists a strict subsolution $x_{\epsilon}$ on $\partial B_S(p,\epsilon)$ for any $0<\epsilon<r$, or else there exists a strict supersolution $x_{\epsilon}$ on $\partial B_S(q,\epsilon)$ for any $0<\epsilon<r$.

\noindent\textbf{Proof of Claim 1:} If there exists $0<\epsilon_0<r$ such that there is no strict supersolution on $\partial B_S(q,\epsilon_0)$. By \eqref{E:homo-sup}, $\tilde{F}_{\lambda}(x)\neq x$, for any $x\in\partial B_S(q,\epsilon_0)$, $\lambda\in[0,1]$. Considering the continuous map $\tilde{F}$ on $\overline{B_S(q,\epsilon_0)}\times[0,1]$, Lemma \ref{L:fixed pt index} (ii)(iii) entail that 
\begin{equation*}
i(\Psi,B_S(q,\epsilon_0))=i(\tilde{F}_1,B_S(q,\epsilon_0))=i(\tilde{F}_0,B_S(q,\epsilon_0))=1.
\end{equation*}
Similarly, we also have $i(\Psi,S)=1$. By Lemma \ref{L:fixed pt index} (i),
\begin{equation*}
i(\Psi,S)=i(\Psi,B_S(q,\epsilon_0))+i(\Psi,B_S(p,\epsilon)),
\end{equation*}
which implies that $i(\Psi,B_S(p,\epsilon))=0$, for any $0<\epsilon<r$.
Suppose on the contrary that, there exists  $0<\epsilon_1<r$ such that there is no strict subsolution on $\partial B_S(p,\epsilon_1)$. By \eqref{E:homo-sub}, $F_{\lambda}(x)\neq x$, for any $x\in\partial B_S(p,\epsilon_1)$, $\lambda\in[0,1]$. Considering the continuous map $F$ on $\overline{B_S(p,\epsilon_1)}\times[0,1]$, Lemma \ref{L:fixed pt index} (ii)(iii) entail that 
\begin{equation}\label{E:F0 index 0}
i(F_0,B_S(p,\epsilon_1))=i(F_1,B_S(p,\epsilon_1))=i(\Psi,B_S(p,\epsilon_1))=0.
\end{equation}
Recall that $F_0(x)=p$, for any $x\in\overline{B_S(p,\epsilon_1)}$. Thus, \eqref{E:F0 index 0} contradicts Lemma \ref{L:fixed pt index} (iii). Hence, we obtain Claim 1.

Without loss of generality, we assume that the first case in Claim 1 holds, that is, there exists a strict subsolution $x_{\epsilon}$ on $\partial B_S(p,\epsilon)$ for any $0<\epsilon<r$. Then, we can choose a sequence $\{x_k\}_{k\geq 1}$ in $S$ such that
$x_k\to p$ as $k\to\infty$, and 
$$p<x_k<\Psi(x_k)\leq\Psi^2(x_k)\leq\Psi^3(x_k)\leq\cdots.$$
Since there is no further fixed point of $\Psi$ distinct from $p, q$ in $S$, it follows from Lemma \ref{L:MCT} (iv) that for each $k\geq 1$, 
\begin{equation}\label{E:Psinxktoq}
\Psi^n(x_k)\to q, \ \text{as}\  n\to\infty. 
\end{equation}
By continuity of $\Psi$ and $p$ being a fixed point of $\Psi$, there exists $\delta_i>0$ for $i\geq0$ such that
$$r>\delta_0>\delta_1>\delta_2>\cdots\to0,$$
and
\begin{equation}\label{E:cts-control}
   \Psi(B_S(p,\delta_i))\subset B_S(p,\delta_{i-1}) \text{ for any  } i\geq 1.
\end{equation}
\eqref{E:Psinxktoq} and \eqref{E:cts-control} entails that, for any $x_i\in B_S(p,\delta_i)$ with $i>1$, there exists $j(i)\geq i-1$ such that $y_i:=\Psi^{j(i)}(x_i)\in  B_S(p,\delta_{0})\backslash B_S(p,\delta_1)$. Since $S$ is compact, $y_i$ has a subsequence $y_{i^\prime}$  converging to some point $u_0\in\overline{B_S(p,\delta_{0})}\backslash B_S(p,\delta_1)$. Since $y_i\leq\Psi(y_i)$, we have 
\begin{equation}\label{E:posi-seq}
p<u_0\leq\Psi(u_0)\leq\Psi^2(u_0)\leq\cdots.
\end{equation}
Together with the fact that there is no further fixed point of $\Psi$ distinct from $p, q$ in $S$, Lemma \ref{L:MCT} (iv) implies that $\Psi^n(u_0)\to q$, as $n\to\infty$.

Since $S$ is compact and $j(i')\geq i'-1$, one has a subsequence $\Psi^{j(i'')-1}(x_{i''})$ of $\Psi^{j(i')-1}(x_{i'})$ converging to some point $u_{-1}$ such that $\Psi(u_{-1})=u_0$. And $\Psi^{j(i'')-1}(x_{i''})\leq \Psi^{j(i'')}(x_{i''})$ entails that $u_{-1}\leq u_0$. Since $S$ is compact and $j(i)\geq i-1$, recursively, we get an increasing negative orbit $\{u_{i}\}_{i\in\mathbb{Z}_-}$.
That is to say, 
\[
\cdots\leq u_{-2}\leq u_{-1}\leq u_0 \textnormal{ and }\Psi(u_{i})=u_{i+1},\textnormal{ for any }i\leq-1.
\]
Lemma \ref{L:MCT} (iv) guarantees that $u_{i}$ converges to a fixed point of $\Psi$, as $i\to-\infty$. Recall that $u_0\in\overline{B_S(p,\delta_{0})}$ and $\delta_0<r<\frac{1}{2}d(p,q)$, which entails that $u_0<q$. Since there is no further fixed point of $\Psi$ in $S$ except $p,q$, we have $u_{i}$ converges to $p$ as $i\to-\infty$. Therefore, togethter with \eqref{E:posi-seq}, we have proved Lemma \ref{T:map-tri} (b).
\end{proof}

\end{appendices}

\clearpage
\phantomsection
\addcontentsline{toc}{section}{References}
\linespread{1.0}
\selectfont
\bibliographystyle{siam}
\bibliography{reference}

@article{Duong-Pavliotis-Tugaut2025,
  title={Multi-species McKean-Vlasov dynamics in non-convex landscapes},
  author={Duong, Manh Hong and Pavliotis, Grigorios A and Tugaut, Julian},
  journal={arXiv preprint arXiv:2507.07617},
  year={2025}
}

@article{Bao-Wang2025,
  title={Stationary distributions of McKean-Vlasov SDEs with jumps: existence, uniqueness, and multiplicity},
  author={Bao, Jianhai and Wang, Jian},
  journal={arXiv preprint arXiv:2504.15898},
  year={2025}
}

@article {Ahmed-Ding1993,
    AUTHOR = {Ahmed, N. U. and Ding, Xinhong},
     TITLE = {On invariant measures of nonlinear {M}arkov processes},
   JOURNAL = {J. Appl. Math. Stochastic Anal.},
  FJOURNAL = {Journal of Applied Mathematics and Stochastic Analysis},
    VOLUME = {6},
      YEAR = {1993},
    NUMBER = {4},
     PAGES = {385--406},
      ISSN = {1048-9533,1687-2177},
   MRCLASS = {60J25 (60H10)},
  MRNUMBER = {1255258},
MRREVIEWER = {Constantin\ Tudor},
       DOI = {10.1155/S1048953393000310},
       URL = {https://doi.org/10.1155/S1048953393000310},
}

@article {Delgadino-Gvalani-Pavliotis2021,
    AUTHOR = {Delgadino, Matias G. and Gvalani, Rishabh S. and Pavliotis,
              Grigorios A.},
     TITLE = {On the diffusive-mean field limit for weakly interacting
              diffusions exhibiting phase transitions},
   JOURNAL = {Arch. Ration. Mech. Anal.},
  FJOURNAL = {Archive for Rational Mechanics and Analysis},
    VOLUME = {241},
      YEAR = {2021},
    NUMBER = {1},
     PAGES = {91--148},
      ISSN = {0003-9527},
   MRCLASS = {82C22 (60J60 82C31)},
  MRNUMBER = {4271956},
MRREVIEWER = {Julien Reygner},
       DOI = {10.1007/s00205-021-01648-1},
       URL = {https://doi.org/10.1007/s00205-021-01648-1},
}

@article {Carrillo-Gvalani-Pavliotis-Schlichting2020,
    AUTHOR = {Carrillo, J. A. and Gvalani, R. S. and Pavliotis, G. A. and
              Schlichting, A.},
     TITLE = {Long-time behaviour and phase transitions for the
              {M}c{K}ean-{V}lasov equation on the torus},
   JOURNAL = {Arch. Ration. Mech. Anal.},
  FJOURNAL = {Archive for Rational Mechanics and Analysis},
    VOLUME = {235},
      YEAR = {2020},
    NUMBER = {1},
     PAGES = {635--690},
      ISSN = {0003-9527},
   MRCLASS = {35R60 (35B40 35K59 60K35)},
  MRNUMBER = {4062483},
MRREVIEWER = {Alp O. Eden},
       DOI = {10.1007/s00205-019-01430-4},
       URL = {https://doi.org/10.1007/s00205-019-01430-4},
}

@article{Feng-Qu-Zhao2023,
      title={Entrance measures for semigroups of time-inhomogeneous {SDE}s: possibly degenerate and expanding}, 
      author={Chunrong Feng and Baoyou Qu and Huaizhong Zhao},
      year={2023},
      journal={arXiv: 2307.07891},
      archivePrefix={arXiv},
      primaryClass={math.PR}
}

@article{Feng-Qu-Zhao2023b,
      title={{Entrance measures for time-inhomogeneous McKean-Vlasov stochastic differential equations}}, 
      author={Chunrong Feng and Baoyou Qu and Huaizhong Zhao},
      year={2023},
     journal={Preprint},
}

@incollection {Hairer-Mattingly2011,
    AUTHOR = {Hairer, Martin and Mattingly, Jonathan C.},
     TITLE = {Yet another look at {H}arris' ergodic theorem for {M}arkov
              chains},
 BOOKTITLE = {Seminar on {S}tochastic {A}nalysis, {R}andom {F}ields and
              {A}pplications {VI}},
    SERIES = {Progr. Probab.},
    VOLUME = {63},
     PAGES = {109--117},
 PUBLISHER = {Birkh\"{a}user/Springer Basel AG, Basel},
      YEAR = {2011},
   MRCLASS = {60J05 (37A30 37A50 47D07)},
  MRNUMBER = {2857021},
MRREVIEWER = {Wojciech Bartoszek},
       DOI = {10.1007/978-3-0348-0021-1\_7},
       URL = {https://doi.org/10.1007/978-3-0348-0021-1_7},
}

@article {panpanren22,
    AUTHOR = {Ren, Panpan},
     TITLE = {Order preservation and positive correlation for nonlinear
              {F}okker-{P}lanck equation},
   JOURNAL = {Electron. Commun. Probab.},
  FJOURNAL = {Electronic Communications in Probability},
    VOLUME = {27},
      YEAR = {2022},
     PAGES = {Paper No. 26, 12},
      ISSN = {1083-589X},
   MRCLASS = {60J60 (60H30)},
  MRNUMBER = {4424032},
MRREVIEWER = {Ren\ Ming\ Song},
       DOI = {10.1214/22-ecp466},
       URL = {https://doi.org/10.1214/22-ecp466},
}

@book {Villani2009,
    AUTHOR = {Villani, C\'{e}dric},
     TITLE = {Optimal transport},
    SERIES = {Grundlehren der mathematischen Wissenschaften [Fundamental
              Principles of Mathematical Sciences]},
    VOLUME = {338},
      NOTE = {Old and new},
 PUBLISHER = {Springer-Verlag, Berlin},
      YEAR = {2009},
     PAGES = {xxii+973},
      ISBN = {978-3-540-71049-3},
   MRCLASS = {49-02 (28A75 37J50 49Q20 53C23 58E30)},
  MRNUMBER = {2459454},
MRREVIEWER = {Dario Cordero-Erausquin},
       DOI = {10.1007/978-3-540-71050-9},
       URL = {https://doi.org/10.1007/978-3-540-71050-9},
}

@article {Wang18,
    AUTHOR = {Wang, Feng-Yu},
     TITLE = {Distribution dependent {SDE}s for {L}andau type equations},
   JOURNAL = {Stochastic Process. Appl.},
  FJOURNAL = {Stochastic Processes and their Applications},
    VOLUME = {128},
      YEAR = {2018},
    NUMBER = {2},
     PAGES = {595--621},
      ISSN = {0304-4149},
   MRCLASS = {60J75 (47G20 60G52 60H10)},
  MRNUMBER = {3739509},
MRREVIEWER = {Hong Zhang},
       DOI = {10.1016/j.spa.2017.05.006},
       URL = {https://doi.org/10.1016/j.spa.2017.05.006},
}

@article{Alecio2023,
title={Phase transitions of {McKean-Vlasov SDEs} in Multi-well Landscapes}, 
author={Alexander Alecio},
year={2023},
journal={arXiv: 2307.16846},
archivePrefix={arXiv},
primaryClass={math.PR}
}

@article {Tugaut2013,
    AUTHOR = {Tugaut, Julian},
     TITLE = {Convergence to the equilibria for self-stabilizing processes
              in double-well landscape},
   JOURNAL = {Ann. Probab.},
  FJOURNAL = {The Annals of Probability},
    VOLUME = {41},
      YEAR = {2013},
    NUMBER = {3A},
     PAGES = {1427--1460},
      ISSN = {0091-1798,2168-894X},
   MRCLASS = {60H10 (35B40 35K55 60B10 60J60)},
  MRNUMBER = {3098681},
MRREVIEWER = {Wilhelm\ Stannat},
       DOI = {10.1214/12-AOP749},
       URL = {https://doi.org/10.1214/12-AOP749},
}

@article {Tugaut2014,
    AUTHOR = {Tugaut, Julian},
     TITLE = {Phase transitions of {M}c{K}ean-{V}lasov processes in
              double-wells landscape},
   JOURNAL = {Stochastics},
  FJOURNAL = {Stochastics. An International Journal of Probability and
              Stochastic Processes},
    VOLUME = {86},
      YEAR = {2014},
    NUMBER = {2},
     PAGES = {257--284},
      ISSN = {1744-2508,1744-2516},
   MRCLASS = {60H10 (60G10 60J60 65C05 82C22)},
  MRNUMBER = {3180036},
MRREVIEWER = {Patr\'{\i}cia\ Gon\c{c}alves},
       DOI = {10.1080/17442508.2013.775287},
       URL = {https://doi.org/10.1080/17442508.2013.775287},
}

@article {Hiai2018,
    AUTHOR = {Hiai, Fumio and Lawson, Jimmie and Lim, Yongdo},
     TITLE = {The stochastic order of probability measures on ordered metric
              spaces},
   JOURNAL = {J. Math. Anal. Appl.},
  FJOURNAL = {Journal of Mathematical Analysis and Applications},
    VOLUME = {464},
      YEAR = {2018},
    NUMBER = {1},
     PAGES = {707--724},
      ISSN = {0022-247X,1096-0813},
   MRCLASS = {60E15 (28C15)},
  MRNUMBER = {3794112},
MRREVIEWER = {Hans\ Weber},
       DOI = {10.1016/j.jmaa.2018.04.038},
       URL = {https://doi.org/10.1016/j.jmaa.2018.04.038},
}

@article {Fritz2020,
    AUTHOR = {Fritz, Tobias and Perrone, Paolo},
     TITLE = {Stochastic order on metric spaces and the ordered
              {K}antorovich monad},
   JOURNAL = {Adv. Math.},
  FJOURNAL = {Advances in Mathematics},
    VOLUME = {366},
      YEAR = {2020},
     PAGES = {107081, 46},
      ISSN = {0001-8708,1090-2082},
   MRCLASS = {60B05 (18C15 28A33 54E35 54F05)},
  MRNUMBER = {4070306},
       DOI = {10.1016/j.aim.2020.107081},
       URL = {https://doi.org/10.1016/j.aim.2020.107081},
}

@article {Geib1994,
    AUTHOR = {Geib , Christel and Manthey, Ralf},
     TITLE = {Comparison theorems for stochastic differential equations in
              finite and infinite dimensions},
   JOURNAL = {Stochastic Process. Appl.},
  FJOURNAL = {Stochastic Processes and their Applications},
    VOLUME = {53},
      YEAR = {1994},
    NUMBER = {1},
     PAGES = {23--35},
      ISSN = {0304-4149,1879-209X},
   MRCLASS = {60H15 (60H10)},
  MRNUMBER = {1290705},
MRREVIEWER = {P.\ Kotelenez},
       DOI = {10.1016/0304-4149(94)90055-8},
       URL = {https://doi.org/10.1016/0304-4149(94)90055-8},
}

@article {Lindvall1999,
    AUTHOR = {Lindvall, Torgny},
     TITLE = {On {S}trassen's theorem on stochastic domination},
   JOURNAL = {Electron. Comm. Probab.},
  FJOURNAL = {Electronic Communications in Probability},
    VOLUME = {4},
      YEAR = {1999},
     PAGES = {51--59},
      ISSN = {1083-589X},
   MRCLASS = {60B05 (60E15 60J05)},
  MRNUMBER = {1711599},
MRREVIEWER = {George\ L.\ O'Brien},
       DOI = {10.1214/ECP.v4-1005},
       URL = {https://doi.org/10.1214/ECP.v4-1005},
}

@incollection {HS05,
    AUTHOR = {Hirsch, M. W. and Smith, Hal},
     TITLE = {Monotone dynamical systems},
 BOOKTITLE = {Handbook of differential equations: ordinary differential
              equations. {V}ol. {II}},
     PAGES = {239--357},
 PUBLISHER = {Elsevier B. V., Amsterdam},
      YEAR = {2005},
      ISBN = {0-444-52027-9},
   MRCLASS = {37B99 (34C12 34D09 34D20 34K20 37C65 37L99)},
  MRNUMBER = {2182759},
MRREVIEWER = {Sylvia\ Novo},
}

@incollection {N93,
    AUTHOR = {Nussbaum, Roger D.},
     TITLE = {The fixed point index and fixed point theorems},
 BOOKTITLE = {Topological methods for ordinary differential equations
              ({M}ontecatini {T}erme, 1991)},
    SERIES = {Lecture Notes in Math.},
    VOLUME = {1537},
     PAGES = {143--205},
 PUBLISHER = {Springer, Berlin},
      YEAR = {1993},
   MRCLASS = {58C30 (47H09 47H10 47H15 47N20)},
  MRNUMBER = {1226931},
MRREVIEWER = {Jacobo Pejsachowicz},
       DOI = {10.1007/BFb0085077},
       URL = {https://doi.org/10.1007/BFb0085077},
}

@article {N71,
    AUTHOR = {Nussbaum, Roger D.},
     TITLE = {The fixed point index for local condensing maps},
   JOURNAL = {Ann. Mat. Pura Appl. (4)},
  FJOURNAL = {Annali di Matematica Pura ed Applicata. Serie Quarta},
    VOLUME = {89},
      YEAR = {1971},
     PAGES = {217--258},
      ISSN = {0003-4622},
   MRCLASS = {47H10},
  MRNUMBER = {312341},
MRREVIEWER = {Heinrich Steinlein},
       DOI = {10.1007/BF02414948},
       URL = {https://doi.org/10.1007/BF02414948},
}

@article{HSW96,
title={Competitive exclusion and coexistence for competitive systems on ordered banach spaces},
author={S.-B. Hsu and H. Smith and P. Waltman},
journal={Trans. Amer. Math. Soc.},
year={1996},
volume={348},
number={},
pages={4083-4094}
}

@article{DH91,
title={Stability of fixed points for order-preserving discrete-time dynamical systems},
author={E. Dancer and P. Hess},
journal={J. Reine Angew. Math.},
year={1991},
volume={419},
number={},
pages={125-139}
}

@article{M84,
title={Existence of nontrivial unstable sets for  equilibriums of strongly order preserving systems},
author={Matano, Hiroshi},
journal={J. Fac. Sci. Univ. Tokyo},
year={1984},
volume={30},
number={},
pages={645-673}
}

@article{WFM95,
title={Heteroclinic orbits and convergence of orderpreserving set-condensing semiflows with applications to integrodifferential equations},
author={J. Wu and H. I. Freedman and R. K. Miller},
journal={J. Integral Equations Appl.},
year={1995},
volume={7},
number={},
pages={115-133}
}

@book{Z17,
  title      = {Dynamical Systems in Population Biology, second edition},
  author     = {X.-Q. Zhao},
  series     = {},
  volume     = {},
  publisher  = {Springer-Verlag},
  address    = {New York},
  year       = {2017},

}

@book{Ha88,
  title      = {Asymptotic Behavior of Dissipative Systems},
  author     = {J. K. Hale},
  series     = {Mathematical Surveys and Monographs},
  volume     = {25},
  publisher  = {American Mathematical Society},
  address    = {Providence, RI},
  year       = {1988},

}

@article {Zhang2023,
    AUTHOR = {Zhang, Shao-Qin},
     TITLE = {Existence and non-uniqueness of stationary distributions for
              distribution dependent {SDE}s},
   JOURNAL = {Electron. J. Probab.},
  FJOURNAL = {Electronic Journal of Probability},
    VOLUME = {28},
      YEAR = {2023},
     PAGES = {1-34},
   MRCLASS = {60H10 (60G10)},
  MRNUMBER = {4613856},
       DOI = {10.1214/23-ejp981},
       URL = {https://doi.org/10.1214/23-ejp981},
}

@book {Chue02,
    AUTHOR = {Chueshov, Igor},
     TITLE = {Monotone random systems theory and applications},
    SERIES = {Lecture Notes in Mathematics},
    VOLUME = {1779},
 PUBLISHER = {Springer-Verlag, Berlin},
      YEAR = {2002},
     PAGES = {viii+234},
      ISBN = {3-540-43246-9},
   MRCLASS = {37H10 (34F05 37C65 37G35 37L30 60H10 82C05)},
  MRNUMBER = {1902500},
MRREVIEWER = {Bj\"{o}rn\ Schmalfuss},
       DOI = {10.1007/b83277},
       URL = {https://doi.org/10.1007/b83277},
}

@article {Smith17,
    AUTHOR = {Smith, Hal L.},
     TITLE = {Monotone dynamical systems: reflections on new advances \&
              applications},
   JOURNAL = {Discrete Contin. Dyn. Syst.},
  FJOURNAL = {Discrete and Continuous Dynamical Systems. Series A},
    VOLUME = {37},
      YEAR = {2017},
    NUMBER = {1},
     PAGES = {485--504},
      ISSN = {1078-0947,1553-5231},
   MRCLASS = {34C12 (34-02 35K51)},
  MRNUMBER = {3583487},
MRREVIEWER = {Yves\ Dumont},
       DOI = {10.3934/dcds.2017020},
       URL = {https://doi.org/10.3934/dcds.2017020},
}

@book {Smith95,
    AUTHOR = {Smith, Hal L.},
     TITLE = {Monotone dynamical systems},
    SERIES = {Mathematical Surveys and Monographs},
    VOLUME = {41},
      NOTE = {An introduction to the theory of competitive and cooperative
              systems},
 PUBLISHER = {American Mathematical Society, Providence, RI},
      YEAR = {1995},
     PAGES = {x+174},
      ISBN = {0-8218-0393-X},
   MRCLASS = {34-02 (34C35 34Cxx 34Dxx 34Kxx 35K57 54H20 58Fxx)},
  MRNUMBER = {1319817},
MRREVIEWER = {Janusz\ Mierczy\'{n}ski},
}

@article {McKean1966,
    AUTHOR = {McKean, Jr., H. P.},
     TITLE = {A class of {M}arkov processes associated with nonlinear
              parabolic equations},
   JOURNAL = {Proc. Nat. Acad. Sci. U.S.A.},
  FJOURNAL = {Proceedings of the National Academy of Sciences of the United
              States of America},
    VOLUME = {56},
      YEAR = {1966},
     PAGES = {1907--1911},
      ISSN = {0027-8424},
   MRCLASS = {60.62},
  MRNUMBER = {221595},
MRREVIEWER = {F.\ B.\ Knight},
       DOI = {10.1073/pnas.56.6.1907},
       URL = {https://doi.org/10.1073/pnas.56.6.1907},
}

@article {Bao2022,
    AUTHOR = {Bao, Jianhai and Scheutzow, Michael and Yuan, Chenggui},
     TITLE = {Existence of invariant probability measures for functional
              {M}c{K}ean-{V}lasov {SDE}s},
   JOURNAL = {Electron. J. Probab.},
  FJOURNAL = {Electronic Journal of Probability},
    VOLUME = {27},
      YEAR = {2022},
     PAGES = {Paper No. 43, 14},
      ISSN = {1083-6489},
   MRCLASS = {60J60 (47D07 60H10)},
  MRNUMBER = {4404942},
       DOI = {10.1214/22-ejp773},
       URL = {https://doi.org/10.1214/22-ejp773},
}

@book {Hess91,
    AUTHOR = {Hess, Peter},
     TITLE = {Periodic-parabolic boundary value problems and positivity},
    SERIES = {Pitman Research Notes in Mathematics Series},
    VOLUME = {247},
 PUBLISHER = {Longman Scientific \& Technical, Harlow; copublished in the
              United States with John Wiley \& Sons, Inc., New York},
      YEAR = {1991},
     PAGES = {viii+139},
      ISBN = {0-582-06478-3},
   MRCLASS = {35-02 (35Bxx 35K57 47D06 47H20 47N20)},
  MRNUMBER = {1100011},
MRREVIEWER = {Paul G. Schmidt},
}

@article {H79,
    AUTHOR = {Matano, Hiroshi},
     TITLE = {Asymptotic behavior and stability of solutions of semilinear
              diffusion equations},
   JOURNAL = {Publ. Res. Inst. Math. Sci.},
  FJOURNAL = {Kyoto University. Research Institute for Mathematical
              Sciences. Publications},
    VOLUME = {15},
      YEAR = {1979},
    NUMBER = {2},
     PAGES = {401--454},
      ISSN = {0034-5318},
   MRCLASS = {35K60 (35B40 92A09)},
  MRNUMBER = {555661},
MRREVIEWER = {Jagdish Chandra},
       DOI = {10.2977/prims/1195188180},
       URL = {https://doi.org/10.2977/prims/1195188180},
}

@incollection {H86,
    AUTHOR = {Matano, Hiroshi},
     TITLE = {Strongly order-preserving local semidynamical systems---theory
              and applications},
 BOOKTITLE = {Semigroups, theory and applications, {V}ol. {I} ({T}rieste,
              1984)},
    SERIES = {Pitman Res. Notes Math. Ser.},
    VOLUME = {141},
     PAGES = {178--185},
 PUBLISHER = {Longman Sci. Tech., Harlow},
      YEAR = {1986},
   MRCLASS = {54H20},
  MRNUMBER = {876941},
MRREVIEWER = {L. Janos},
}

@article {Hirsch84,
    AUTHOR = {Hirsch, Morris W.},
     TITLE = {The dynamical systems approach to differential equations},
   JOURNAL = {Bull. Amer. Math. Soc. (N.S.)},
  FJOURNAL = {American Mathematical Society. Bulletin. New Series},
    VOLUME = {11},
      YEAR = {1984},
    NUMBER = {1},
     PAGES = {1--64},
      ISSN = {0273-0979},
   MRCLASS = {58Fxx (00A69 01A60 34C35 35B99)},
  MRNUMBER = {741723},
MRREVIEWER = {D. K. Arrowsmith},
       DOI = {10.1090/S0273-0979-1984-15236-4},
       URL = {https://doi.org/10.1090/S0273-0979-1984-15236-4},
}

@article {Hirsch88,
    AUTHOR = {Hirsch, Morris W.},
     TITLE = {Stability and convergence in strongly monotone dynamical
              systems},
   JOURNAL = {J. Reine Angew. Math.},
  FJOURNAL = {Journal f\"{u}r die Reine und Angewandte Mathematik. [Crelle's
              Journal]},
    VOLUME = {383},
      YEAR = {1988},
     PAGES = {1--53},
      ISSN = {0075-4102},
   MRCLASS = {58F25 (47H20)},
  MRNUMBER = {921986},
MRREVIEWER = {W. R. Utz},
       DOI = {10.1515/crll.1988.383.1},
       URL = {https://doi.org/10.1515/crll.1988.383.1},
}

@book {Carmona-Delarue2018i,
    AUTHOR = {Carmona, Ren\'{e} and Delarue, Fran\c{c}ois},
     TITLE = {Probabilistic theory of mean field games with applications.
              {I}},
    SERIES = {Probability Theory and Stochastic Modelling},
    VOLUME = {83},
      NOTE = {Mean field FBSDEs, control, and games},
 PUBLISHER = {Springer, Cham},
      YEAR = {2018},
     PAGES = {xxv+713},
      ISBN = {978-3-319-56437-1; 978-3-319-58920-6},
   MRCLASS = {60-02 (35R60 49N70 49N90 60H15 60H30 91A15 93E20)},
  MRNUMBER = {3752669},
MRREVIEWER = {Vassili N. Kolokol\cprime tsov},
}

@book {Carmona-Delarue2018ii,
    AUTHOR = {Carmona, Ren\'{e} and Delarue, Fran\c{c}ois},
     TITLE = {Probabilistic theory of mean field games with applications.
              {II}},
    SERIES = {Probability Theory and Stochastic Modelling},
    VOLUME = {84},
      NOTE = {Mean field games with common noise and master equations},
 PUBLISHER = {Springer, Cham},
      YEAR = {2018},
     PAGES = {xxiv+697},
      ISBN = {978-3-319-56435-7; 978-3-319-56436-4},
   MRCLASS = {60-02 (35R60 49L20 60G55 60H10 60H30 91A13 91A15)},
  MRNUMBER = {3753660},
MRREVIEWER = {Vassili N. Kolokol\cprime tsov},
}

@article {Bouchitté-Champion-Jimenez2005,
    AUTHOR = {Bouchitt\'{e}, G. and Champion, T. and Jimenez, C.},
     TITLE = {Completion of the space of measures in the {K}antorovich norm},
   JOURNAL = {Riv. Mat. Univ. Parma (7)},
  FJOURNAL = {Rivista di Matematica della Universit\`a di Parma. Serie 7},
    VOLUME = {4*},
      YEAR = {2005},
     PAGES = {127--139},
      ISSN = {0035-6298},
   MRCLASS = {49J45 (28A33 46E27)},
  MRNUMBER = {2197484},
MRREVIEWER = {Wilfrid\ Gangbo},
}

@article {Gess17,
    AUTHOR = {Flandoli, Franco and Gess, Benjamin and Scheutzow, Michael},
     TITLE = {Synchronization by noise for order-preserving random dynamical
              systems},
   JOURNAL = {Ann. Probab.},
  FJOURNAL = {The Annals of Probability},
    VOLUME = {45},
      YEAR = {2017},
    NUMBER = {2},
     PAGES = {1325--1350},
      ISSN = {0091-1798,2168-894X},
   MRCLASS = {37B25 (37G35 37H15 60H10)},
  MRNUMBER = {3630300},
MRREVIEWER = {Jos\'{e}\ A.\ Langa},
       DOI = {10.1214/16-AOP1088},
       URL = {https://doi.org/10.1214/16-AOP1088},
}

@book {Chow82,
    AUTHOR = {Chow, Shui Nee and Hale, Jack K.},
     TITLE = {Methods of bifurcation theory},
    SERIES = {Grundlehren der Mathematischen Wissenschaften},
    VOLUME = {251},
 PUBLISHER = {Springer-Verlag, New York-Berlin},
      YEAR = {1982},
     PAGES = {xv+515},
      ISBN = {0-387-90664-9},
   MRCLASS = {58E07 (34-02 58-02 58F14)},
  MRNUMBER = {660633},
MRREVIEWER = {Norman\ Dancer},
}

@article {MayLeonard75,
    AUTHOR = {May, Robert M. and Leonard, Warren J.},
     TITLE = {Nonlinear aspects of competition between three species},
   JOURNAL = {SIAM J. Appl. Math.},
  FJOURNAL = {SIAM Journal on Applied Mathematics},
    VOLUME = {29},
      YEAR = {1975},
    NUMBER = {2},
     PAGES = {243--253},
      ISSN = {0036-1399},
   MRCLASS = {92A15},
  MRNUMBER = {392035},
MRREVIEWER = {L.\ Billard},
       DOI = {10.1137/0129022},
       URL = {https://doi.org/10.1137/0129022},
}

@incollection {Balmforth95,
    AUTHOR = {Balmforth, N. J.},
     TITLE = {Solitary waves and homoclinic orbits},
 BOOKTITLE = {Annual review of fluid mechanics, {V}ol. 27},
     PAGES = {335--373},
 PUBLISHER = {Annual Reviews, Palo Alto, CA},
      YEAR = {1995},
      ISBN = {0-8243-0727-5},
   MRCLASS = {35Q35 (35B32 58F14 58F39 76B25 76E30)},
  MRNUMBER = {1312620},
MRREVIEWER = {J.\ S.\ Joel},
}

@article {Carrillo-McCann-Villani2003,
    AUTHOR = {Carrillo, Jos\'{e} A. and McCann, Robert J. and Villani,
              C\'{e}dric},
     TITLE = {Kinetic equilibration rates for granular media and related
              equations: entropy dissipation and mass transportation
              estimates},
   JOURNAL = {Rev. Mat. Iberoamericana},
  FJOURNAL = {Revista Matem\'{a}tica Iberoamericana},
    VOLUME = {19},
      YEAR = {2003},
    NUMBER = {3},
     PAGES = {971--1018},
      ISSN = {0213-2230},
   MRCLASS = {35K55 (35B40 35K65 76T25)},
  MRNUMBER = {2053570},
MRREVIEWER = {Thomas\ P.\ Witelski},
       DOI = {10.4171/RMI/376},
       URL = {https://doi.org/10.4171/RMI/376},
}

@article {Carrillo-McCann-Villani2006,
    AUTHOR = {Carrillo, Jos\'{e} A. and McCann, Robert J. and Villani,
              C\'{e}dric},
     TITLE = {Contractions in the 2-{W}asserstein length space and
              thermalization of granular media},
   JOURNAL = {Arch. Ration. Mech. Anal.},
  FJOURNAL = {Archive for Rational Mechanics and Analysis},
    VOLUME = {179},
      YEAR = {2006},
    NUMBER = {2},
     PAGES = {217--263},
      ISSN = {0003-9527,1432-0673},
   MRCLASS = {76M30 (35K57 58E50 76T25 82C40)},
  MRNUMBER = {2209130},
MRREVIEWER = {Beno\^{i}t\ P.\ Desjardins},
       DOI = {10.1007/s00205-005-0386-1},
       URL = {https://doi.org/10.1007/s00205-005-0386-1},
}

@article {Tamura1984,
    AUTHOR = {Tamura, Yozo},
     TITLE = {On asymptotic behaviors of the solution of a nonlinear
              diffusion equation},
   JOURNAL = {J. Fac. Sci. Univ. Tokyo Sect. IA Math.},
  FJOURNAL = {Journal of the Faculty of Science. University of Tokyo.
              Section IA. Mathematics},
    VOLUME = {31},
      YEAR = {1984},
    NUMBER = {1},
     PAGES = {195--221},
      ISSN = {0040-8980},
   MRCLASS = {60K35 (58G32 60J60)},
  MRNUMBER = {743525},
MRREVIEWER = {S.\ R. S. Varadhan},
}

@article {Tamura1987,
    AUTHOR = {Tamura, Yozo},
     TITLE = {Free energy and the convergence of distributions of diffusion
              processes of {M}c{K}ean type},
   JOURNAL = {J. Fac. Sci. Univ. Tokyo Sect. IA Math.},
  FJOURNAL = {Journal of the Faculty of Science. University of Tokyo.
              Section IA. Mathematics},
    VOLUME = {34},
      YEAR = {1987},
    NUMBER = {2},
     PAGES = {443--484},
      ISSN = {0040-8980},
   MRCLASS = {60J60 (35R60 60H10)},
  MRNUMBER = {914029},
MRREVIEWER = {Ralf\ Manthey},
}

@article {Dawson1983,
    AUTHOR = {Dawson, Donald A.},
     TITLE = {Critical dynamics and fluctuations for a mean-field model of
              cooperative behavior},
   JOURNAL = {J. Statist. Phys.},
  FJOURNAL = {Journal of Statistical Physics},
    VOLUME = {31},
      YEAR = {1983},
    NUMBER = {1},
     PAGES = {29--85},
      ISSN = {0022-4715,1572-9613},
   MRCLASS = {82A25 (60K35)},
  MRNUMBER = {711469},
MRREVIEWER = {Ronald\ F.\ Fox},
       DOI = {10.1007/BF01010922},
       URL = {https://doi.org/10.1007/BF01010922},
}

@article{Cormier2024,
author = {Quentin Cormier},
title = {{On the stability of the invariant probability measures of McKean–Vlasov equations}},
volume = {61},
journal = {Annales de l'Institut Henri Poincaré, Probabilités et Statistiques},
number = {4},
publisher = {Institut Henri Poincaré},
pages = {2405 -- 2429},
year = {2025},
doi = {10.1214/24-AIHP1504},
URL = {https://doi.org/10.1214/24-AIHP1504}
}

@article{Monmarche-Reygner2024,
	author = {Monmarch{\'e}, Pierre and Reygner, Julien},
	date = {2025/07/31},
	doi = {10.1007/s00440-025-01399-0},
	id = {Monmarch{\'e}2025},
	isbn = {1432-2064},
	journal = {Probability Theory and Related Fields},
	title = {Local convergence rates for Wasserstein gradient flows and McKean-Vlasov equations with multiple stationary solutions},
	url = {https://doi.org/10.1007/s00440-025-01399-0},
	year = {2025}
}

@article {Huang-Liu-Wang2018,
    AUTHOR = {Huang, Xing and Liu, Chang and Wang, Feng-Yu},
     TITLE = {Order preservation for path-distribution dependent {SDE}s},
   JOURNAL = {Commun. Pure Appl. Anal.},
  FJOURNAL = {Communications on Pure and Applied Analysis},
    VOLUME = {17},
      YEAR = {2018},
    NUMBER = {5},
     PAGES = {2125--2133},
      ISSN = {1534-0392},
   MRCLASS = {60H10},
  MRNUMBER = {3809144},
MRREVIEWER = {Jun Shen},
       DOI = {10.3934/cpaa.2018100},
       URL = {https://doi.org/10.3934/cpaa.2018100},
}

@article {Meyn-Tweedie1993,
    AUTHOR = {Meyn, Sean P. and Tweedie, R. L.},
     TITLE = {Stability of {M}arkovian processes. {III}. {F}oster-{L}yapunov
              criteria for continuous-time processes},
   JOURNAL = {Adv. in Appl. Probab.},
  FJOURNAL = {Advances in Applied Probability},
    VOLUME = {25},
      YEAR = {1993},
    NUMBER = {3},
     PAGES = {518--548},
      ISSN = {0001-8678},
   MRCLASS = {60J27},
  MRNUMBER = {1234295},
MRREVIEWER = {Esa Nummelin},
       DOI = {10.2307/1427522},
       URL = {https://doi.org/10.2307/1427522},
}

@book {Meyn-Tweedie1993-book,
    AUTHOR = {Meyn, S. P. and Tweedie, R. L.},
     TITLE = {Markov chains and stochastic stability},
    SERIES = {Communications and Control Engineering Series},
 PUBLISHER = {Springer-Verlag London, Ltd., London},
      YEAR = {1993},
     PAGES = {xvi+ 548},
      ISBN = {3-540-19832-6},
   MRCLASS = {60J05},
  MRNUMBER = {1287609},
MRREVIEWER = {Esa Nummelin},
       DOI = {10.1007/978-1-4471-3267-7},
       URL = {https://doi.org/10.1007/978-1-4471-3267-7},
}

@book {Lindvall1992,
    AUTHOR = {Lindvall, Torgny},
     TITLE = {Lectures on the coupling method},
    SERIES = {Wiley Series in Probability and Mathematical Statistics:
              Probability and Mathematical Statistics},
      NOTE = {A Wiley-Interscience Publication},
 PUBLISHER = {John Wiley \& Sons, Inc., New York},
      YEAR = {1992},
     PAGES = {xiv+257},
      ISBN = {0-471-54025-0},
   MRCLASS = {60-01 (60J10 60K05 60K35)},
  MRNUMBER = {1180522},
MRREVIEWER = {Uwe R\"{o}sler},
}

@article{Tugaut2023,
author = {Julian Tugaut},
title = {On the steady states for the granular media equation: existence, local uniqueness, local stability and rate of convergence},
journal = {Stochastics},
volume = {0},
number = {0},
pages = {1--18},
year = {2025},
publisher = {Taylor \& Francis},
}

@article {ShenYi98,
    AUTHOR = {Shen, Wenxian and Yi, Yingfei},
     TITLE = {Almost automorphic and almost periodic dynamics in
              skew-product semiflows},
   JOURNAL = {Mem. Amer. Math. Soc.},
  FJOURNAL = {Memoirs of the American Mathematical Society},
    VOLUME = {136},
      YEAR = {1998},
    NUMBER = {647},
     PAGES = {x+93},
      ISSN = {0065-9266,1947-6221},
   MRCLASS = {34D05 (34C27 34C35 35K55 47H20 54H20 58F11)},
  MRNUMBER = {1445493},
MRREVIEWER = {Russell\ A.\ Johnson},
       DOI = {10.1090/memo/0647},
       URL = {https://doi.org/10.1090/memo/0647},
}

@article {ChueshovScheutzow2004,
    AUTHOR = {Chueshov, Igor and Scheutzow, Michael},
     TITLE = {On the structure of attractors and invariant measures for a
              class of monotone random systems},
   JOURNAL = {Dyn. Syst.},
  FJOURNAL = {Dynamical Systems. An International Journal},
    VOLUME = {19},
      YEAR = {2004},
    NUMBER = {2},
     PAGES = {127--144},
      ISSN = {1468-9367,1468-9375},
   MRCLASS = {37H10 (34D45 34F05 37C65 37C70 82C05)},
  MRNUMBER = {2060422},
MRREVIEWER = {Bj\"{o}rn\ Schmalfuss},
       DOI = {10.1080/1468936042000207792},
       URL = {https://doi.org/10.1080/1468936042000207792},
}

@article {Butkovsky2020,
    AUTHOR = {Butkovsky, Oleg and Scheutzow, Michael},
     TITLE = {Couplings via comparison principle and exponential ergodicity
              of {SPDE}s in the hypoelliptic setting},
   JOURNAL = {Comm. Math. Phys.},
  FJOURNAL = {Communications in Mathematical Physics},
    VOLUME = {379},
      YEAR = {2020},
    NUMBER = {3},
     PAGES = {1001--1034},
      ISSN = {0010-3616,1432-0916},
   MRCLASS = {60J35 (47D07 60H15)},
  MRNUMBER = {4163359},
MRREVIEWER = {Sonja\ Cox},
       DOI = {10.1007/s00220-020-03834-w},
       URL = {https://doi.org/10.1007/s00220-020-03834-w},
}

@article {RobertTweedie2000,
    AUTHOR = {Roberts, G. O. and Tweedie, R. L.},
     TITLE = {Rates of convergence of stochastically monotone and continuous
              time {M}arkov models},
   JOURNAL = {J. Appl. Probab.},
  FJOURNAL = {Journal of Applied Probability},
    VOLUME = {37},
      YEAR = {2000},
    NUMBER = {2},
     PAGES = {359--373},
      ISSN = {0021-9002,1475-6072},
   MRCLASS = {60J05 (60F05 60J25 60K30)},
  MRNUMBER = {1780996},
MRREVIEWER = {Jean\ Diebolt},
       DOI = {10.1239/jap/1014842542},
       URL = {https://doi.org/10.1239/jap/1014842542},
}

@article {RobertRichard1996,
    AUTHOR = {Lund, Robert B. and Tweedie, Richard L.},
     TITLE = {Geometric convergence rates for stochastically ordered
              {M}arkov chains},
   JOURNAL = {Math. Oper. Res.},
  FJOURNAL = {Mathematics of Operations Research},
    VOLUME = {21},
      YEAR = {1996},
    NUMBER = {1},
     PAGES = {182--194},
      ISSN = {0364-765X,1526-5471},
   MRCLASS = {60J05 (60K25)},
  MRNUMBER = {1385873},
MRREVIEWER = {Hermann\ Thorisson},
       DOI = {10.1287/moor.21.1.182},
       URL = {https://doi.org/10.1287/moor.21.1.182},
}

@article {Kamae1977,
    AUTHOR = {Kamae, T. and Krengel, U. and O'Brien, G. L.},
     TITLE = {Stochastic inequalities on partially ordered spaces},
   JOURNAL = {Ann. Probab.},
  FJOURNAL = {The Annals of Probability},
    VOLUME = {5},
      YEAR = {1977},
    NUMBER = {6},
     PAGES = {899--912},
      ISSN = {0091-1798},
   MRCLASS = {60G10},
  MRNUMBER = {494447},
MRREVIEWER = {R.\ M.\ Dudley},
       DOI = {10.1214/aop/1176995659},
       URL = {https://doi.org/10.1214/aop/1176995659},
}

@article{Zhang2025,
  title={Local convergence near equilibria for distribution dependent SDEs},
  author={Zhang, Shao-Qin},
  journal={arXiv preprint arXiv:2501.04313},
  year={2025}
}

@incollection {Aoki1980,
    AUTHOR = {Aoki, Masanao},
     TITLE = {Dynamics and control of a system composed of a large number of
              similar subsystems},
 BOOKTITLE = {Dynamic optimization and mathematical economics},
    SERIES = {Math. Concepts Methods Sci. Engrg.},
    VOLUME = {19},
     PAGES = {183--203},
 PUBLISHER = {Plenum, New York},
      YEAR = {1980},
   MRCLASS = {93A15},
  MRNUMBER = {567121},
}

@book {Haken77,
    AUTHOR = {Haken, Hermann},
     TITLE = {Synergetics---an introduction},
      NOTE = {Nonequilibrium phase transitions and self-organization in
              physics, chemistry and biology},
 PUBLISHER = {Springer-Verlag, Berlin-New York},
      YEAR = {1977},
     PAGES = {xii+325},
      ISBN = {3-540-07885-1},
   MRCLASS = {82.60 (92A05)},
  MRNUMBER = {471840},
MRREVIEWER = {V.\ K.\ Wong},
}

@article {Horsthemke77,
    AUTHOR = {Horsthemke, W. and Malek-Mansour, M. and Hayez, B.},
     TITLE = {An asymptotic expansion of the nonlinear master equation},
   JOURNAL = {J. Statist. Phys.},
  FJOURNAL = {Journal of Statistical Physics},
    VOLUME = {16},
      YEAR = {1977},
    NUMBER = {2},
     PAGES = {201--215},
      ISSN = {0022-4715,1572-9613},
   MRCLASS = {82.35},
  MRNUMBER = {449370},
MRREVIEWER = {Mayer\ Humi},
       DOI = {10.1007/BF01418752},
       URL = {https://doi.org/10.1007/BF01418752},
}

@article {Kometani75,
    AUTHOR = {Kometani, K. and Shimizu, H.},
     TITLE = {A study of self-organizing processes of nonlinear stochastic
              variables},
   JOURNAL = {J. Statist. Phys.},
  FJOURNAL = {Journal of Statistical Physics},
    VOLUME = {13},
      YEAR = {1975},
    NUMBER = {6},
     PAGES = {473--490},
      ISSN = {0022-4715,1572-9613},
   MRCLASS = {92A05},
  MRNUMBER = {449768},
MRREVIEWER = {K.\ Dietz},
       DOI = {10.1007/BF01013146},
       URL = {https://doi.org/10.1007/BF01013146},
}

@article {DuongHongTugaut20,
    AUTHOR = {Duong, Manh Hong and Tugaut, Julian},
     TITLE = {Coupled {M}c{K}ean-{V}lasov diffusions: wellposedness,
              propagation of chaos and invariant measures},
   JOURNAL = {Stochastics},
  FJOURNAL = {Stochastics. An International Journal of Probability and
              Stochastic Processes},
    VOLUME = {92},
      YEAR = {2020},
    NUMBER = {6},
     PAGES = {900--943},
      ISSN = {1744-2508,1744-2516},
   MRCLASS = {60H10 (35B30 35K55 35R60 60G10 60J60)},
  MRNUMBER = {4139089},
       DOI = {10.1080/17442508.2019.1677663},
       URL = {https://doi.org/10.1080/17442508.2019.1677663},
}

\end{document}